\documentclass[11pt,twoside,letterpaper]{article}
\usepackage{mathtools} \mathtoolsset{showonlyrefs=true}
\usepackage{cases}

\usepackage{amsthm, amsmath, amssymb, amsfonts, graphicx, epsfig}
\usepackage{accents}

\usepackage{bbm}
\usepackage{mathrsfs}

\usepackage{ulem}

\usepackage{cancel}

\usepackage{epstopdf}

\usepackage{graphicx}
\usepackage[font=sl,labelfont=bf]{caption}
\usepackage{subcaption}

\usepackage{float}
\restylefloat{table}

\usepackage{enumerate}

\usepackage[usenames,dvipsnames]{color}

\definecolor{labelkey}{rgb}{0,0,1}

\usepackage[colorlinks=true, pdfstartview=FitV, linkcolor=blue,
            citecolor=blue, urlcolor=blue]{hyperref}
\usepackage[usenames]{color}

\usepackage{url}
\makeatletter\def\url@leostyle{%
 \@ifundefined{selectfont}{\def\UrlFont{\sf}}{\def\UrlFont{\scriptsize\ttfamily}}} \makeatother

\usepackage[framemethod=default]{mdframed}

\usepackage{array}

\usepackage[margin=2.5cm, letterpaper]{geometry}

\newtheorem{theorem}{Theorem}
\newtheorem{proposition}[theorem]{Proposition}
\newtheorem{lemma}[theorem]{Lemma}
\newtheorem{corollary}[theorem]{Corollary}
\newtheorem{assumption}[theorem]{Assumption}
\theoremstyle{definition}

\theoremstyle{remark}
\newtheorem{remark}[theorem]{Remark}

\numberwithin{equation}{section}
\numberwithin{theorem}{section}

\definecolor{Red}{rgb}{1.0,0,0.0}
\definecolor{Blue}{rgb}{0,0.0,1.0}
\definecolor{Green}{rgb}{0.2,0.5,0.2}

\def\cB{\mathcal{B}}

\def\cE{\mathcal{E}}
\def\cF{\mathcal{F}}

\def\cH{\mathcal{H}}

\def\cM{\mathcal{M}}

\def\cP{\mathcal{P}}
\def\cQ{\mathcal{Q}}

\def\cT{\mathcal{T}}

\def\bE{\mathbb{E}}
\def\bF{\mathbb{F}}

\def\bN{\mathbb{N}}

\def\bP{\mathbb{P}}

\def\bR{\mathbb{R}}

\def\bZ{\mathbb{Z}}

\def\sD{\mathscr{D}}

\def\sF{\mathscr{F}}
\def\sG{\mathscr{G}}

\def\sZ{\mathscr{Z}}

\def\mQ{\mathsf{Q}}

\def\mV{\mathsf{V}}

\newcommand{\wt}{\widetilde}
\newcommand{\wh}{\widehat}

\newcommand{\1}{\mathbbm{1}}            
\newcommand{\set}[1]{\{#1\}}            

\DeclareMathOperator{\dif}{d \!}        

\DeclareMathOperator{\supp}{supp}

\title{Wiener-Hopf Factorization for Arithmetic Brownian Motion with Time-Dependent Drift and Volatility}

\author{
    Tomasz R. Bielecki 
	\thanks{Department of Applied Mathematics, Illinois Institute of Technology, USA. \textbf{Email:} tbielecki@iit.edu}
 \and
    Ziteng Cheng 
	\thanks{Department of Statistical Sciences, University of Toronto, Canada. \textbf{Email:} ziteng.cheng@utoronto.ca}
 \and
     Ruoting Gong 
	\thanks{Department of Applied Mathematics, Illinois Institute of Technology, USA. \textbf{Email:} rgong2@iit.edu}
        }

\date{
First Circulated: June 3, 2020\\
This Version: August 01, 2022
}

\begin{document}

\maketitle




\smallskip

{\footnotesize
\begin{tabular}{l@{} p{350pt}}
  \hline \\[-.2em]
  \textsc{Abstract}: \ & In this paper we obtain a Wiener-Hopf type factorization for a real-valued arithmetic Brownian motion with time-dependent drift and volatility. To the best of our knowledge, this paper is the very first step towards realizing the objective of deriving Wiener-Hopf type factorizations for (real-valued) time-inhomogeneous L\'{e}vy processes. In order to prove our main theorem, we derive some new results regarding time-inhomogeneous noisy Wiener-Hopf factorization. We demonstrate that in the special case of the arithmetic Brownian motion with constant drift and volatility our main result agrees with classical Wiener-Hopf factorization for this particular time-homogenous L\'{e}vy process.   \\[0.5em]
\textsc{Keywords:} \ & Wiener-Hopf factorization, time-inhomogeneous L\'{e}vy process, Markov family, Feller semigroup, time-homogenization \\[0.5em]
\textsc{MSC2010:} \ & 60G51, 60J25, 60J65 \\[1em]
  \hline
\end{tabular}
}


\section{Introduction}\label{sec:Intro}

In this paper we obtain a Wiener-Hopf type factorization for a time-inhomogeneous arithmetic Brownian motion with deterministic time-dependent drift and volatility. To the best of our knowledge, this paper is the very first step towards realizing the objective of deriving Wiener-Hopf type factorizations for real-valued {\it time-inhomogeneous L\'{e}vy processes}.\footnote{By time-inhomogeneous L\'{e}vy process, we mean a continuous-time stochastic process that has the independent increments property, but not the stationary increments property. This type of processes are also known as {\it additive processes} (cf. \cite[Definition 1.6]{Sato:1999}).}

In order to motivate the main results presented in this paper, we first provide a brief account of three forms of Wiener-Hopf factorizations for time-homogeneous real-valued L\'{e}vy processes based on \cite[Section 11.2.1]{BoyarchenkoLevendorski:2007}, \cite[Section I.29]{RogersWilliams:2000}, and \cite[Section 45]{Sato:1999}, and a numerical argument indicating that such forms of Wiener-Hopf factorizations are no longer valid for time-inhomogeneous L\'{e}vy processes.

So, let $X:=(X_{t})_{t\in\bR_{+}}$ be a time-homogeneous real-valued L\'{e}vy processes defined on some probability space $(\Omega,\sF,\bP)$ with $X_{0}=0$ $\bP$-a.s., where $\bR_{+}:=[0,\infty)$. We denote by $\psi(\xi)$, $\xi\in\bR$, the characteristic exponent of $X$, so that $\bE(e^{i\xi X_{t}})=e^{t\psi(\xi)}$, for any $t\in\bR_{+}$. For any fixed $c\in(0,\infty)$, we consider an exponentially distributed random variable $\mathbf{e}_{c}$ on $(\Omega,\sF,\bP)$ with $\bE(\mathbf{e}_{c})=c^{-1}$, and we assume that $\mathbf{e}_{c}$ and $X$ are independent under $\bP$. Denote by
\begin{align*}
\underline{X}_{t}:=\inf_{s\in[0,t]}X_{s},\quad\overline{X}_{t}:=\sup_{s\in[0,t]}X_{s},\quad t\in\bR_{+},
\end{align*}
the running maximum and running minimum processes of $X$, respectively. It is well known (cf. \cite[Chapter I, (29.4) \& (29.5)]{RogersWilliams:2000}) that
\begin{align}\label{eq:LevyPathDecompIndep}
\text{$\overline{X}_{\mathsf{e}_{c}}$ and $X_{\mathbf{e}_{c}}-\overline{X}_{\mathsf{e}_{c}}$ are independent},
\end{align}
and that
\begin{align}\label{eq:LevyPathDecompIdtDist}
\text{$\underline{X}_{\mathsf{e}_{c}}$ and $X_{\mathsf{e}_{c}}-\overline{X}_{\mathsf{e}_{c}}$ have the same distribution}.
\end{align}
The above two properties imply that
\begin{align}\label{eq:WHLevyChf}
\bE\Big(e^{i\xi X_{\mathsf{e}_{c}}}\Big)=\bE\Big(e^{i\xi\overline{X}_{\mathsf{e}_{c}}}\Big)\bE\Big(e^{i\xi\underline{X}_{\mathsf{e}_{c}}}\Big),\quad\xi\in\bR.
\end{align}
Formula \eqref{eq:WHLevyChf} is known as the {\it Wiener-Hopf factorization (WHf) for the real-valued L\'{e}vy process $X$} (cf. \cite[(11.9)]{BoyarchenkoLevendorski:2007}, \cite[Chapter I. (29.2) (iii)]{RogersWilliams:2000}). It is a particular version of the so-called Pecherskii-Rogozin-Spitzer identity (see e.g. \cite{Bertoin:1996}).

Next, we denote by $\phi_{c}^{+}(\xi)$ (respectively, $\phi_{c}^{-}(\xi)$), $\xi\in\bR$, the characteristic function of $\overline{X}_{\mathsf{e}_{c}}$ (respectively, $\underline{X}_{\mathsf{e}_{c}}$). Noting that
\begin{align*}
\bE\Big(e^{i\xi X_{\mathsf{e}_{c}}}\Big)=c\,\bE\bigg(\int_{0}^{\infty}e^{-ct}e^{i\xi X_{t}}\,dt\bigg)=c\int_{0}^{\infty}e^{-ct}e^{\psi(\xi)t}\,dt=\frac{c}{c-\psi(\xi)},
\end{align*}
we obtain the following equivalent form of \eqref{eq:WHLevyChf} (e.g. \cite[(11.12)]{BoyarchenkoLevendorski:2007} and \cite[(45.1)]{Sato:1999})\footnote{There are other equivalent expressions for $\phi_{c}^{\pm}$, e.g. \cite[(45.2) \& (45.3)]{Sato:1999}. Moreover, it is well-known that $(\phi_{c}^{+},\phi_{c}^{-})$ is the unique pair of characteristic functions of infinitely divisible distributions having drift $0$ supported on $[0,\infty)$ and $(-\infty,0]$, respectively, such that \eqref{eq:WHLevySym} holds (cf. \cite[Theorem 45.2]{Sato:1999}). Those results are irrelevant to our later discussions, and are therefore omitted here.}
\begin{align}\label{eq:WHLevySym}
\frac{c}{c-\psi(\xi)}=\phi_{c}^{+}(\xi)\phi_{c}^{-}(\xi),\quad\xi\in\bR.
\end{align}

Now we define the following operators on $L^{\infty}(\bR)$
\begin{align*}
\big(\cH_{c}u\big)(x)&:=c\,\bE\bigg(\int_{0}^{\infty}e^{-ct}u\big(x+X_{t}\big)\,dt\bigg)=\bE\big(u\big(x+X_{\mathsf{e}_{c}}\big)\big),\\
\big(\cH_{c}^{+}u\big)(x)&:=c\,\bE\bigg(\int_{0}^{\infty}e^{-ct}u\big(x+\overline{X}_{t}\big)\,dt\bigg)=\bE\big(u(x+\overline{X}_{\mathsf{e}_{c}}\big)\big),\\
\big(\cH_{c}^{-}u\big)(x)&:=c\,\bE\bigg(\int_{0}^{\infty}e^{-ct}u\big(x+\underline{X}_{t}\big)\,dt\bigg)=\bE\big(u\big(x+\underline{X}_{\mathsf{e}_{c}}\big)\big).
\end{align*}
It can be shown that \eqref{eq:LevyPathDecompIndep} and \eqref{eq:LevyPathDecompIdtDist} also imply
\begin{align}\label{eq:WHLevyOpr}
\cH_{c}u=\cH_{c}^{+}\cH_{c}^{-}u=\cH_{c}^{-}\cH_{c}^{+}u,\quad u\in L^{\infty}(\bR),
\end{align}
(cf. \cite[(11.16)]{BoyarchenkoLevendorski:2007}), which in turn implies \eqref{eq:WHLevyChf} (equivalently, \eqref{eq:WHLevySym}). Thus we call \eqref{eq:WHLevyOpr} {\it the operator form} of the WHf for the time-homogeneous real-valued L\'{e}vy process $X$.

To see that \eqref{eq:WHLevyOpr} indeed implies \eqref{eq:WHLevyChf} (equivalently, \eqref{eq:WHLevySym}), by \eqref{eq:WHLevyOpr} we first have
\begin{align*}
\big(\cH_{c}\sin(\xi\,\cdot)\big)(a)=\big(\cH_{c}^{+}\big(\cH_{c}^{-}\sin(\xi\,\cdot)\big)\big)(a),\quad\big(\cH_{c}\cos(\xi\,\cdot)\big)(a)=\big(\cH_{c}^{+}\big(\cH_{c}^{-}\cos(\xi\,\cdot)\big)\big)(a),\quad a,\xi\in\bR.
\end{align*}
Thus, by the Euler formula, as well as the linearity of $\cH_{c}$ and $\cH_{c}^{\pm}$, we see that \eqref{eq:WHLevyOpr} holds true for $u(x)=e^{i\xi x}$, $x\in\bR$. It follows immediately that for all $\xi\in\bR$,
\begin{align*}
\bE\Big(e^{i\xi X_{\mathsf{e}_{c}}}\Big)&=\big(\cH_{c}\,e^{i\xi\,\cdot}\big)(0)=\big(\cH_{c}^{+}\cH_{c}^{-}\,e^{i\xi\,\cdot}\big)(0)=\left(\cH_{c}^{+}\bigg(c\,\bE\bigg(\int_{0}^{\infty}e^{-ct}e^{i\xi(\cdot\,+\underline{X}_{t})}\,dt\bigg)\bigg)\right)(0)\\
&=\big(\cH_{c}^{+}e^{i\xi\,\cdot}\big)(0)\cdot c\,\bE\bigg(\int_{0}^{\infty}e^{-ct}e^{i\xi\underline{X}_{t}}\,dt\bigg)\\
&=c\,\bE\bigg(\int_{0}^{\infty}e^{-ct}e^{i\xi\overline{X}_t}\,dt\bigg)\cdot c\,\bE\bigg(\int_{0}^{\infty}e^{-ct}e^{i\xi\underline{X}_{t}}\,dt\bigg)=\bE\Big(e^{i\xi\overline{X}_{\mathsf{e}_{c}}}\Big)\bE\Big(e^{i\xi\underline{X}_{\mathsf{e}_{c}}}\Big).
\end{align*}

We conjecture that when $X$ is a time-inhomogeneous real-valued L\'{e}vy process then properties \eqref{eq:LevyPathDecompIndep} and \eqref{eq:LevyPathDecompIdtDist} as well as the identity \eqref{eq:WHLevyChf} do not hold any more due to the lack of stationarity of increments. Such conjecture is strongly supported by our numerical simulations, though at this time we do not have a formal proof for it. More precisely, we consider a Brownian motion with time-dependent drift, namely,
\begin{align*}
X_{t}=\int_{0}^{t}v(s)\,ds+W_{t},\quad t\in\bR_{+},
\end{align*}
for some deterministic bounded and measurable function $v$ on $\bR_{+}$. For various choices of $v$, $c=1$, and $\xi=1$, we use Monte Carlo method to compute\footnote{The exact calculations of $\Delta_{1}$, $\Delta_{2}$, and $\Delta_{3}$ with non-constant $v$ are not available at this moment (even when $v$ is piecewise constant with a single jump).}
\begin{align*}
\Delta_{1}:=\bE\big(\overline{X}_{\mathsf{e}_{c}}(X_{\mathbf{e}_{c}}-\overline{X}_{\mathsf{e}_{c}})\big)-\bE\big(\overline{X}_{\mathsf{e}_{c}}\big)\bE\big(X_{\mathbf{e}_{c}}-\overline{X}_{\mathsf{e}_{c}}\big),\qquad\qquad\\
\Delta_{2}:=\bE\big(X_{\mathsf{e}_{c}}-\overline{X}_{\mathsf{e}_{c}}\big)-\bE\big(\underline{X}_{\mathsf{e}_{c}}\big),\quad\Delta_{3}:=\bE\big(e^{i X_{\mathsf{e}_{c}}}\big)-\bE\big(e^{i\overline{X}_{\mathsf{e}_{c}}}\big)\bE\big(e^{i\underline{X}_{\mathsf{e}_{c}}}\big),
\end{align*}
with $n=10^{4}$ sample paths and time step $\Delta t=10^{-4}$. The simulation results, summarized in Table \ref{Tab:ExpDiffLevy}, show that when $v$ is constant, the $\Delta_{1}$, $\Delta_{2}$ $\Delta_{3}$ are sufficiently close to zero (in theory, they are equal to zero). However, when $v$ is non-constant, each of $\Delta_{1}$, $\Delta_{2}$, and $\Delta_{3}$ is significantly away from zero, which is a clear contradiction to \eqref{eq:LevyPathDecompIndep}$-$\eqref{eq:WHLevyChf}. Therefore, there is no hope of deriving Wiener-Hopf type factorizations (in an analogous form of \eqref{eq:WHLevyChf}, \eqref{eq:WHLevySym}, or \eqref{eq:WHLevyOpr}) for a time-inhomogeneous real-valued L\'{e}vy process $X$ using properties like \eqref{eq:LevyPathDecompIndep} and \eqref{eq:LevyPathDecompIdtDist}, and other methods are sought.

\begin{center}
\renewcommand{\arraystretch}{1.5}
{\small
\begin{tabular}{|c|c|c|c|c|}
\hline
 & $v(s)\equiv 1$ & $v(s)\equiv -1$ & $v(s)=\1_{[0,1/2]}(s)-\1_{[1,3/2]}(s)$ & $v(s)=\cos(s)$ \\
\hline
$\Delta_{1}$ & $0.0004$ & -0.0020 & $-0.0672$ & $-0.1394$ \\
$\Delta_{2}$ & $0.0023$ & $0.0019$ & $-0.1428$ & $0.1447$ \\
$\Delta_{3}$ & -0.0037-0.0017i & -0.0021-0.0018i & $0.0254-0.0793i$ & $0.0439-0.0423i$ \\
\hline
\end{tabular}
}
\captionof{table}{Simulation results for $\Delta_{1}$, $\Delta_{2}$, and $\Delta_{3}$. For each choice of function $v$, the expectations are computed based on $n=10^{4}$ sample paths with time step $\Delta t=10^{-4}$.}\label{Tab:ExpDiffLevy}
\end{center}

In this paper, we derive a WHf for a time-inhomogeneous diffusion process $\varphi(s,a)$ (defined in \eqref{eq:varphi}) with time-dependent deterministic drift and volatility coefficients. Here, the WHf means a specific decomposition of the quantity
\begin{align}\label{eq:MainExp}
\bE\bigg(\int_{s}^{\tau}u(\varphi_{t}(s,a))h(t)\dif t\bigg),
\end{align}
where $\tau$ is an arbitrary stopping time, and $u$ and $h$ are suitable test functions. The main result of this paper regards such decomposition. It is presented in Theorem \ref{thm:Whvarphi}  in terms of the two passage times $\tau_{\ell}^{\pm}(s,a)$ of $\varphi(s,a)$ (defined in \eqref{eq:tauPlusMinus}) and of their functionals. To the best of our knowledge, Theorem \ref{thm:Whvarphi} is the very first result in the literature regarding the WHf for time-inhomogeneous L\'{e}vy processes. Corollary \ref{cor:WHvarphiInf} to Theorem \ref{thm:Whvarphi} is the key gateway to relating our present theory, via Corollary \ref{cor:WHABM}, to the classical WHf for time-homogenous arithmetic Brownian motion.

We admit that the form of the WHf presented in Theorem \ref{thm:Whvarphi} and, consequently, its version presented in Corollary \ref{cor:WHvarphiInf}, are more complicated, and less elegant -- one might say, than the classical WHf in the time-homogenous case. This is the consequence, or a price one needs to pay -- if you will, of both the time-inhomogeneity and the form of the functional that we are decomposing. The good (and desired) news though is that the classical WHf known for the time-homogenous case is nested in our formula. In fact, upon noting that for any $\ell\in\bR$ and $t\in(s,\infty)$, $\{\tau_{\ell}^{+}(s,a)\leq t\}=\{\sup_{r\in[s,t]}\varphi_{r}(s,a)\geq\ell\}$ and $\{\tau_{\ell}^{-}(s,a)\leq t\}=\{\inf_{r\in[s,t]}\varphi_{r}(s,a)\leq\ell\}$, we see that our result preserves the spirit of the operator-form WHf \eqref{eq:WHLevyOpr} that decomposes the $c$-resolvent of a time-homogeneous real-valued L\'{e}vy process in terms of functionals of its running maximum/minimum processes. In particular, when $v$ and $\sigma$ are both constant and with a special choice of $h$, our factorization of \eqref{eq:MainExp} recovers the operator-form WHf \eqref{eq:WHLevyOpr} for a Brownian motion with drift (see Corollary \ref{cor:WHABM} below).

Due to the lack of properties \eqref{eq:LevyPathDecompIndep} and \eqref{eq:LevyPathDecompIdtDist} in the time-inhomogeneous framework, we develop an entirely new theory in order to derive our WHf of \eqref{eq:MainExp}. Specifically, our methodology employs an in-depth analysis of the semigroups $(\cP_{\ell}^{\pm})_{\ell\in\bR_{+}}$ associated with $\tau_{\ell}^{\pm}(s,a)$, \textcolor{Brown}{defined in \eqref{eq:cPPlusMinus},} and of their generators $\Gamma^{\pm}$, together with a time-homogenization technique (cf. \cite[Section 3]{Bottcher:2014}), which does not rely on any property of $\varphi(s,a)$ analogous to \eqref{eq:LevyPathDecompIndep} or \eqref{eq:LevyPathDecompIdtDist}. One of the key tools in our analysis is a result of the so-called ``Noisy" Wiener-Hopf factorization (NWHf) which is presented as in Theorem \ref{thm:NoisyWHExistence}. The NWHf theory has been well-studied for time-homogeneous finite-state Markov chains (cf. \cite{JiangPistorius:2008}, \cite{KennedyWilliams:1990}, and \cite{Rogers:1994}). It provides a unique factorization for the generator matrix $\mathsf{\Lambda}$ of a time-homogeneous finite-state Markov chain $Y$ in terms of generator matrices of time-changed Markov chains related to $Y$,\footnote{The explicit factorization of $\mathsf{\Lambda}$ is only available when the volatility function $\sigma(\cdot)$ is constant and strictly positive. In general, the two generator matrices of time-changed processes are shown to be solutions to a matrix-valued quadratic equation \eqref{eq:NoisyWHFinMC} related to $\mathsf{\Lambda}$.} where the new time clocks are given as passage times of a Markov-modulated diffusion process of the form \eqref{eq:AddFuntlFinMC} below.\footnote{The name ``Noisy Wiener-Hopf" comes from the fact that the additive functional \eqref{eq:AddFuntlFinMC} contains an additional independent Brownian component compared to the classical matrix WHf theory (cf. \cite{BarlowRogersWilliams:1980}, \cite{Rogers:1994}, and \cite{Williams:2008}).} Theorem \ref{thm:NoisyWHExistence} can be regarded as an extension of the NWHf theory where the underlying Markov process is deterministic but with continuous state-space $\bR_{+}$ (see Remark \ref{rem:NoisyWH} below for further discussions), and is thus an equally important contribution to the literature.

There are several potential applications of our main result, as well as of the tools developed for the proofs. For instance, as discussed in Remarks \ref{rem:ProbIntpncPell} and \ref{rem:Volterra} below, Corollary \ref{cor:WHvarphiInf} can be used to compute the distribution of the local time of $\varphi$ at zero. Moreover, the semigroups $(\cP_{\ell}^{\pm})_{\ell\in\bR_{+}}$, of which the generators $\Gamma^{\pm}$ can in principle be computed from the operator equations \eqref{eq:NoisyWHPlus}$-$\eqref{eq:NoisyWHMinus} given in Theorem \ref{thm:NoisyWHExistence},\footnote{In order to solve $\Gamma^{\pm}$ from the operator equations \eqref{eq:NoisyWHPlus}$-$\eqref{eq:NoisyWHMinus}, we still need to establish the uniqueness of solutions to those equations, which remains an open problem.} can be used to solve the two-sided exit problems of $\varphi(s,a)$ from finite intervals; see \cite[Proposition 1]{JiangPistorius:2008} for such application in a time-homogeneous framework. The two-sided exit problems are frequently used in various applied disciplines, such as finance, insurance, fluid dynamics, and engineering; we refer to e.g. \cite{Rogers:1994}, \cite{JiangPistorius:2008} and references therein for specific applications.

We need to add that the present paper is a continuation of work towards developing Wiener-Hopf type theory for time-inhomogeneous Markov processes. The previous work in this direction is presented in \cite{BieCiaGonHua2019} and \cite{BieleckiChengCialencoGong:2019}.

The rest of the paper is organized as follows. In Section \ref{sec:SetupMainResult}, we first introduce the basic setup and assumptions of our model. Our main result on the WHf of the time-inhomogeneous arithmetic Brownian motion process $\varphi(s,a)$ is presented in Section \ref{subsec:MainResult}, followed by several important remarks in Section \ref{subsec:Comments}. A discussion of the relation of our main result with the WHf of time-homogeneous L\'{e}vy processes is carried out in Section \ref{subsec:TimeHomoLevy}. Section \ref{sec:Auxiliary} contains some auxiliary results which are needed in the proof of the main result. In Section \ref{sec:NoisyWHGammaPlusMinus}, we present a property of $\Gamma^{\pm}$ that is analogous to the NWHf. Section \ref{sec:MainProof} contains proofs of key results. {In Section \ref{sec:EgOneJumpDriftVol}}, we present a nontrivial example of our model for which the main assumptions are shown to be satisfied. Finally, in the appendix, we provide the proofs of some technical results.

\section{Setup and the Main Result}\label{sec:SetupMainResult}

\subsection{Basic Setup}\label{subsec:Setup}

Throughout this paper, we let $W:=(W_{t})_{t\in\bR_{+}}$ be a one-dimensional standard Brownian motion defined on a complete stochastic basis $(\Omega,\sF,\bF,\bP)$, where $\bF:=(\sF_{t})_{t\in\bR_{+}}$ is a filtration satisfying the usual conditions, and $\bR_{+}:=[0,\infty)$. For any $a\in\bR$ and $s\in\bR_{+}$, we consider a time-inhomogeneous diffusion process $\varphi(s,a):=(\varphi_{t}(s,a))_{t\in[s,\infty)}$, defined by
\begin{align}\label{eq:varphi}
\varphi_{t}(s,a):=a+\int_{s}^{t}v(r)\dif r+\int_{s}^{t}\sigma(r)\dif W_{r},\quad t\in[s,\infty),
\end{align}
where $v:\bR_{+}\rightarrow\bR$ and $\sigma:\bR_{+}\rightarrow\bR_{+}$ are $\cB(\bR_{+})$-measurable bounded functions.

For any $s\in\bR_{+}$ and $a,\ell\in\bR$, we define the passage times of $\varphi(s,a)$ as
\begin{align}\label{eq:tauPlusMinus}
\tau_{\ell}^{+}(s,a):=\inf\big\{t\in [s,\infty):\varphi_{t}(s,a)\geq\ell\big\}\,\,\,\,\,\text{and}\,\,\,\,\,\tau_{\ell}^{-}(s,a):=\inf\big\{t\in [s,\infty):\varphi_{t}(s,a)\leq\ell\big\},
\end{align}
with the convention $\inf\emptyset=\infty$. Both $\tau_{\ell}^{+}(s,a)$ and $\tau_{\ell}^{-}(s,a)$ are $\bF$-stopping times since $\varphi(s,a)$ is $\bF$-adapted and has continuous sample paths, and $\bF$ is right-continuous (cf. \cite[Chapter I, Proposition 1.28]{JacodShiryaev:2003}). In view of \eqref{eq:varphi}, for any $s\in\bR_{+}$ and $a,\ell\in\bR$, we have
\begin{align}\label{eq:Shifttau}
\tau_{\ell}^{\pm}(s,a)=\tau_{\ell-a}^{\pm}(s,0),
\end{align}
and $\tau^{+}_{\ell}(s,a)=s$ (respectively, $\tau^{-}_{\ell}(s,a)=s$) when $a\geq\ell$ (respectively, $a\leq\ell$). For notational convenience, hereafter we will write $\varphi_{t}(s)$ and $\tau_{\ell}^{\pm}(s)$ in place of $\varphi_{t}(s,0)$ and $\tau^{\pm}_{\ell}(s,0)$, respectively.

\medskip
We will use the following notations for various spaces of functions,
\begin{itemize}
\item $B_{b}(\bR_{+})$ is the space of $\cB(\bR_{+})$-measurable bounded real-valued functions on $\bR_{+}$. If need be we extend the domain of a function $f\in B_{b}(\bR_{+})$ to include infinity, and in such case we set $f(\infty)=0$.
\item $C(\bR_{+})$ (respectively, $C(\bR)$) is the space of continuous real-valued functions on $\bR_{+}$ (respectively, $\bR$).
\item $C_{c}(\bR_{+})$ (respectively, $C_{c}(\bR)$) is the space of continuous real-valued functions on $\bR_{+}$ (respectively, $\bR$) with compact support.
\item $C_{0}(\bR_{+})$ is the space of $f\in C(\bR_{+})$ such that $f$ vanishes at infinity.
\item $C_{e}(\bR_{+})$ is the space of $f\in C_{0}(\bR_{+})$ such that $f$ decays with exponential rate, i.e., there exist constants $K,\kappa\in(0,\infty)$, such that $|f(t)|\leq Ke^{-\kappa t}$ for all $t\in\bR_{+}$.
\item $C_{c}^{1}(\bR_{+})$ is the space of $f\in C_{0}(\bR_{+})$ such that $f$ is continuously differentiable on $\bR_{+}$ and has a compact support.
\item $C_{e,\text{cdl}}^{\text{ac}}(\bR_{+})$ is the space of $f\in C_{0}(\bR_{+})$ such that there exists a c\`{a}dl\`{a}g real-valued function $g_{f}$ on $\bR_{+}$, which {decays} with exponential rate, and
    \begin{align}\label{eq:SpaceCacecdl}
    f(t)= -\int_{t}^{\infty}g_{f}(r)\,dr,\quad\text{for all }t\in\bR_{+}.
    \end{align}
\item $C^{k}(\bR)$, $k\in\bN$, is the space of real-valued functions on $\bR$ which have continuous derivatives on $\bR$ up to order $k$.
\end{itemize}

We now introduce the following two families of operators associated with the passage times $\tau_{\ell}^{\pm}(s)$, which are key ingredients in our main result. For any $\ell\in\bR_{+}$, we define $\cP_{\ell}^{+}:B_{b}(\bR_{+})\rightarrow B_{b}(\bR_{+})$ and $\cP_{\ell}^{-}:B_{b}(\bR_{+})\rightarrow B_{b}(\bR_{+})$ as
\begin{align}\label{eq:cPPlusMinus}
\big(\cP^{+}_{\ell}f\big)(s):=\bE\Big(f\big(\tau^{+}_{\ell}(s)\big)\Big)\quad\text{and}\quad\big(\cP^{-}_{\ell}f\big)(s):=\bE\Big(f\big(\tau^{-}_{-\ell}(s)\big)\Big),\quad s\in\bR_{+}.
\end{align}
We will stipulate $(\cP^{\pm}_{\ell}f)(\infty)=0$ whenever we need to evaluate its value at infinity. Clearly, for any $f\in B_{b}(\bR_{+})$, $|(\cP^{\pm}_{\ell}f)(s)|\leq\|f\|_{\infty}<\infty$ for any $s\in\bR_{+}$, so that $\cP^{\pm}_{\ell}f\in B_{b}(\bR_{+})$.

\begin{remark}\label{rem:PassageTimes}
In most of the literature on WHf for Markov processes (cf. \cite{BarlowRogersWilliams:1980}, \cite{KennedyWilliams:1990}, and \cite{Williams:2008}), the passage times of additive functionals with strict inequalities are considered. In our setup this would mean considering
\begin{align}\label{eq:etaPlusMinus}
\eta_{\ell}^{+}(s,a):=\inf\big\{t\in [s,\infty):\varphi_{t}(s,a)>\ell\big\}\,\,\,\,\,\text{and}\,\,\,\,\,\eta_{\ell}^{-}(s,a):=\inf\big\{t\in
[s,\infty):\varphi_{t}(s,a)<\ell\big\},
\end{align}
instead of $\tau_{\ell}^{\pm}(s,a)$ given as in \eqref{eq:tauPlusMinus}, for any $s\in\bR_{+}$ and $a,\ell\in\bR$. Nevertheless, since our additive functionals are driven by Brownian noise, these two types of passage times are equal to each other $\bP$-a.s. Consequently, if we define $\cQ_{\ell}^{+}:B_{b}(\bR_{+})\rightarrow B_{b}(\bR_{+})$ and $\cQ_{\ell}^{-}:B_{b}(\bR_{+})\rightarrow B_{b}(\bR_{+})$ by
\begin{align*}
\big(\cQ^{+}_{\ell}f\big)(s):=\bE\Big(f\big(\eta^{+}_{\ell}(s)\big)\Big)\quad\text{and}\quad\big(\cQ^{-}_{\ell}f\big)(s):=\bE\Big(f\big(\eta^{-}_{-\ell}(s)\big)\Big),\quad s\in\bR_{+},
\end{align*}
then $(\cP_{\ell}^{\pm})_{\ell\in\bR_{+}}$ coincides with $(\cQ_{\ell}^{\pm})_{\ell\in\bR_{+}}$ on $B_{b}(\bR_{+})$. Therefore, our main factorization for \eqref{eq:MainExp}, given as in Theorem \ref{thm:Whvarphi} below, holds for either type of passage times.
\end{remark}

\subsection{Assumptions and Preliminaries}\label{sec:AssumpPrelim}

In this section, we will introduce some assumptions and state some preliminary results. We begin with the following mild assumption on the coefficient functions $v$ and $\sigma$.

\begin{assumption}\label{assump:vsigma}
Functions $v$ and $\sigma$ are such that
\begin{itemize}
\item [(i)] $v$ is bounded and c\`{a}dl\`{a}g;
\item [(ii)] $\sigma$ is c\`{a}dl\`{a}g, and there exists $0<\underline{\sigma}<\overline{\sigma}<\infty$ such that $\underline{\sigma}\leq\sigma(t)\leq\overline{\sigma}$, for all $t\in\bR_{+}$.
\end{itemize}
\end{assumption}

The following result, the proof of which is deferred to Appendix \ref{subsec:ProofPropgamma}, introduces two functions $\gamma^{+}$ and $\gamma^{-}$ that are key for our main result.

\begin{proposition}\label{prop:gammaPlusMinus}
Suppose that Assumption \ref{assump:vsigma} is valid. Then,
\begin{itemize}
\item [(i)] For any $0\leq s<t$, the following limits
    \begin{align*}
    \lim_{\ell\rightarrow 0+}\frac{1}{\ell}\,\bP\big(\tau^{+}_{\ell}(s)>t\big),\quad\lim_{\ell\rightarrow 0+}\frac{1}{\ell}\,\bP\big(\tau^{-}_{-\ell}(s)>t\big)
    \end{align*}
    exist and are both finite.
\item [(ii)] For any $s\in\bR_{+}$, define
    \begin{align*}
    \gamma^{+}(s,t)\!:=\!\begin{cases} \displaystyle{\lim_{\ell\rightarrow 0+}\!\frac{1}{\ell}\,\bP\big(\tau^{+}_{\ell}(s)>t\big)},&t\!\in\!(s,\infty), \\ \,0, &t\!\in\![0,s], \end{cases} \quad \gamma^{-}(s,t)\!:=\!\begin{cases} \displaystyle{\lim_{\ell\rightarrow 0+}\!\frac{1}{\ell}\,\bP\big(\tau^{-}_{-\ell}(s)>t\big)},&t\in(s,\infty), \\ \,0, &t\in[0,s]. \end{cases}
    \end{align*}
    Then $\gamma^{\pm}(s,\cdot)$ are non-increasing and continuous on $(s,\infty)$.
\item [(iii)] For any $0\leq s<t$, we have
    \begin{align*}
    &\frac{\sqrt{2}}{\sqrt{\pi\overline{\sigma}^{2}(t-s)}}\exp\bigg(\!-\frac{\overline{\sigma}^{2}\|v\|_{\infty}^{2}}{2\underline{\sigma}^{4}}(t-s)\bigg)-\frac{2\|v\|_{\infty}}{\underline{\sigma}^{2}}\Phi\bigg(\!-\frac{\overline{\sigma}\|v\|_{\infty}}{\underline{\sigma}^{2}}\sqrt{t-s}\bigg)\\
    &\quad\leq\gamma^{\pm}(s,t)\leq\frac{\sqrt{2}}{\sqrt{\pi\underline{\sigma}^{2}(t-s)}}\exp\bigg(\!-\frac{\|v\|_{\infty}^{2}}{2\underline{\sigma}^{2}}(t-s)\bigg)+\frac{2\|v\|_{\infty}}{\underline{\sigma}^{2}}\Phi\bigg(\frac{\|v\|_{\infty}}{\underline{\sigma}}\sqrt{t-s}\bigg).
    \end{align*}
\end{itemize}
\end{proposition}

Denoting by $\gamma:=\gamma^{+}+\gamma^{-}$, we consider the operator $\Gamma$ on $C_{e,\text{cdl}}^{\text{ac}}(\bR_{+})$ defined by
\begin{align}\label{eq:Gamma}
\big(\Gamma f\big)(s):=\int_{s}^{\infty}g_{f}(r)\gamma(s,r)\,dr,\quad s\in\bR_{+},
\end{align}
where we recall \eqref{eq:SpaceCacecdl} for the definition of $g_{f}$. The following lemma, the proof of which is deferred to Appendix \ref{subsec:ProofLemGamma}, establishes  well-definedness of $\Gamma$ and provides some of its basic properties.

\begin{lemma}\label{lem:Gamma}
Under Assumption \ref{assump:vsigma}, for every $f\in C_{e,\text{cdl}}^{\text{ac}}(\bR_{+})$, the integral on the right-hand side of \eqref{eq:Gamma} is finite for every $s\in\bR_{+}$. Moreover, $\Gamma f\in C_{e}(\bR_{+})$.
\end{lemma}

Our next assumption {regarding the range of $\lambda-\Gamma$ is key} in identifying $\Gamma$ as a strong generator.

\begin{assumption}\label{assump:RangeDens}
The functions $v$ and $\sigma$ are such that $\{(\lambda-\Gamma)f: f\in C_{e,\text{cdl}}^{\text{ac}}(\bR_{+})\}$ is dense in $C_{0}(\bR_{+})$ for some $\lambda>0$.
\end{assumption}

\begin{remark}
Assumption \ref{assump:RangeDens} is needed for application of the Hille-Yosida-Ray theorem (cf. \cite[Theorem 1.30]{BottcherSchillingWang:2014}). As is shown in Section \ref{sec:EgOneJumpDriftVol}, all $v$ and $\sigma$ satisfying Assumption \ref{assump:vsigma}, that are piecewise constant with finitely many jumps, satisfy Assumption \ref{assump:RangeDens}. However, it is difficult to verify that Assumption \ref{assump:RangeDens} is satisfied by any $v$ and $\sigma$ satisfying Assumption \ref{assump:vsigma} only. One possible strategy would be to study the existence of unknown $g_f$ that solves
\begin{equation}\label{eq:Vol}
\int_{s}^T g_f(r)(\gamma(s,r)+\lambda)dr = h(s),\,  s\in[0,T]
\end{equation}
 for an arbitrary $h\in H$, where $H$ is a dense subset of $h\in C_0([0,T])$, and for some fixed $\lambda>0$. This is indeed related to the properties of the family $\{(\lambda-\Gamma)f: f\in C_{e,\text{cdl}}^{\text{ac}}(\bR_{+})\}$ because for $f\in C_{e,\text{cdl}}^{\text{ac}}(\bR_{+})$ with $\supp f\subset[0,T]$, we have $f(t)=-\int_{t}^{T} g_f(r)dr$ and $(\lambda-\Gamma)f(s)=\int_{s}^T g_f(r)(\gamma(s,r)+\lambda)dr$.

Equation \eqref{eq:Vol} is a backward Volterra equation of the first kind with weakly singular kernel (cf.  \cite[Section 10.6]{PolyaninManzhirov:2008}). The existence of solutions to such an equation is usually studied by transforming it to a Volterra equation of the second kind using techniques in \cite[Section 10.3 and 10.6]{PolyaninManzhirov:2008}. This may eventually demand obtaining a sharp estimate for $\frac{\partial}{\partial s}\gamma^\pm(s,t)$ for $t>s$. This type of estimation is in general not easy to carry out even though $(s,t,\ell)\mapsto \bP(\tau^\pm_\ell(s)>t)$ exhibits certain smoothness, as the procedure of defining $\gamma$ involves unbounded operation $\lim_{\ell\to0+}\ell^{-1}$.
\end{remark}

For the rest of the paper, we will follow \cite[Definitions 1.1 \& 1.2]{BottcherSchillingWang:2014} for the definitions of semigroup and Feller semigroup.

We conclude the section by presenting a proposition, which is a key step towards establishing our main result. The proof of this proposition will be provided in Section \ref{subsec:ProofPropGammaGen}.

\begin{proposition}\label{prop:GammaGen}
Under Assumptions \ref{assump:vsigma}  and \ref{assump:RangeDens}, we have
\begin{itemize}
\item[(i)] the operator $\Gamma$ on $C_{e,\text{cdl}}^{\text{ac}}(\bR_{+})$ is closable and the closure of $\Gamma$, denoted by $(\overline{\Gamma},\sD(\overline{\Gamma}))$, is the strong generator of a Feller semigroup, say, $(\cP_{\ell})_{\ell\in\bR_{+}}$;
\item[(ii)] for any $f\in C_{e}(\bR_{+})$, the integral $\int_{0}^{L}\cP_{\ell}f\,d\ell$ converges in $B_{b}(\bR_{+})$, as $L\rightarrow \infty$, to the limit denoted by $\int_{0}^{\infty}\cP_{\ell}f\,d\ell$. Moreover, $\int_{0}^{\infty}\cP_{\ell}f\,d\ell\in\sD(\overline{\Gamma})$, and
    \begin{align}\label{eq:GammaIntcPell}
    \overline{\Gamma}\int_{0}^{\infty}\cP_{\ell}f\,d\ell= -f.
    \end{align}
\end{itemize}
\end{proposition}

We refer to Remark \ref{rem:ProbIntpncPell} below for a probabilistic interpretation of the semigroup $(\cP_\ell)_{\ell\in\bR_{+}}$.

\begin{remark}\label{rem:GammaDomain}
According to Proposition \ref{prop:GammaGen} (i), the operator $\Gamma$ is closable and its closure is the strong generator of a Feller semigroup. In the proof of the proposition we will first consider $\Gamma$ on $ C_{c}^{1}(\bR_{+})\subset C_{e,\text{cdl}}^{\text{ac}}(\bR_{+})$, which is dense in $C_{0}(\bR_{+})$. However, noting that each $f\in C_{e,\text{cdl}}^{\text{ac}}(\bR_{+})$ is differentiable a.e. on $\bR_{+}$ with $f'=g_{f}$ a.e. on $\bR_{+}$, the space $C_{e,\text{cdl}}^{\text{ac}}(\bR_{+})$ contains not only all $C^{1}_{c}(\bR_{+})$ functions (when $g_{f}\in C_{c}(\bR_{+})$), but also functions with {a certain type of} discontinuous derivatives. In particular, it contains functions $f$ such that $f'(t)=\1_{[0,T)}(t)$ a.e., for some $T\in(0,\infty)$. Those functions will turn out to be especially helpful in the study of the regularity of the semigroup generated by $\Gamma$ (see Lemma \ref{lem:EstcPell} below), which is crucial in proving Proposition \ref{prop:GammaGen} part (ii).
\end{remark}

\subsection{Main Result}\label{subsec:MainResult}

We now state the main result of this paper, Theorem \ref{thm:Whvarphi}, which provides a factorization for the expectation \eqref{eq:MainExp} in terms of the operators $(\cP_{\ell})_{\ell\in\bR_{+}}$ and $(\cP_{\ell}^{\pm})_{\ell\in\bR_{+}}$. The proof of this theorem will be provided in Section \ref{subsec:ProofThmWHvarphi}.

\begin{theorem}\label{thm:Whvarphi}
Under Assumptions \ref{assump:vsigma} and \ref{assump:RangeDens}, for any $h\in C_{e}(\bR_{+})$ and $u\in C(\bR)$ with
\begin{align}\label{eq:FinExpIntuh}
\bE\bigg(\int_{s}^{\infty}\big|u\big(\varphi_{t}(s,a)\big)h(t)\big|\,dt\bigg)<\infty,
\end{align}
for any $(s,a)\in\bR_{+}\times\bR$, we have
\begin{align}
\bE\bigg(\!\int_{s}^{\tau}\!\!u\big(\varphi_{t}(s,a)\big)h(t)\sigma^{2}(t)dt\!\bigg)\!&=2\!\int_{0}^{\infty}\!\!\!u(a\!+\!\ell)\bigg(\!\cP^{+}_{\ell}\!\!\!\int_{0}^{\infty}\!\!\!\cP_{y}h\,dy\!\bigg)\!(s)d\ell+2\!\int_{0}^{\infty}\!\!\!u(a\!-\!\ell)\bigg(\!\cP^{-}_{\ell}\!\!\!\int_{0}^{\infty}\!\!\!\cP_{y}h\,dy\!\bigg)\!(s)d\ell\nonumber\\
&\quad
-2\bE\left(\1_{\{\tau<\infty\}}\int_{0}^{\infty}u\big(\varphi_{\tau}(s,a)+\ell\big)\bigg(\cP^{+}_{\ell}\!\int_{0}^{\infty}\cP_{y}h\,dy\bigg)(\tau)\,d\ell\right)\nonumber\\
\label{eq:WHvarphi} &\quad
-2\bE\left(\1_{\{\tau<\infty\}}\int_{0}^{\infty}u\big(\varphi_{\tau}(s,a)-\ell\big)\bigg(\cP^{-}_{\ell}\!\int_{0}^{\infty}\cP_{y}h\,dy\bigg)(\tau)\,d\ell\right).
\end{align}
for any $\bF$-stopping time $\tau$.
\end{theorem}

In particular, by taking $\tau\equiv\infty$, we obtain the following corollary.

\begin{corollary}\label{cor:WHvarphiInf}
Under the setting of Theorem \ref{thm:Whvarphi}, for any $(s,a)\in\bR_{+}\times\bR$, we have
\begin{align}
\bE\bigg(\int_{s}^{\infty}u\big(\varphi_{t}(s,a)\big)h(t)\sigma^{2}(t)\,dt\bigg)&=2\int_{0}^{\infty}u(a+\ell)\bigg(\cP^{+}_{\ell}\!\int_{0}^{\infty}\!\cP_{y}h\,dy\bigg)(s)\,d\ell\nonumber\\
\label{eq:WHvarphiInf} &\quad +2\int_{0}^{\infty}u(a-\ell)\bigg(\cP^{-}_{\ell}\!\int_{0}^{\infty}\!\cP_{y}h\,dy\bigg)(s)\,d\ell.
\end{align}
Moreover, as a consequence of \eqref{eq:WHvarphi} and \eqref{eq:WHvarphiInf}, for any $\bF$-stopping time $\tau$, we have
\begin{align}
\bE\bigg(\int_{\tau}^{\infty}u\big(\varphi_{t}(s,a)\big)h(t)\sigma^{2}(t)\,dt\bigg)&=2\bE\left(\1_{\{\tau<\infty\}}\int_{0}^{\infty}u\big(\varphi_{\tau}(s,a)+\ell\big)\bigg(\cP^{+}_{\ell}\!\int_{0}^{\infty}\cP_{y}h\,dy\bigg)(\tau)\,d\ell\right)\nonumber\\
&\quad
+2\bE\left(\1_{\{\tau<\infty\}}\int_{0}^{\infty}u\big(\varphi_{\tau}(s,a)-\ell\big)\bigg(\cP^{-}_{\ell}\!\int_{0}^{\infty}\cP_{y}h\,dy\bigg)(\tau)\,d\ell\right).
\end{align}
\end{corollary}

The identity \eqref{eq:WHvarphiInf} generalizes the operator form of the WHf \eqref{eq:WHLevyOpr} for Brownian motion with drift, and is therefore named as the {\it WHf for the time-inhomogeneous arithmetic Brownian motion} $\varphi(s,a)$. We refer to Section \ref{subsec:TimeHomoLevy} for the detailed discussion regarding the connection between Corollary \ref{cor:WHvarphiInf} and classical WHf for time-homogeneous arithmetic Brownian motion.

\subsection{Some Comments on the Main Result}\label{subsec:Comments}

In this section, we will provide three important remarks regarding our main result. Let $L_{T}^{(s,a),x}$ be the local time of $\varphi(s,a)$ at level $x\in\bR$ up to time $T\in[s,\infty)$. Without loss of generality, we fix $s\in\bR_{+}$ and let $a=0$ throughout this section, and we write $L_{T}^{s,x}$ in place of $L_{T}^{(s,0),x}$\footnote{Recall that for simplicity we always write $\varphi_{t}(s)$ and $\tau_{\ell}^{\pm}(s)$ in place of $\varphi_{t}(s,0)$ and $\tau^{\pm}_{\ell}(s,0)$, respectively.}.

\begin{remark}[Probabilistic interpretation of $(\cP_\ell)_{\ell\in\bR_+}$]\label{rem:ProbIntpncPell}$\,$

\vspace{0.1cm}
\noindent
In this remark, we provide a probabilistic interpretation to $(\cP_\ell)_{\ell\in\bR_{+}}$ by connecting $(\cP_\ell)_{\ell\in\bR_{+}}$ to the distribution of $L_{T}^{s,0}$. Without loss of generality, we assume $\sigma(\cdot)\equiv 1$ throughout this remark. Otherwise, we consider the time-changed process $(\varphi_{s+\beta(t)}(s))_{t\in\bR_{+}}$, where $\beta(t):=\inf\{b\in\bR_{+}:\int_{s}^{s+b}\sigma^{2}(r)dr>t\}$, which, in view of Dambis-Dubins-Schwarz theorem (cf. \cite[Chapter 3, Theorem 4.6]{KaratzasShreve:1998}),  is a diffusion process of the form \eqref{eq:varphi} with drift function $v(\beta(\cdot))/\sigma^{2}(\beta(\cdot))$ and volatility function being constant and equal to one.

To begin with, for any $\ell\in(0,\infty)$, we define $\tau^{+,1}_{\ell}(s):=\tau^{+}_{\ell}(s)$, and for $j\in\bN$,
\begin{align*}
\tau^{-,j}_{\ell}(s):=\inf\big\{t>\tau^{+,j}_{\ell}(s):\,\varphi_{t}(s)= -\ell\big\},\quad\tau^{+,j+1}_{\ell}(s):=\inf\big\{t>\tau^{-,j}_{\ell}(s):\,\varphi_{t}(s)=\ell\big\}.
\end{align*}
The stopping time $\tau^{-,j}_{\ell}(s)$ is called the $j$-th downcrossing time of $[-\ell,\ell]$ by the process $\varphi(s)$.\footnote{Note that $\tau^{-,1}_{\ell}$ is different from $\tau^{-}_\ell$ defined in \eqref{eq:tauPlusMinus}.} In view of Girsanov's theorem, there exists a probability measure $\bP_{0}$ on $\sF^{W}_{T}$ which is equivalent to $\bP$ restricted on $\sF^{W}_{T}$, such that $(\varphi_{t}(s))_{t\in[s,T]}$ is a standard Brownian motion starting at time $s$ under $\bP_{0}$. Using the downcrossing
representation of the Brownian local time at zero (cf. \cite[Theorems 6.1 \& 6.16]{MortersPeres:2010}), we have
\begin{align*}
\lim_{\ell\rightarrow 0+}4\ell\sum_{j=1}^{\infty}\1_{\{\tau^{-,j}_{\ell}(s,0)\leq T\}}=L_{T}^{s,0},\quad\bP_{0}-\text{a.s.}\,\,\,\text{and thus }\,\bP-\text{a.s.}
\end{align*}
It is well-known that for standard Brownian motion, the local time at zero up to time $T$ has the same distribution as the running maximum up to time $T$ (cf. \cite[Theorem 6.10]{MortersPeres:2010}),  which is a continuous distribution. It then follows from \cite[Theorem 18.4 (a)]{JacodProtter:2004} that
\begin{align*}
\bP\Big(L_{T}^{s,0}\leq r\Big)=\bP\bigg(\lim_{\ell\rightarrow 0+}4\ell\sum_{j=1}^{\infty}\1_{\{\tau^{-,j}_{\ell}(s,0)\leq T\}}\leq r\bigg)=\lim_{\ell\rightarrow 0+}\bP\Bigg(\sum_{j=1}^{\infty}\1_{\{\tau^{-,j}_{\ell}(s,0)\leq T\}}\leq\frac{r}{4\ell}\Bigg).
\end{align*}
Note that
\begin{align*}
\bP\Bigg(\sum_{j=1}^{\infty}\1_{\{\tau^{-,j}_{\ell}(s,0)\leq T\}}\leq\frac{r}{4\ell}\Bigg)=\bP\Big(\tau^{-,\lfloor r/(4\ell)\rfloor+1}_{\ell}(s,0)\!>\!T\Big)=1\!-\!\bE\Big(\1_{[0,T]}\Big(\tau^{-,\lfloor r/(4\ell)\rfloor+1}_{\ell}(s,0)\Big)\Big).
\end{align*}
Moreover, by \eqref{eq:cPPlusMinus} and the strong Markov property of $\varphi(s)$,
\begin{align*}
\bE\Big(\1_{[0,T]}\Big(\tau^{-,\lfloor r/(4\ell)\rfloor+1}_{\ell}(s,0)\Big)\Big)=\Big(\cP^{+}_{\ell}\cP^{-}_{2\ell}\big(\cP^{+}_{2\ell}\cP^{-}_{2\ell}\big)^{\lfloor r/(4\ell)\rfloor}\1_{[0,T]}\Big)(s).
\end{align*}
In addition, we conjecture that
\begin{align}\label{eq:PplusPminusP}
\lim_{\ell\rightarrow 0+}\Big(\cP^{+}_{\ell}\cP^{-}_{2\ell}\big(\cP^{+}_{2\ell}\cP^{-}_{2\ell}\big)^{\lfloor r/(4\ell)\rfloor}\1_{[0,T]}\Big)(s)=\big(\cP_{r/2}\1_{[0,T]}\big)(s),\quad s\in\bR_{+}.
\end{align}
This, together with all three identities above, implies that
\begin{align}\label{eq:localtimeDist}
\bP\Big(L_{T}^{s,0}\leq r\Big)=1-\big(\cP_{r/2}\1_{[0,T]}\big)(s),
\end{align}
which provides a probabilistic interpretation to $(\cP_\ell)_{\ell\in\bR_+}$.

Because $\cP^{+}_{\ell}\cP^{-}_{2\ell}$ converges to identity as $\ell\to0+$, conjecture \eqref{eq:PplusPminusP} is essentially about the convergence of $\big(\cP^{+}_{2\ell}\cP^{-}_{2\ell}\big)^{\lfloor r/(4\ell)\rfloor}$ as $\ell\to0+$; for the notational convenience, we replace $2\ell$ and $r/2$ in \eqref{eq:PplusPminusP} by $\ell$ and $r$, and investigate $\big(\cP^{+}_{\ell}\cP^{-}_{\ell}\big)^{\lfloor r/\ell\rfloor}$ instead. To see why conjecture \eqref{eq:PplusPminusP} can be true we consider a simplified  case where the semigroups $(\cP^{\pm}_{\ell})_{\ell\in\bR_{+}}$ are of the form $\cP_{\ell}^{\pm}=e^{\ell A^{\pm}}$, where $A^{\pm}$s are operators with norms bounded by $\lambda$. Let $0 <\ell < r$. Applying  triangle inequality and a telescoping sum we have
\begin{align*}
&\left\|(e^{\ell A^+}e^{\ell A^-})^{\lfloor r/\ell\rfloor} - e^{r(A^++A^-)}\right\|_\infty \\
&\quad\le \left\| (e^{\ell A^+}e^{\ell A^-})^{\lfloor r/\ell\rfloor} - (e^{\ell(A^++A^-)})^{\lfloor r/\ell\rfloor} \right\|_\infty + \left\| (e^{\ell(A^++A^-)})^{\lfloor r/\ell\rfloor} -  e^{r(A^++A^-)} \right\|_\infty\\
&\quad\le \sum_{n=0}^{\lfloor r/\ell\rfloor} \left\|(e^{\ell A^+}e^{\ell A^-})^{n} \left((e^{\ell A^+}e^{\ell A^-}) - e^{\ell(A^++A^-)}\right) e^{(\lfloor r/\ell\rfloor-n)\ell(A^++A^-)}\right\|_\infty\\
&\qquad +  \left\|e^{\lfloor r/\ell\rfloor\ell(A^++A^-)}\left(1-e^{(r-\lfloor r/\ell\rfloor\ell)(A^++A^-)}\right)\right\|_\infty\\
&\quad\le e^{2r\lambda}\left(\sum_{n=0}^{\lfloor r/\ell\rfloor}\left\|(e^{\ell A^+}e^{\ell A^-}) - e^{\ell(A^++A^-)}\right\|_\infty + \left\|1-e^{(r-\lfloor r/\ell\rfloor\ell)(A^++A^-)}\right\|_\infty\right).
\end{align*}
Observe that, by Taylor expansion, we obtain
\begin{align*}
\left\|(e^{\ell A^+}e^{\ell A^-}) - e^{\ell(A^++A^-)}\right\|_\infty& = \left\|\left(\sum_{k=0}^\infty\frac{\ell^k \left(A^{+}\right )^k}{k!}\right)\left(\sum_{k=0}^\infty\frac{\ell^k \left(A^{-}\right )^k}{k!}{k!}\right)-\left(\sum_{k=0}^\infty\frac{\ell^k (A^++A^{-})^k}{k!}\right)\right\|_\infty \\
& \le  \ell^2 \left( \left(\sum_{k=2}^\infty\frac{\ell^{k-2}\lambda^k}{k!}\right)^2 + \sum_{k=2}^\infty\frac{\ell^{k-2}(2\lambda)^k}{k!} \right)
\end{align*}
and  $\left\|1-e^{(r-\lfloor r/\ell\rfloor\ell)(A^++A^-)}\right\|_\infty \le \ell \sum_{k=1}^\infty \frac{\ell^{k-1}\lambda^k}{k!}$. Combining the above, we obtain $\lim_{\ell\rightarrow 0+}\Big\|\big(e^{\ell A^{+}}e^{\ell A^{-}}\big)^{\lfloor r/\ell\rfloor}-e^{r( A^{+}+ A^{-})}\Big\|_{\infty}=0$.

To conclude this remark, let us perform a quick sanity check on \eqref{eq:localtimeDist} with $v\equiv 0$. In this case, it can be proved (e.g. using Lemma \ref{lem:SuvDens}) that $\cP^{+}_{\ell}=\cP^{-}_{\ell}$ and $\cP^{+}_{\ell}\cP^{-}_{\ell}=\cP_{\ell}$, and thus $\cP_{r/2}\1_{[0,T]}=\cP^{+}_{r}\1_{[0,T]}$. The identity \eqref{eq:localtimeDist} then follows from the fact, already recalled above, that for a standard Brownian motion, the local time at zero has the same distribution as the running maximum. 
\end{remark}

\begin{remark}[Intuition behind the WHf \eqref{eq:WHvarphiInf}]\label{rem:IntuitWHvarphiInf}$\,$

\vspace{0.1cm}
\noindent
In this remark, we present some intuition behind the formula \eqref{eq:WHvarphiInf}. The following discussion is mainly based on \eqref{eq:localtimeDist}, motivated by the downcrossing representation of the Brownian local time at zero (cf. \cite[Theorems 6.1 \& 6.16]{MortersPeres:2010}), as well as the occupation density formula for local time (cf. \cite[Chapter IV, (45.4)]{RogersWilliams:2000}). However, this intuitive discussion is {\it not} the blueprint for rigorous proof of \eqref{eq:WHvarphiInf}, mainly due to the lack of a rigorous validation to the conjecture \eqref{eq:PplusPminusP} and thus the lack of a rigorous validation to the conjecture \eqref{eq:localtimeDist}.


As in Remark \ref{rem:ProbIntpncPell}, we assume here that $\sigma(\cdot)\equiv 1$. For simplicity, we will formally establish \eqref{eq:WHvarphiInf} with $h\in C_{e}(\bR_{+})$ of the following piecewise-constant form
\begin{align}\label{eq:PiecewiseConstFuncth}
h(t)=\sum_{k=1}^{n}a_{k}\1_{[0,T_{k})}(t),\quad 0<T_{1}<\cdots<T_{n}\leq T,\quad a_{1},\ldots,a_{n}\in\bR,\quad n\in\bN.
\end{align}
Towards this end we proceed as follows.

We start by observing that
\begin{align}
\bE\bigg(\int_{s}^{\infty}h(t)\,dL^{s,0}_{t}\bigg)&=\sum_{k=1}^{n}a_{k}\bE\big(L^{s,0}_{T_{k}}\big)=\sum_{k=1}^{n}a_{k}\int_{0}^{\infty}\bP\big(L^{s,0}_{T_{k}}>r\big)\,dr\\
\label{eq:ExpnlocaltimeP} &=\sum_{k=1}^{n}a_{k}\int_{0}^{\infty}\big(\cP_{r/2}\1_{[0,T_{k})}\big)(s)\,dr=2\bigg(\int_{0}^{\infty}\cP_{y}h\,dy\bigg)(s).
\end{align}

Next, for any $x\in(0,\infty)$, using the strong Markov property of $\varphi(s)$ and noting that $L_{t}^{s,x}=0$ on $\{\tau_{x}^{+}(s)\geq t\}$, we have
\begin{align*}
\bE\bigg(\int_{s}^{\infty}h(t)\,dL_{t}^{s,x}\bigg)&=\sum_{k=1}^{n}a_{k}\,\bE\big(L^{s,x}_{T_{k}}\big)=\sum_{k=1}^{n}a_{k}\,\bE\Big(\bE\big(\1_{\{\tau_{x}^{+}(s)<T_{k}\}}L^{s,x}_{T_{k}}\,\big|\sF_{\tau_{x}^{+}(s)}\big)\Big)\\
&=\sum_{k=1}^{n}a_{k}\,\bE\big(H_{k}\big(\tau_{x}^{+}(s),\varphi_{\tau^{+}_{x}(s)}(s)\big)\big)=\sum_{k=1}^{n}a_{k}\,\bE\big(H_{k}\big(\tau_{x}^{+}(s),x\big)\big),
\end{align*}
where for each $k=1,\ldots,n$,
\begin{align*}
H_{k}(r,x):=\bE\Big(\1_{\{r<T_{k}\}}L_{T_{k}}^{(r,x),x}\Big)=\bE\big(\1_{\{r<T_{k}\}}L_{T_{k}}^{r,0}\big)=\bE\big(L_{T_{k}}^{r,0}\big),
\end{align*}
which is indeed constant for all $x\in(0,\infty)$. Hence, for any $x\in(0,\infty)$ we obtain that
\begin{align*}
\bE\bigg(\int_{s}^{\infty}h(t)\,dL_{t}^{s,x}\bigg)=\bE\left(\bE\bigg(\sum_{k=1}^{n}a_{k}L_{T_{k}}^{r,0}\bigg)\bigg|_{r=\tau_{x}^{+}(s)}\right)=\bE\left(\bE\bigg(\int_{r}^{\infty}h(t)\,dL_{t}^{r,0}\bigg)\bigg|_{r=\tau_{x}^{+}(s)}\right).
\end{align*}
Similarly, for any $x\in(-\infty,0)$ we have
\begin{align*}
\bE\bigg(\int_{s}^{\infty}h(t)\,dL_{t}^{s,x}\bigg)=\bE\left(\bE\bigg(\int_{r}^{\infty}h(t)\,dL_{t}^{r,0}\bigg)\bigg|_{r=\tau_{x}^{-}(s)}\right).
\end{align*}
Together with \eqref{eq:ExpnlocaltimeP} and recalling the definitions of $\cP^{\pm}_{x}$ in \eqref{eq:cPPlusMinus}, we deduce that
\begin{align}\label{eq:ExpInthLocalTime}
\bE\bigg(\int_{s}^{\infty}h(t)\,dL_{t}^{s,x}\bigg)=\begin{cases} \,\displaystyle{2\bigg(\cP^{+}_{x}\int_{0}^{\infty}\cP_{y}h\,dr\bigg)(s)}, & x\in\bR_{+} \\ \vspace{-0.3cm} \\ \,\displaystyle{2\bigg(\cP^{-}_{-x}\int_{0}^{\infty}\cP_{y}h\,dy\bigg)(s)}, & x\in(-\infty,0) \end{cases}.
\end{align}

Finally, we recall the well-known occupation density formula (cf. \cite[Chapter IV, (45.4)]{RogersWilliams:2000}) which implies that, for any Borel-measurable functions $u$ such that $\int_{s}^{\infty}|u\big(\varphi_{t}(s,0)\big)h(t)|\,dt<\infty$ $\bP$-a.s. with $h$ given as in \eqref{eq:PiecewiseConstFuncth}, we have
\begin{align}\label{eq:IntuvarphihLocalTime}
\int_{s}^{\infty}u\big(\varphi_{t}(s)\big)h(t)\,dt=\int_{-\infty}^{\infty}u(x)\bigg(\int_{s}^{\infty}h(t)\,dL^{s,x}_{t}\bigg)dx,\quad\bP-\text{a.s.}
\end{align}
The formula \eqref{eq:WHvarphiInf} follows from \eqref{eq:ExpInthLocalTime} and \eqref{eq:IntuvarphihLocalTime} upon taking expectation on both sides in \eqref{eq:IntuvarphihLocalTime}.
\end{remark}

\begin{remark}[A Volterra Equation for the kernel $\gamma$ of the generator $\Gamma$ of $(\cP_{\ell})_{\ell\in\bR_{+}}$]\label{rem:Volterra}$\,$

\vspace{0.1cm}
\noindent
In this remark, we will conclude from \eqref{eq:WHvarphiInf} that $\gamma$ is a solution of a Volterra equation of the first kind. The uniqueness of solution of the respective equation, however, remains an open problem so far. Nevertheless, the discussion in this remark provides a basis for a potential method of computing the kernel $\gamma$ and thus the associated generator $\Gamma$ and semigroup $(\cP_{\ell})_{\ell\in\bR_{+}}$ using the WHf \eqref{eq:WHvarphiInf}. This, together with \eqref{eq:localtimeDist}, leads to a potential application of \eqref{eq:WHvarphiInf} in computing the distribution of the local time of $\varphi(s)$.

To begin with, for any $T\in(s,\infty)$, we choose an arbitrary c\`{a}dl\`{a}g function $g_{f}$ on $\bR_{+}$ with $\supp(g_{f})\subset [0,T]$, and define $f\in C_{e,\text{cdl}}^{\text{ac}}(\bR_{+})$ as in \eqref{eq:SpaceCacecdl}. By \eqref{eq:WHvarphiInf} with $h=\Gamma f$ (so that $\supp(\Gamma f)\subset[0,T]$), for any $u\in C(\bR)$ satisfying \eqref{eq:FinExpIntuh} (in particular, for any $u\in C_{c}(\bR)$), we have
\begin{align}
&\int_{\bR}u(\ell)\int_{s}^{T}\rho(\ell,s,t)(\Gamma f)(t)\sigma^{2}(t)\,dt\,d\ell=\bE\bigg(\int_{s}^{T}u\big(\varphi_{t}(s,0)\big)(\Gamma f)(t)\sigma^{2}(t)\,dt\bigg)\nonumber\\
\label{eq:WHvarphiGammaf} &\quad =2\int_{0}^{\infty}u(\ell)\bigg(\cP^{+}_{\ell}\!\int_{0}^{\infty}\cP_{y}(\Gamma f)\,dy\bigg)(s)\,d\ell+2\int_{0}^{\infty}u(-\ell)\bigg(\cP^{-}_{\ell}\!\int_{0}^{\infty}\cP_{y}(\Gamma f)\,dy\bigg)(s)\,d\ell,\qquad
\end{align}
where we denote by $\rho(\cdot,s,t)$ the density of $\varphi_{t}(s)$ which has a $N(\int_{s}^{t}v(r)dr,\int_{s}^{t}\sigma^{2}(r)dr)$ distribution.\footnote{Throughout this paper we denote by $N(\mu,\sigma^{2})$ a univariate normal distribution with mean $\mu$ and variance $\sigma^{2}$.} Since $f\in\sD(\overline\Gamma)$, in view of Proposition \ref{prop:GammaGen} and \cite[Chapter 1, Proposition 1.5 (a) \& (c)]{EthierKurtz:2005}, we have
\begin{align*}
-f=\overline{\Gamma}\!\int_{0}^{\infty}\!\cP_{\ell}f\,d\ell=\lim_{L\rightarrow\infty}\overline{\Gamma}\!\int_{0}^{L}\!\cP_{\ell}f\,d\ell=\lim_{L\to\infty}\int_{0}^{L}\cP_{\ell}\,\overline{\Gamma}f\,d\ell=\int_{0}^{\infty}\cP_{\ell}\,\Gamma f\,d\ell,
\end{align*}
where the limits are taken in $L^{\infty}$-norm. Consequently, we can rewrite \eqref{eq:WHvarphiGammaf} as
\begin{align*}
\int_{\bR}\!u(\ell)\!\int_{s}^{T}\!\rho(\ell,s,t)(\Gamma f)(t)\sigma^{2}(t)\,dt\,d\ell=-2\!\int_{\bR}\!u(\ell)\Big(\1_{\bR_{+}}(\ell)\big(\cP^{+}_{\ell}f\big)(s)+\1_{(-\infty,0)}(\ell)\big(\cP^{-}_{-\ell}f\big)(s)\Big)d\ell,
\end{align*}
for any $u\in C_{c}(\bR)$. By the uniform continuity of $\rho(\cdot,s,t)$ on $\bR$ and Proposition \ref{prop:FellerSemiGroupcPPlusMinus}, both
\begin{align*}
\int_{s}^{T}\rho(\ell,s,t)(\Gamma f)(t)\sigma^{2}(t)\,dt\quad\text{and}\quad\1_{\bR_{+}}(\ell)\big(\cP^{+}_{\ell}f\big)(s)+\1_{(-\infty,0)}(\ell)\big(\cP^{-}_{-\ell}f\big)(s)
\end{align*}
are continuous with respect to $\ell$ on $\bR$. Hence we have
\begin{align*}
\int_{s}^{T}\rho(\ell,s,t)(\Gamma f)(t)\sigma^{2}(t)\,dt=-2\Big(\1_{\bR_{+}}(\ell)\big(\cP^{+}_{\ell}f\big)(s)+\1_{(-\infty,0)}(\ell)\big(\cP^{-}_{-\ell}f\big)(s)\Big),\quad\text{for any }\,\ell\in\bR.
\end{align*}
In particular, by letting $\ell=0$ and defining $\rho(s,t):=\rho(0,s,t)$, we deduce from the above equality, together with \eqref{eq:SpaceCacecdl} and \eqref{eq:Gamma}, that
\begin{align*}
\int_{s}^{T}\rho(s,t)\bigg(\int_{t}^{T}g_{f}(r)\gamma(t,r)\,dr\bigg)\sigma^{2}(t)\,dt=-2f(s)=2\int_{s}^{T}g_{f}(r)\,dr.
\end{align*}

Notice that $g_f$ and $\sigma^2$ are bounded. In view of the definition of $\rho(s,t)$ and Proposition \ref{prop:gammaPlusMinus} (iii), we see that $(r,t)\mapsto\rho(s,t)g_{f}(r)\gamma(t,r)\sigma^{2}(t)$ is integrable in $\{(r,t)\in\bR^{2}:r\in(s,T],t\in(s,r)\}$. It follows from Fubini's theorem that
\begin{align}\label{eq:WeakVolterraEq}
\int_{s}^{T}\bigg(\int_{s}^{r}\rho(s,t)\gamma(t,r)\sigma^{2}(t)\,dt\bigg)g_{f}(r)\,dr=\int_{s}^{T}2g_{f}(r)\,dr,
\end{align}
where we recall $\gamma=\gamma^{+}+\gamma^{-}$, and $\gamma^{\pm}$ are given in Proposition \ref{prop:gammaPlusMinus} (ii).

Next, we claim that $r\mapsto\int_{s}^{r}\rho(s,t)\gamma(t,r)\sigma^{2}(t)\,dt$ is continuous on $(s,T]$ for any $T>0$. In view of Proposition \ref{prop:gammaPlusMinus} (iii) and the definition of $\rho(s,t)$, for any $\varepsilon>0$ and $r\in(s,T]$, we can find some $\delta_{r}\in(0,r-s)$ small enough such that, for any $r'\in(r-\delta_{r},r+\delta_{r})$,
\begin{align*}
\bigg|\int_{r-\delta_{r}}^{r'}\rho(s,t)\gamma(t,r')\sigma^{2}(t)\,dt\bigg|\leq\int_{r-\delta_{r}}^{r'}\bigg(\frac{C_{1}}{\sqrt{r'-t}}+C_{2}\bigg)dt\leq\varepsilon,
\end{align*}
for some constants $C_{1},C_{2}>0$ depending only on $s$, $T$, $\|v\|_{\infty}$, $\overline{\sigma}^{2}$, and $\underline{\sigma}^{2}$. Therefore, we obtain that
\begin{align*}
&\lim_{r'\rightarrow r}\bigg|\int_{s}^{r'}\rho(s,t)\gamma(t,r')\sigma^{2}(t)\,dt-\int_{s}^{r}\rho(s,t)\gamma(t,r)\sigma^{2}(t)\,dt\bigg|\\
&\quad\leq\lim_{r'\rightarrow r}\bigg|\int_{s}^{r-\delta_{r}}\rho(s,t)\gamma(t,r')\sigma^{2}(t)\,dt-\int_{s}^{r-\delta_{r}}\rho(s,t)\gamma(t,r)\sigma^{2}(t)\,dt\bigg|+2\varepsilon=2\varepsilon,
\end{align*}
where in the last equality we used Proposition \ref{prop:gammaPlusMinus} (ii), the dominated convergence, and the fact that $|\rho(s,t)\gamma(t,r')\sigma^{2}(t)|\leq C_{1}/\sqrt{(t-s)(r-\delta_{r}-t)}+C_{2}$ which is integrable on $(s,r-\delta_{r})$. 

By \eqref{eq:WeakVolterraEq}, the continuity of $r\mapsto\int_{s}^{r}\rho(s,t)\gamma(t,r)\sigma^{2}(t)\,dt$ on $(s,T]$, as well as the hypothesis that $g_{f}$ is any c\`{a}dl\`{a}g function with $\supp(g_{f})\subset[0,T]$ and that $T\in(s,\infty)$ is arbitrary, we must have\footnote{Otherwise, without loss of generality, there exists $(a,b)\subset(s,\infty)$ such that $\int_{s}^{r}\rho(s,t)\gamma(t,r)\sigma^{2}(t)\,dt>2$ for all $r\in (a,b)$. Then taking $g_{f}(t):=\1_{(a,b]}(t)$ leads to a contradiction of \eqref{eq:WeakVolterraEq}.}
\begin{align*}
\int_{s}^{r}\gamma(t,r)\rho(s,t)\sigma^{2}(t)\,dt=2,\quad\text{for any }\,0\leq s\leq r.
\end{align*}
Therefore, for any (fixed) $r\in\bR_{+}$, by letting $\zeta(z):=\gamma(r-z,r)$ and $\kappa(q,z):=\rho(r-q,r-z)\sigma^{2}(r-z)$, we deduce that $\zeta$ solves the following Volterra equation of the first kind:
\begin{align}\label{eq:VolterraEqZeta}
\int_{0}^{q}\zeta(z)\kappa(q,z)\,dz=2,\quad q\in[0,r].
\end{align}
In addition, Proposition \ref{prop:gammaPlusMinus} (iii) implies that $\zeta\in L^{1}([0,z])$.
\end{remark}

\subsection{Connection with the Classical Wiener-Hopf Factorization for Real-Valued Time-Homogeneous L\'{e}vy Processes}\label{subsec:TimeHomoLevy}

In this section we will establish a connection between Corollary \ref{cor:WHvarphiInf} and the classical WHf for real-valued time-homogeneous L\'{e}vy processes. More precisely, in the following corollary, we recover from \eqref{eq:WHvarphiInf} the operator form of the WHf \eqref{eq:WHLevyOpr} in the setup of a time-homogenous one-dimensional Brownian motion with drift. The proof of the corollary is deferred to Section \ref{subsec:ProofCorWHABM}. As will be seen in the proof, Corollary \ref{cor:WHABM} is essentially a reformulation of \eqref{eq:WHvarphiInf} with $h(t)=e^{-ct}$, using the fact that $\cP^{+}_{\ell}$ and $\cP^{-}_{r}$ commute when $v$ and $\sigma$ are constant (cf. Lemma \ref{lem:SemigroupCommutABM}).

\begin{corollary}\label{cor:WHABM}
Suppose that $v(t)= v\in\bR$ and $\sigma(t)=\sigma\in(0,\infty)$, for all $t\in\bR_{+}$. For any $c\in(0,\infty)$ and $u\in C(\bR)$ with
\begin{equation}\label{eq:FinExpIntuh0}
\bE\bigg(\int_{0}^{\infty}e^{-ct}\big|u\big(\varphi_{t}(0,a)\big)\big|\,dt\bigg)<\infty,\quad\text{for any }\,a\in\bR,
\end{equation}
we define, for any $a\in\bR$,
\begin{align*}
\big(\cE_{c}u\big)(a)&:=c\,\bE\bigg(\int_{0}^{\infty}e^{-ct}u\big(\varphi_{t}(0,a)\big)\,dt\bigg),\\
\big(\cE_{c}^{+}u\big)(a)&:=c\,\bE\bigg(\int_{0}^{\infty}e^{-ct}u\big(\overline{\varphi}_{t}(0,a)\big)\,dt\bigg),\\
\big(\cE_{c}^{-}u\big)(a)&:=c\,\bE\bigg(\int_{0}^{\infty}e^{-ct}u\big(\underline{\varphi}_{t}(0,a)\big)\,dt\bigg),
\end{align*}
where for any $t\in\bR_{+}$, $\overline{\varphi}_{t}(0,a)\!:=\sup_{r\in[0,t]}\varphi_{r}(0,a)$ and $\underline{\varphi}_{t}(0,a)\!:=\inf_{r\in[0,t]}\varphi_{r}(0,a)$. Then, we have
\begin{align}\label{eq:WHABM}
\cE_{c}u=\cE_{c}^{+}\cE_{c}^{-}u=\cE_{c}^{-}\cE_{c}^{+}u.
\end{align}
\end{corollary}

\section{Auxiliaries}\label{sec:Auxiliary}

In this section, we will present some auxiliary results needed for the proof of our main result.

\subsection{An Auxiliary Time-Homogeneous Markov Family}\label{subsec:TimeHomMarkovFamily}

Here we will introduce a time-homogeneous Markov family $\wt{\cM}$ by applying standard time homogenization techniques to the time-inhomogeneous Markov family $\wh{\cM}$ to be defined below and associated with $\{\varphi(s,a),(s,a)\in\bR_{+}\times\bR\}$. Similar construction was done in \cite[Section 4.1]{BieleckiChengCialencoGong:2019}. Hereafter, we denote by $\sZ:=\bR_{+}\times\bR$, $\overline{\sZ}:=\sZ\cup\{(\infty,\infty)\}$, $\overline{\bR}_{+}:=[0,\infty]$, and $\overline{\bR}:=\bR\cup\{\infty\}$.

Let $\wh{\Omega}=C(\bR_{+})$ which is the space of real-valued continuous functions on $\bR_{+}$. We stipulate $\wh{\omega}(\infty)=\infty$ for every $\wh{\omega}\in\wh{\Omega}$. One can construct a {\it standard} canonical\footnote{By canonical, we mean $\wh{\varphi}_{t}(\wh{\omega})=\wh{\omega}(t)$ for all $t\in\overline\bR_{+}$.} time-inhomogeneous Markov family (cf. \cite[Section I.6, Definition 6]{GikhmanSkorokhod:2004})
\begin{align*}
\wh{\cM}:=\Big\{\Big(\wh{\Omega},\wh{\sF},\wh{\bF}_{s},\big(\wh{\varphi}_{t}\big)_{t\in[s,\infty]},\wh{\bP}_{s,a}\Big),\,(s,a)\in\overline{\sZ}\Big\}
\end{align*}
with transition function $\wh{P}$ given by\footnote{We stipulate $\varphi_{\infty}(s,a)\equiv\infty$, for any $(s,a)\in\overline{\sZ}$.}
\begin{align}\label{eq:TranProbWhvarphi}
\wh{P}(s,a,t,A):=\bP\big(\varphi_{t}(s,a)\in A\big),\quad t\in[s,\infty],\,\,\,\,(s,a)\in\overline{\sZ},\,\,\,\,A\in\cB(\overline{\bR}).
\end{align}
A routine check verifies that $\wh{P}$ is a Feller transition function.\footnote{We refer to \cite[Section I.6, Page 78]{GikhmanSkorokhod:2004} for the definition of the Feller transition function.} This allows us to apply \cite[Section I.6, Theorem 3]{GikhmanSkorokhod:2004} in order to prove existence of such Markov family $\wh{\cM}$.\footnote{The details of construction of the family $\wh{\cM}$ and its properties take much space, and therefore are not given here. They can be obtained from the authors upon request.} In particular, it holds that, for $0\leq t \leq t'$ and $A\in \cB(\overline{\bR})$,
\begin{align}\label{eq:TranProbWhvarphiWhProb}
\wh{P}(t,y,t',A)=\wh{\bP}_{t,y}\big(\wh{\varphi}_{t'}\in A\big).
\end{align}
Moreover, by investigating the finite dimensional distributions of $\wh{\varphi}:=(\wh{\varphi}_{t})_{t\in[s,\infty]}$ under $\wh{\bP}_{s,a}$, and since $\wh{\varphi}$ admits continuous sample paths, it can be shown that, for any $(s,a)\in\overline{\sZ}$,
\begin{align}\label{eq:LawWhvarphi}
\text{the law of }\,\wh{\varphi}\,\text{ under }\,\wh{\bP}_{s,a}\,\,=\,\,\text{the law of }\,\varphi(s,a)\,\text{ under }\,\bP.
\end{align}

Considering the standard Markov family $\wh{\cM}$, for any $s\in\bR_{+}$ and $\ell\in\bR$, we define
\begin{align}\label{eq:WhtauPlusMinus}
\wh{\tau}^{+}_{\ell}(s):=\inf\big\{t\in [s,\infty]:\,\wh{\varphi}_{t}\geq\ell\big\}\quad\text{and}\quad\wh{\tau}^{-}_{\ell}(s):=\inf\big\{t\in [s,\infty]:\,\wh{\varphi}_{t}\leq\ell\big\},
\end{align}
which are both $\wh{\bF}_{s}$-stopping times in light of the continuity of $\wh{\varphi}$ and the right-continuity of the filtration $\wh{\bF}_{s}$. In what follows, when no confusion arises, we will omit the variable $s$ in $\wh{\tau}^{\pm}_{\ell}(s)$. The following proposition is an immediate consequence of \eqref{eq:tauPlusMinus}, \eqref{eq:LawWhvarphi}, and \eqref{eq:WhtauPlusMinus}, so its proof is skipped.

\begin{proposition}\label{prop:ExpWhPExpP}
For any $f\in L^{\infty}(\overline{\bR}_{+})$, $(s,a)\in\sZ$, and $\ell\in\bR$, we have
\begin{align*}
\wh{\bE}_{s,a}\Big(f\big(\wh{\tau}^{\pm}_{\ell}\big)\Big)&=\bE\Big(f\big(\tau^{\pm}_{\ell}(s,a)\big)\Big).
\end{align*}
\end{proposition}

Next, we will transform the time-inhomogeneous Markov family $\wh{\cM}$ into a {\it time-homogeneous} Markov family
\begin{align*}
\wt{\cM}:=\Big\{\big(\wt{\Omega},\wt{\sF},\wt{\bF},(Z_{t})_{t\in\overline{\bR}_{+}},(\theta_{r})_{r\in\bR_{+}},\wt{\bP}_{s,a}\Big),\,(s,a)\in\overline{\sZ}\Big\}
\end{align*}
following the setup in \cite[Section 3]{Bottcher:2014}. The construction of $\wt{\cM}$ proceeds as follows.
\begin{itemize}
\item We let $\wt{\Omega}:=\overline{\bR}_{+}\times\wh{\Omega}$ to be the new sample space, with elements $\wt{\omega}=(s,\omega)$, where $s\in\overline{\bR}_{+}$ and
    $\omega\in\Omega$. On $\wt{\Omega}$ we consider the $\sigma$-field
    \begin{align*}
    \wt{\sF}:=\Big\{\wt{A}\subset\wt{\Omega}:\,\wt{A}_{s}\in\wh{\sF}_{\infty}^{s}\,\,\,\text{for any }s\in\overline{\bR}_{+}\Big\},
    \end{align*}
    where $\wt{A}_{s}:=\{\omega\in\Omega:(s,\omega)\in\wt{A}\}$ and $\wh{\sF}_{\infty}^{s}:=\sigma(\bigcup_{t\ge s}\wh\sF^s_t)$.
\item We let $\overline{\sZ}=\sZ\cup\{(\infty,\infty)\}$ be the new state space, with elements $z=(s,a)$. On $\sZ=\bR_{+}\times\bR$ we consider the $\sigma$-field
	\begin{align*}
    \wt{\cB}(\sZ):=\Big\{\wt{B}\subset\sZ:\,\wt{B}_{s}\in\cB(\bR)\,\,\,\text{for any }s\in\bR_{+}\Big\},
	\end{align*}
    where $\wt{B}_{s}:=\{a\in\bR:(s,a)\in\wt{B}\}$. Let $\wt{\cB}(\overline{\sZ}):=\sigma(\wt{\cB}(\sZ)\cup\{(\infty,\infty)\})$.
\item We consider a family of probability measures $\{\wt{\bP}_{(s,a)},(s,a)\in\overline{\sZ}\}$, where, for $(s,a)\in\overline{\sZ}$,
	\begin{align}\label{eq:WtProb}
    \wt{\bP}_{s,a}\big(\wt{A}\big):=\wh{\bP}_{s,a}\big(\wt{A}_{s}\big),\quad\wt{A}\in\wt{\sF}.
	\end{align}
    Frequently, for convenience, we will write $\wt{\bP}_{z}\big(\wt{A}\big)$ in place of $\wt{\bP}_{(s,a)}\big(\wt{A}\big)$ where $z=(s,a)$.
\item We consider the process $Z:=(Z_{t})_{t\in\overline{\bR}_{+}}$ on $(\wt{\Omega},\wt{\sF})$, where, for $t\in\overline{\bR}_{+}$,
	\begin{align}\label{eq:ProcZ}
    Z_{t}(\wt{\omega}):=\big(s+t,\widehat{\varphi}_{s+t}(\omega)\big),\quad\wt{\omega}=(s,\omega)\in\wt{\Omega}.
	\end{align}
	Hereafter, we denote the two components of $Z$ by $Z^{1}$ and $Z^{2}$, respectively.
\item On $(\wt{\Omega},\wt{\sF})$, we define $\wt{\bF}:=(\wt{\sF}_{t})_{t\in\overline{\bR}_{+}}$, where $\wt{\sF}_{t}:=\wt{\sG}_{t+}$ (with the convention
    $\wt{\sG}_{\infty+}=\wt{\sG}_{\infty}$), and $(\wt{\sG}_{t})_{t\in\overline{\bR}_{+}}$ is the completion of the natural filtration generated by $Z$ with respect to the set of probability measures $\{\wt{\bP}_{z},z\in\overline{\sZ}\}$ (cf. \cite[Chapter
    I, Page 43]{GikhmanSkorokhod:2004}).
\item Finally, for any $r\in\bR_{+}$, we consider the shift operator ${\theta}_{r}:\wt{\Omega}\rightarrow\wt{\Omega}$ defined by
    \begin{align*}
    \theta_{r}\,\wt{\omega}=(u+r,\omega_{\cdot+r}),\quad\wt{\omega}=(u,\omega)\in\wt{\Omega}.
    \end{align*}
    It follows that $Z_{t}\circ{\theta}_{r}=Z_{t+r}$, for any $t,r\in\bR_{+}$.
\end{itemize}

We define a transition function $\wt{P}$ on $\overline{\sZ}\times\overline{\bR}_{+}\times\wt{\cB}(\overline{\sZ})$ by
\begin{align*}
\wt{P}\big(z,t,\wt{B}\big):=\wt{\bP}_{z}\big(Z_{t}\in\wt{B}\big),\quad z=(s,a)\in\overline{\sZ},\quad t\in\overline{\bR}_{+},\quad\wt{B}\in\wt{\cB}(\overline{\sZ}).
\end{align*}
In view of \eqref{eq:TranProbWhvarphiWhProb} and \eqref{eq:WtProb} we have
\begin{align}\label{eq:TranProbZ}
\wt{P}\big(z,t,\wt{B}\big)=\wh{\bP}_{s,a}\Big(\wh{\varphi}_{t+s}\in\wt{B}_{s+t}\Big)=\wh{P}\big(s,a,s+t,\wt{B}_{s+t}\big).
\end{align}
Recall that the transition function $\wh{P}$, defined in \eqref{eq:TranProbWhvarphi}, is a Feller transition function. This, together with \cite[Theorem 3.2]{Bottcher:2014}, implies that $\wt{P}$ is also a Feller transition function. In light of the continuity of the sample paths of $Z$, and invoking \cite[Section I.4, Theorem 7]{GikhmanSkorokhod:2004}, we conclude that $\wt{\cM}$ is a {\it time-homogeneous strong} Markov family.

For $\ell\in\bR$, we define
\begin{align*}
\wt{\tau}^{+}_{\ell}:=\inf\big\{t\in\overline{\bR}_{+}:\,Z^{2}_{t}\geq\ell\big\}\quad\text{and}\quad\wt{\tau}^{-}_{\ell}:=\inf\big\{t\in\overline{\bR}_{+}:Z^{2}_{t}\leq\ell\big\}.
\end{align*}
Both $\wt{\tau}_{\ell}^{+}$ and $\wt{\tau}_{\ell}^{-}$ are $\wt{\bF}$-stopping times since $Z^{2}$ has continuous sample paths and since $\wt{\bF}$ is right-continuous. By Proposition \ref{prop:ExpWhPExpP}, \eqref{eq:WtProb}, and \eqref{eq:ProcZ}, for any $f\in B_{b}(\bR_{+})$, $(s,a)\in\sZ$, and $\ell\in\bR$,
\begin{align}\label{eq:WtExptauExptau}
\wt{\bE}_{s,a}\Big(f\Big(Z^{1}_{\wt{\tau}^{\pm}_{\ell}}\Big)\Big)=\bE\Big(f\big(\tau^{\pm}_{\ell}(s,a)\big)\Big),
\end{align}
which, together with \eqref{eq:Shifttau}, implies that
\begin{align}\label{eq:ShiftWttau}
\wt{\bE}_{s,a}\Big(f\Big(Z^{1}_{\wt{\tau}^{\pm}_{\ell}}\Big)\Big)=\wt{\bE}_{s,0}\Big(f\Big(Z^{1}_{\wt{\tau}^{\pm}_{\ell-a}}\Big)\Big).
\end{align}
We conclude this section by presenting a couple of lemmas which will be needed in the sequel. The first one provides an important identity related to the strong Markov property of $Z$. It is a simple adaption of \cite[Lemma 4.2]{BieleckiChengCialencoGong:2019}, and the proof is therefore omitted here.

\begin{lemma}\label{lem:StrongMarkovWtcM}
Let $\wt{\tau}$ is an $\wt{\bF}$-stopping time and let $f\in B_{b}(\bR_{+})$. Then, for any $(s,a)\in{\sZ}$ and $\ell\in\bR$, we have
\begin{align*}
\1_{\{\wt{\tau}\leq\wt{\tau}^{\pm}_{\ell}\}}\wt{\bE}_{s,a}\Big(f\Big(Z^{1}_{\wt{\tau}^{\pm}_{\ell}}\Big)\Big|\wt{\sF}_{\wt{\tau}}\Big)=\1_{\{\wt{\tau}\leq\wt{\tau}^{\pm}_{\ell}\}}\wt{\bE}_{Z^{1}_{\wt{\tau}},Z^{2}_{\wt{\tau}}}\Big(f\Big(Z^{1}_{\wt{\tau}^{\pm}_{\ell}}\Big)\Big),
\end{align*}
where we clarify that $\wt{\bE}_{Z^{1}_{\wt{\tau}},Z^{2}_{\wt{\tau}}}(\,\cdot\,)$ reads $\wt{\bE}_{t,b}(\,\cdot\,)|_{(t,b)=(Z^{1}_{\wt{\tau}},Z^{2}_{\wt{\tau}})}$.
\end{lemma}

The next lemma, the proof of which is deferred to Appendix \ref{subsec:ProofCont}, provides the regularity of $\wt{\bE}_{s,0}(f(\wt{\tau}^{\pm}_{\ell}))$ with respect to different variables.

\begin{lemma}\label{lem:ContWtExptau}
Under Assumption \ref{assump:vsigma}, for any $f\in C_{0}(\bR_{+})$, we have
\begin{itemize}
\item[(i)] $\ell\mapsto\wt{\bE}_{s,0}(f(Z^{1}_{\wt{\tau}^{\pm}_{\ell}}))$ is uniformly continuous on $\bR_{+}$, uniformly in $s\in\bR_{+}$;
\item[(ii)] for each $\ell\in\bR_{+}$, $s\mapsto\wt{\bE}_{s,0}(f(Z^{1}_{\wt{\tau}^{\pm}_{\ell}}))$ belongs to $C_{0}(\bR_{+})$.
\end{itemize}
\end{lemma}

\subsection{The Feller Semigroup Property and Strong Generators of \texorpdfstring{$(\cP^{\pm}_{\ell})_{\ell\in\bR_{+}}$}{}}\label{sec:SemiGroupGencPPlusMinus}

In this section, we will investigate the Feller semigroup property (cf. \cite[Definitions 1.2]{BottcherSchillingWang:2014}) of $(\cP^{\pm}_{\ell})_{\ell\in\bR_{+}}$ defined by \eqref{eq:cPPlusMinus}. Moreover, we will characterize the strong generators of $(\cP^{\pm}_{\ell})_{\ell\in\bR_{+}}$ on $C_{e,\text{cdl}}^{\text{ac}}(\bR_{+})$, which is a dense subset of their domains.

In view of \eqref{eq:cPPlusMinus} and \eqref{eq:WtExptauExptau} we can rewrite $(\cP^{\pm}_{\ell})_{\ell\in\bR_{+}}$ in terms of the time-homogeneous Markov family $\wt{\cM}$ as follows: for any $f\in B_{b}(\bR_{+})$,
\begin{align}\label{eq:cPPlusMinusWt}
\big(\cP^{\pm}_{\ell}f\big)(s)=\wt{\bE}_{s,0}\Big(f\Big(Z^{1}_{\wt{\tau}^{\pm}_{\ell}}\Big)\Big),\quad s\in\bR_{+}.
\end{align}
This representation will be conveniently used later, starting with the following proposition.

\begin{proposition}\label{prop:FellerSemiGroupcPPlusMinus}
Under Assumption \ref{assump:vsigma}, $(\cP^{\pm}_{\ell})_{\ell\in\bR_{+}}$ is a Feller semigroup.
\end{proposition}

\begin{proof}
We will only verify the Feller semigroup property for $(\cP^{+}_{\ell})_{\ell\in\bR_{+}}$, as the ``$-$" case can be dealt with in an analogous way.

We first verify the semigroup property of $(\cP^+_\ell)_{\ell \in \bR_+}$. To this end, we fix any $f\in B_{b}(\bR_{+})$ and $s\in\bR_{+}$. By \eqref{eq:tauPlusMinus} and \eqref{eq:cPPlusMinus}, we  have $(\cP^{+}_{0}f)(s)=\bE(f(\tau^{+}_{0}(s)))=f(s)$. Next, for any $\ell\in\bR_{+}$ and $h>0$, by \eqref{eq:cPPlusMinusWt} and \eqref{eq:WtExpftauell2}, we obtain that
\begin{align*}
\big(\cP^{+}_{\ell+h}f\big)(s)=\wt{\bE}_{s,0}\Big(f\Big(Z^{1}_{\wt{\tau}^{+}_{\ell+h}}\Big)\Big)&=\wt{\bE}_{s,0}\bigg(\1_{\{\wt{\tau}^{+}_{\ell}<\infty\}}\wt{\bE}_{Z^{1}_{\wt{\tau}^{+}_{\ell}},0}\Big(f\Big(Z^{1}_{\wt{\tau}^{+}_{h}}\Big)\Big)\bigg)=\wt{\bE}_{s,0}\Big(\1_{\{\wt{\tau}^{+}_{\ell}<\infty\}}\cP^{+}_{h}f\Big(Z^{1}_{\wt{\tau}^{+}_{\ell}}\Big)\Big)\\
&=\wt{\bE}_{s,0}\Big(\cP^{+}_{h}f\Big(Z^{1}_{\wt{\tau}^{+}_{\ell}}\Big)\Big)=\big(\cP^{+}_{\ell}\cP^{+}_{h}f\big)(s),
\end{align*}
where we use our convention that $g(\infty)=0$ for any $g\in B_{b}(\bR_{+})$. Hence, $(\cP^{+}_{\ell})_{\ell\in\bR_{+}}$ is a semigroup on $B_{b}(\bR_{+})$. The positivity and contraction properties of $\cP^{+}_{\ell}$ on $B_{b}(\bR_{+})$, for any $\ell\in\bR_{+}$, follow immediately from its definition \eqref{eq:cPPlusMinus}.

Finally, the Feller property of $(\cP^{+}_{\ell})_{\ell\in\bR_{+}}$ follows immediately from Lemma \ref{lem:ContWtExptau} (ii) and from \eqref{eq:cPPlusMinusWt}, while the strong continuity on $C_{0}(\bR_{+})$ is a direct consequence of Lemma \ref{lem:ContWtExptau} (i). The proof of the proposition is now complete.
\end{proof}

Let $(\Gamma^{+},\sD(\Gamma^{+}))$ (respectively, $(\Gamma^{-},\sD(\Gamma^{-}))$) be the strong generator of $(\cP^{+}_{\ell})_{\ell\in\bR_{+}}$ (respectively, $(\cP^{-}_{\ell})_{\ell\in\bR_{+}}$).\footnote{Although the strong and weak generators are the same for Feller semigroups, cf. \cite[Theorem 2.1.3]{Pazy:1983}, we use the convention of referring to a generator as the strong generator.} Since $(\cP^{\pm}_{\ell})_{\ell\in\bR_{+}}$ are Feller semigroups it holds that $\sD(\Gamma^{\pm})\subset C_{0}(\bR_{+})$.

The next proposition provides an integral representation for $\Gamma^{\pm}$ on $C_{e,\text{cdl}}^{\text{ac}}(\bR_{+})$.

\begin{proposition}\label{prop:GammaPlusMinusGen}
Under Assumption \ref{assump:vsigma}  we have $C_{e,\text{cdl}}^{\text{ac}}(\bR_{+})\subset\sD(\Gamma^{\pm})$. Moreover, for any $f\in C_{e,\text{cdl}}^{\text{ac}}(\bR_{+})$,
\begin{align}\label{eq:GenGammaPlusMinus}
\big(\Gamma^{\pm}f\big)(s)=\int_{0}^{\infty}g_{f}(t)\gamma^{\pm}(s,t)\,dt,\quad s\in\bR_{+},
\end{align}
where the integral on the right-hand side is finite.
\end{proposition}

\begin{proof}
We will only present the proof for the ``$+$" case, as the ``$-$" case can be proved in an analogous way. For any $f\in C_{e,\text{cdl}}^{\text{ac}}(\bR_{+})$, using arguments similar to those in the proof of Lemma \ref{lem:Gamma}, we deduce that the integral on the right-hand side of \eqref{eq:GenGammaPlusMinus} is finite for every $s\in\bR_{+}$, and belongs to $C_{e}(\bR_{+})$.

In view of \cite[Theorem 1.33]{BottcherSchillingWang:2014} and Proposition \ref{prop:FellerSemiGroupcPPlusMinus}, in order to prove $C_{e,\text{cdl}}^{\text{ac}}(\bR_{+})\subset\sD(\Gamma^{+})$ and \eqref{eq:GenGammaPlusMinus}, we only need to show that, for any $f\in C_{e,\text{cdl}}^{\text{ac}}(\bR_{+})$ and $s\in\bR_{+}$,
\begin{align}\label{eq:PointLimGammaf}
\lim_{\ell\rightarrow 0+}\frac{1}{\ell}\big(\big(\cP^{+}_{\ell}f\big)(s)-f(s)\big)=\int_{0}^{\infty}g_{f}(t)\gamma^{+}(s,t)\,dt,
\end{align}
Towards this end, for any $\ell\in(0,1)$, an application of integration by parts yields
\begin{align*}
\frac{1}{\ell}\big(\big(\cP^{+}_{\ell}f\big)(s)-f(s)\big)&=\frac{1}{\ell}\Big(\bE\big(f\big(\tau^{+}_{\ell}(s)\big)\big)-f(s)\Big)=\frac{1}{\ell}\bigg(\!-\!\int_{s}^{\infty}g_{f}(t)\,\bP\big(\tau^{+}_{\ell}(s)\leq t)\,dt+\!\int_{s}^{\infty}g_{f}(t)\,dt\bigg)\\
&=\int_{s}^{\infty}g_{f}(t)\frac{1}{\ell}\,\bP\big(\tau^{+}_{\ell}(s)>t\big)\,dt.
\end{align*}
In light of Lemma \ref{lem:EstHitTimePlusMinus} and \eqref{eq:ABMHitTimePlusPhi}, there exists $b\in(0,\ell)$, such that
\begin{align*}
&\frac{1}{\ell}\,\bP\big(\tau^{+}_{\ell}(s)>t\big)\leq\frac{1}{\ell}\left(\Phi\bigg(\frac{\ell+\|v\|_{\infty}(t-s)}{\underline{\sigma}\sqrt{t-s}}\bigg)-e^{-2\|v\|_{\infty}\ell/\underline{\sigma}^{2}}\Phi\bigg(\!-\frac{\ell-\|v\|_{\infty}(t-s)}{\underline{\sigma}\sqrt{t-s}}\bigg)\right)\\
&\quad =\frac{2e^{-(b+\|v\|_{\infty}(t-s))^{2}/(2\underline{\sigma}^{2}(t-s))}}{\underline{\sigma}\sqrt{2\pi(t-s)}}\!+\!\frac{2\|v\|_{\infty}}{\underline{\sigma}^{2}}e^{-2\|v\|_{\infty}b/\underline{\sigma}^{2}}\Phi\bigg(\!\!-\!\frac{\ell\!-\!\|v\|_{\infty}(t\!-\!s)}{\underline{\sigma}\sqrt{t-s}}\!\bigg)\leq\frac{2}{\underline{\sigma}\!\sqrt{2\pi(t\!-\!s)}}\!+\!\frac{2\|v\|_{\infty}}{\underline{\sigma}^{2}}.
\end{align*}
Therefore, \eqref{eq:PointLimGammaf} follows immediately from Proposition \ref{prop:gammaPlusMinus} (ii), the exponential decay of $g_{f}$, and the dominated convergence, which completes the proof of the proposition.
\end{proof}

\section{An Existence Result for Noisy Wiener-Hopf Factorization}\label{sec:NoisyWHGammaPlusMinus}

In this section, we will establish an operator equation for $\Gamma^{\pm}$, which is crucial for the proof of {Theorem \ref{thm:Whvarphi}. Let $\sD((\Gamma^{+})^{2})$ (respectively, $\sD((\Gamma^{-})^{2})$) be the largest possible domain on which $\Gamma^{+}\circ\Gamma^{+}$ (respectively, $\Gamma^{-}\circ\Gamma^{-}$) is well-defined, namely,
\begin{align*}
\sD\big((\Gamma^{+})^{2}\big):=\big\{f\in\sD(\Gamma^{+}):\,\Gamma^{+}f\in\sD(\Gamma^{+})\big\}\quad\text{and}\quad\sD\big((\Gamma^{-})^{2}\big):=\big\{f\in\sD(\Gamma^{-}):\,\Gamma^{-}f\in\sD(\Gamma^{-})\big\}.
\end{align*}
It is well known that both $\sD((\Gamma^{+})^{2})$ and $\sD((\Gamma^{-})^{2})$ are dense in $C_{0}(\bR_{+})$ (cf. \cite[Chpater I, Theorem 2.7]{Pazy:1983}), and therefore $\sD((\Gamma^{+})^{2})$ is not trivially equal to $\{0\}$.

\begin{theorem}\label{thm:NoisyWHExistence}
Under Assumption \ref{assump:vsigma}, for any $f\in\sD((\Gamma^{\pm})^{2})$, $f$ is right-differentiable on $\bR_{+}$ and left-differentiable on $(0,\infty)$. Moreover, we have
\begin{align}\label{eq:NoisyWHPlus}
f'_{+}(s)\mp v(s)\big(\Gamma^{\pm}f\big)(s)+\frac{1}{2}\sigma^{2}(s)\big((\Gamma^{\pm})^{2}f\big)(s)&=0,\quad s\in\bR_{+},\\
\label{eq:NoisyWHMinus} f'_{-}(s)\mp v(s-)\big(\Gamma^{\pm}f\big)(s)+\frac{1}{2}\sigma^{2}(s-)\big((\Gamma^{\pm})^{2}f\big)(s)&=0,\quad s\in (0,\infty).
\end{align}
In particular, $f$ is differentiable on $D(v,\sigma):=\{s\in\bR_{+}:v\,\,\text{and}\,\,\sigma\,\,\text{are continuous at}\,\,s\}$, and
\begin{align}\label{eq:NoisyWH}
f'(s)\mp v(s)\big(\Gamma^{\pm}f\big)(s)+\frac{1}{2}\sigma^{2}(s)\big((\Gamma^{\pm})^{2}f\big)(s)&=0,\quad s\in D(v,\sigma).
\end{align}
\end{theorem}

\begin{remark}\label{rem:NoisyWH}
Theorem \ref{thm:NoisyWHExistence} can be regarded as an analogue of the NWHf that was studied in e.g. \cite{JiangPistorius:2008}, \cite{KennedyWilliams:1990}, and \cite{Rogers:1994}. In order to describe the NWHf consider the following Markov-modulated process on a probability space $(\Omega,\sF,\bP)$
\begin{align}\label{eq:AddFuntlFinMC}
\psi_{t}:=\int_{0}^{t}v\big(Y_{s}\big)\,ds+\int_{0}^{t}\sigma\big(Y_{s}\big)\,dW_{s},\quad t\in\bR_{+},
\end{align}
where $Y:=(Y_{t})_{t\in\bR_{+}}$ is a continuous-time time-homogeneous Markov chain,  independent of $W$, with finite state space $\mathbf{E}$ and (possibly sub-Markovian) generator matrix $\mathsf{\Lambda}$, and $v:\mathbf{E}\rightarrow\bR$ and $\sigma:\mathbf{E}\rightarrow(0,\infty)$ are deterministic functions. Define the passage times
\begin{align}\label{eq:zetaPlusMinus}
\zeta_{\ell}^{+}:=\inf\big\{t\in\bR_{+}:\psi_{t}>\ell\big\}\quad\text{and}\quad\zeta_{\ell}^{-}:=\inf\big\{t\in\bR_{+}:\psi_{t}< -\ell\big\},\quad\ell\in\bR_{+}.
\end{align}
It was shown in \cite{JiangPistorius:2008}, \cite{KennedyWilliams:1990}, and \cite{Rogers:1994}, respectively under different technical assumptions, that there exists a unique pair of $|\mathbf{E}|\times|\mathbf{E}|$ generator matrices $(\mQ^{+},\mQ^{-})$ such that
\begin{align}\label{eq:NoisyWHFinMC}
\mathsf{\Lambda}\mp\mV\mQ^{\pm}+\frac{1}{2}\mathsf\Sigma^{2}\big(\mQ^{\pm}\big)^{2}=0,
\end{align}
where $\mV:=\text{diag}\{v(i),i\in\mathbf{E}\}$ and $\mathsf\Sigma:=\text{diag}\{\sigma(i),i\in\mathbf{E}\}$. Moreover, $\mQ^{\pm}$ admits the following probabilistic interpretation
\begin{align*}
e^{\ell\mQ^{\pm}}(i,j)=\bP\Big(Y_{\zeta^{\pm}_{\ell}}=j\,\Big|\,Y_{0}=i\Big),\quad i,j\in\mathbf{E},\quad\ell\in\bR_{+}.
\end{align*}
That is, $\mQ^{\pm}$ is the generator matrix of the time-changed Markov chain $(Y_{\zeta^{\pm}_{\ell}})_{\ell\in\bR_{+}}$.

Theorem \ref{thm:NoisyWHExistence} generalizes the above result in the following manner. We first replace the time-homogenous finite-state Markov chain in \eqref{eq:AddFuntlFinMC} with the time-homogeneous deterministic Markov process $Y_{t}\equiv s+t$, $t\in\bR_{+}$, for some $s\in\bR_{+}$. Clearly, the generator of $Y$ is the first-order differential operator and its state space is $\bR_{+}$. We still define the passage times $\zeta_{\ell}^{\pm}$ as in \eqref{eq:zetaPlusMinus}. Then for each $s\in\bR_{+}$, the time-changed process $(Y_{\zeta_{\ell}^{\pm}})_{\ell\in\bR_{+}}$ is given by $Y_{\zeta_{\ell}^{\pm}}=s+\zeta_{\ell}^{\pm}$, $\ell\in\bR_{+}$, which has the same law as $\tau_{\ell}^{\pm}(s)$ (given as in \eqref{eq:tauPlusMinus}) under $\bP$. Therefore, the equation \eqref{eq:NoisyWH} is analogous to \eqref{eq:NoisyWHFinMC} with $\mathsf{\Lambda}$ replaced by the first-order differential operator, and $\cQ^{\pm}$ replaced by the generator of the time-changed process $(Y_{\zeta_{\ell}^{\pm}})_{\ell\in\bR_{+}}$, which coincides with the generator $\Gamma^{\pm}$ of $(\tau_{\ell}^{\pm})_{\ell\in\bR_{+}}$.
\end{remark}

\begin{proof}[Proof of Theorem \ref{thm:NoisyWHExistence}]
We will only provide the proof for the ``$+$" case, as the ``$-$" case can be proved in an analogous way. We fix $f\in\sD((\Gamma^{+})^{2})$ for the rest of the proof.

We start by showing that in order to prove the theorem it suffices to show that for any $\ell\in (0,\infty)$,
\begin{align}\label{eq:NoisyWHSuffCond}
\lim_{\delta\rightarrow 0+}\frac{1}{\delta}\Big(\big(\cP^{+}_{\ell}\!f\big)(s\!+\!\delta)\!-\!\big(\cP^{+}_{\ell}\!f\big)(s)\Big)\!=\!v(s)\big(\cP^{+}_{\ell}\Gamma^{+}\!f\big)(s)\!-\!\frac{1}{2}\sigma^{2}(s)\big(\cP^{+}_{\ell}(\Gamma^{+})^{2}f\big)(s),\quad s\in\bR_{+}.
\end{align}
To see this, we first note that for $f\in\sD((\Gamma^{+})^{2})$ we have $\Gamma^{+}f\in\sD(\Gamma^{+})\subset C_{0}(\bR_{+})$ and $(\Gamma^{+})^{2}f\in C_{0}(\bR_{+})$, and so, by the Feller property of $\cP^{+}$, we have $\cP^{+}_{\ell}\Gamma^{+}f,\ \cP^{+}_{\ell}(\Gamma^{+})^{2}f\in C_{0}(\bR_{+})$. This, together with Assumption \ref{assump:vsigma}, implies that the function on the right-hand side of \eqref{eq:NoisyWHSuffCond} is c\`{a}dl\`{a}g and bounded on $\bR_{+}$. Therefore, by Lemma \ref{lem:CadlagRightDerivImpAbsCont}, we obtain that
\begin{align}\label{eq:NoisyWHSuffCondInt}
\big(\cP^{+}_{\ell}f\big)(s)-\big(\cP^{+}_{\ell}f\big)(0)=\!\int_{0}^{s}\bigg(\!v(r)\big(\cP^{+}_{\ell}\Gamma^{+}f\big)(r)-\frac{1}{2}\sigma^{2}(r)\big(\cP^{+}_{\ell}(\Gamma^{+})^{2}f\big)(r)\bigg)dr,\quad s\in\bR_{+}.
\end{align}
By letting $\ell\rightarrow 0+$ in \eqref{eq:NoisyWHSuffCondInt} and using the strong continuity of $(\cP^{+}_{\ell})_{\ell\in\bR_{+}}$, we deduce that
\begin{align*}
f(s)-f(0)=\int_{0}^{s}\bigg(v(r)\big(\Gamma^{+}f\big)(r)-\frac{1}{2}\sigma^{2}(r)\big((\Gamma^{+})^{2}f\big)(r)\bigg)dr,\quad s\in\bR_{+}.
\end{align*}
The statement of theorem follows immediately from the above identity, Assumption \ref{assump:vsigma}, and a routine proof of elementary calculus.

It remains to prove \eqref{eq:NoisyWHSuffCond}. We will fix $\ell\in(0,\infty)$ and $s\in\bR_{+}$ for the rest of the proof. To begin with, for any $\delta>0$, by Lemma \ref{lem:StrongMarkovWtcM} and \eqref{eq:cPPlusMinusWt}, we have
\begin{align}
&\frac{1}{\delta}\Big(\!\big(\cP^{+}_{\ell}\!f\big)(s\!+\!\delta)\!-\!\big(\cP^{+}_{\ell}\!f\big)(s)\!\Big)=\frac{1}{\delta}\bigg(\!\wt{\bE}_{s+\delta,0}\Big(\!f\Big(Z^{1}_{\wt{\tau}^{+}_{\ell}}\Big)\!\Big)\!-\wt{\bE}_{s,0}\Big(\!\1_{\{\wt{\tau}^{+}_{\ell}\geq\delta\}}f\Big(Z^{1}_{\wt{\tau}^{+}_{\ell}}\Big)\!\Big)\!-\wt{\bE}_{s,0}\Big(\!\1_{\{\wt{\tau}^{+}_{\ell}<\delta\}}f\Big(Z^{1}_{\wt{\tau}^{+}_{\ell}}\Big)\!\Big)\!\bigg)\nonumber\\
&=\frac{1}{\delta}\bigg(\wt{\bE}_{s+\delta,0}\Big(f\Big(Z^{1}_{\wt{\tau}^{+}_{\ell}}\Big)\Big)-\wt{\bE}_{s,0}\Big(\1_{\{\wt{\tau}^{+}_{\ell}\geq\delta\}}\wt{\bE}_{Z^{1}_{\delta},Z^{2}_{\delta}}\Big(f\Big(Z^{1}_{\wt{\tau}^{+}_{\ell}}\Big)\Big)\Big)-\wt{\bE}_{s,0}\Big(\1_{\{\wt{\tau}^{+}_{\ell}<\delta\}}f\Big(Z^{1}_{\wt{\tau}^{+}_{\ell}}\Big)\Big)\bigg)\nonumber\\
\label{eq:DecompPWGencPPlus} &=\!\frac{1}{\delta}\!\left(\!\wt{\bE}_{s+\delta,0}\Big(\!f\Big(\!Z^{1}_{\wt{\tau}^{+}_{\ell}}\!\Big)\!\Big)\!-\!\wt{\bE}_{s,0}\bigg(\!\wt{\bE}_{s+\delta,Z^{2}_{\delta}}\Big(\!f\Big(\!Z^{1}_{\wt{\tau}^{+}_{\ell}}\!\Big)\!\Big)\!\bigg)\!\right)\!+\!\frac{1}{\delta}\,\wt{\bE}_{s,0}\bigg(\!\1_{\{\wt{\tau}^{+}_{\ell}<\delta\}}\!\bigg(\!\wt{\bE}_{s+\delta,Z^{2}_{\delta}}\Big(\!f\Big(\!Z^{1}_{\wt{\tau}^{+}_{\ell}}\!\Big)\!\Big)\!-\!f\Big(\!Z^{1}_{\wt{\tau}^{+}_{\ell}}\!\Big)\!\bigg)\!\bigg).
\end{align}
For the second term above, it follows from \eqref{eq:WtExptauExptau}, {Lemma \ref{lem:EstHitTimePlusMinus} and \eqref{eq:ABMHitTimePlusPhi}} that, as $\delta\rightarrow 0+$,
\begin{align}
&\frac{1}{\delta}\,\wt{\bE}_{s,0}\bigg(\1_{\{\wt{\tau}^{+}_{\ell}<\delta\}}\Big|\wt{\bE}_{s+\delta,Z^{2}_{\delta}}\Big(f\Big(Z^{1}_{\wt{\tau}^{+}_{\ell}}\Big)\Big)\!-\!f\Big(Z^{1}_{\wt{\tau}^{+}_{\ell}}\Big)\Big|\bigg)\leq\frac{2\|f\|_{\infty}}{\delta}\,\wt{\bP}_{s,0}\big(\wt{\tau}^{+}_{\ell}\!\leq\delta\big)=\frac{2\|f\|_{\infty}}{\delta}\,\bP\big(\tau_{\ell}^{+}(s)\leq\delta+s\big)\nonumber\\
\label{eq:LimPWGencPPlus2}&\quad\leq\frac{2\|f\|_{\infty}}{\delta}\left(1-\Phi\bigg(\frac{\ell-\|v\|_{\infty}\underline{\sigma}^{-2}\overline{\sigma}^{2}\delta}{\overline{\sigma}\sqrt{\delta}}\bigg)+e^{2\|v\|_{\infty}\ell/\underline{\sigma}^{2}}\Phi\bigg(\!-\frac{\ell+\|v\|_{\infty}\underline{\sigma}^{-2}\overline{\sigma}^{2}\delta}{\overline{\sigma}\sqrt{\delta}}\bigg)\right)\rightarrow 0.
\end{align}

In order to compute the limit of the first term in \eqref{eq:DecompPWGencPPlus}, as $\delta\rightarrow 0+$, ideally we wish to apply It\^{o}'s formula directly for the function
\begin{align*}
\psi(a):=\wt{\bE}_{s+\delta,a}\Big(f\Big(Z^{1}_{\wt{\tau}^{+}_{\ell}}\Big)\Big)=\wt{\bE}_{s+\delta,0}\Big(f\Big(Z^{1}_{\wt{\tau}^{+}_{\ell-a}}\Big)\Big)=\begin{cases} \big(\cP^{+}_{\ell-a}f\big)(s+\delta),& a\leq\ell, \\ f(s+\delta),& a>\ell,\end{cases}
\end{align*}
where we have used \eqref{eq:ShiftWttau} in the second equality, and \eqref{eq:cPPlusMinusWt} in the second equality. However, $\psi$ may not be differentiable at $a=\ell$, so some alternative approach has to be sought. To this end, we define an auxiliary function $h:\sZ\rightarrow\bR$ (recall $\sZ=\bR_{+}\times\bR$) as
\begin{align}\label{eq:Functh}
h(t,r):=\begin{cases} \big(\cP^{+}_{r}\Gamma^{+}f\big)(t),& (t,r)\in\bR_{+}^{2}, \\ 2\big(\Gamma^{+}f\big)(t)-\big(\cP^{+}_{-r}\Gamma^{+}f\big)(t),& (t,r)\in\bR_{+}\times(-\infty,0).
\end{cases}
\end{align}
Since $f\in\sD((\Gamma^{+})^{2})$, $\Gamma^{+}f\in\sD(\Gamma^{+})\subset C_{0}(\bR_{+})$, which, together with Proposition \ref{prop:FellerSemiGroupcPPlusMinus}, implies that $h$ is continuous on $\sZ$. The contraction property of $\cP^{+}_{\ell}$, for any $\ell\in\bR_{+}$, ensures that
\begin{align}\label{eq:BoundInftyNormh}
\|h\|_{\infty}\leq 3\big\|\Gamma^{+}f\big\|_{\infty}.
\end{align}
Moreover, by {\cite[Chapter 1, Proposotion 1.5 (b)]{EthierKurtz:2005}},
\begin{align}\label{eq:1stDerhr}
\frac{\partial}{\partial r}h(t,r)=\big(\cP^{+}_{|r|}(\Gamma^{+})^{2}f\big)(t),\quad (t,r)\in\sZ.
\end{align}
The choice of $f\in\sD((\Gamma^{+})^{2})$ (so that $(\Gamma^{+})^{2}f\in C_{0}(\bR_{+})$) together with Proposition \ref{prop:FellerSemiGroupcPPlusMinus} again implies that $\partial h/\partial r$ is continuous on $\sZ$. Next, we define $H:\sZ\rightarrow\bR$ as
\begin{align}\label{eq:FunctH}
H(t,r):=f(t)+\int_{0}^{r}h(t,y)\,dy,\quad (t,r)\in\sZ.
\end{align}
It follows immediately from \eqref{eq:BoundInftyNormh} that
\begin{align}\label{eq:BoundFunctH}
\big|H(t,r)\big|\leq\|f\|_{\infty}+3\big\|\Gamma^{+}f\big\|_{\infty}|r|,\quad (t,r)\in\sZ.
\end{align}
Note that for any $t\in\bR_{+}$, $H(t,\cdot)\in C^{2}(\bR)$, and by \eqref{eq:cPPlusMinusWt} and \eqref{eq:Functh} as well as \cite[Chapter 1, Proposition 1.5 (c)]{EthierKurtz:2005}, we have
\begin{align*}
H(t,r)=\big(\cP^{+}_{r}f\big)(t)=\wt{\bE}_{t,0}\Big(f\Big(Z^{1}_{\wt{\tau}^{+}_{r}}\Big)\Big),\quad(t,r)\in\bR_{+}^{2}.
\end{align*}
Together with \eqref{eq:ShiftWttau}, we obtain that
\begin{align*}
\wt{\bE}_{s+\delta,Z^{2}_{\delta}}\Big(f\Big(Z^{1}_{\wt{\tau}^{+}_{\ell}}\Big)\Big)&=\1_{\{Z^{2}_{\delta}\leq\ell\}}\,\wt{\bE}_{t,b}\Big(f\Big(Z^{1}_{\wt{\tau}^{+}_{\ell}}\Big)\Big)\Big|_{(t,b)=(s+\delta,Z^{2}_{\delta})}+\1_{\{Z^{2}_{\delta}>\ell\}}\,\wt{\bE}_{s+\delta,Z^{2}_{\delta}}\Big({f}\Big(Z^{1}_{\wt{\tau}^{+}_{\ell}}\Big)\Big)\\
&=\1_{\{Z^{2}_{\delta}\leq\ell\}}\,\wt{\bE}_{t,0}\Big(f\Big(Z^{1}_{\wt{\tau}^{+}_{\ell-b}}\Big)\Big)\Big|_{(t,b)=(s+\delta,Z^{2}_{\delta})}+\1_{\{Z^{2}_{\delta}>\ell\}}\,\wt{\bE}_{s+\delta,Z^{2}_{\delta}}\Big({f}\Big(Z^{1}_{\wt{\tau}^{+}_{\ell}}\Big)\Big)\\
&=\1_{\{Z^{2}_{\delta}\leq\ell\}}H\big(s+\delta,\ell-Z^{2}_{\delta}\big)+\1_{\{Z^{2}_{\delta}>\ell\}}\,\wt{\bE}_{s+\delta,Z^{2}_{\delta}}\Big(f\Big(Z^{1}_{\wt{\tau}^{+}_{\ell}}\Big)\Big).
\end{align*}
Hence, the first term in \eqref{eq:DecompPWGencPPlus} can be further decomposed as
\begin{align}
&\frac{1}{\delta}\left(\wt{\bE}_{s+\delta,0}\Big(f\Big(Z^{1}_{\wt{\tau}^{+}_{\ell}}\Big)\Big)-\wt{\bE}_{s,0}\bigg(\wt{\bE}_{s+\delta,Z^{2}_{\delta}}\Big(f\Big(Z^{1}_{\wt{\tau}^{+}_{\ell}}\Big)\Big)\bigg)\right)\nonumber\\
&=\frac{1}{\delta}\left(H(s+\delta,\ell)-\wt{\bE}_{s,0}\bigg(\1_{\{Z^{2}_{\delta}\leq\ell\}}H\big(s+\delta,\ell-Z^{2}_{\delta}\big)+\1_{\{Z^{2}_{\delta}>\ell\}}\,\wt{\bE}_{s+\delta,Z^{2}_{\delta}}\Big(f\Big(Z^{1}_{\wt{\tau}^{+}_{\ell}}\Big)\Big)\bigg)\right)\nonumber\\
\label{eq:DecompPWGencPPlus1} &=\frac{1}{\delta}\,\wt{\bE}_{s,0}\Big(\!H(s\!+\!\delta,\ell)\!-\!H\big(s\!+\!\delta,\ell\!-\!Z^{2}_{\delta}\big)\!\Big)\!+\!\frac{1}{\delta}\,\wt{\bE}_{s,0}\bigg(\!\1_{\{Z^{2}_{\delta}>\ell\}}\!\Big(\!H\big(s\!+\!\delta,\ell\!-\!Z^{2}_{\delta}\big)\!-\!\wt{\bE}_{s+\delta,Z^{2}_{\delta}}\Big(\!f\Big(\!Z^{1}_{\wt{\tau}^{+}_{\ell}}\!\Big)\!\Big)\!\Big)\!\bigg).\qquad
\end{align}

For the first term in \eqref{eq:DecompPWGencPPlus1}, by \eqref{eq:TranProbWhvarphi} and \eqref{eq:TranProbZ}, we have
\begin{align*}
\wt{\bE}_{s,0}\Big(H(s+\delta,\ell)-H\big(s+\delta,\ell-Z^{2}_{\delta}\big)\Big)&=\wh{\bE}_{s,0}\Big(H(s+\delta,\ell)-H\big(s+\delta,\ell-\wh{\varphi}_{s+\delta}\big)\Big)\\
&=\bE\Big(H(s+\delta,\ell)-H\big(s+\delta,\ell-\varphi_{s+\delta}(s)\big)\Big).
\end{align*}
Recalling $H(t,\cdot)\in C^{2}(\bR)$ for any $t\in\bR_{+}$, by \eqref{eq:varphi}, \eqref{eq:1stDerhr}, \eqref{eq:FunctH}, and It\^{o}'s formula, we deduce that
\begin{align*}
&\frac{1}{\delta}\,\wt{\bE}_{s,0}\Big(H(s+\delta,\ell)-H\big(s+\delta,\ell-Z^{2}_{\delta}\big)\Big)\\
&\quad =\bE\bigg(\frac{1}{\delta}\int_{s}^{s+\delta}h\big(s+\delta,\ell-\varphi_{t}(s)\big)v(t)\,dt-\frac{1}{2\delta}\int_{s}^{s+\delta}\Big(\cP^{+}_{|\ell-\varphi_{t}(s)|}(\Gamma^{+})^{2}f\Big)(s+\delta)\sigma^{2}(t)\,dt\bigg).
\end{align*}
Note that for $\bP$-a.e. $\omega\in\Omega$, $\varphi_{\cdot}(s)(\omega)$ is continuous on $[s,\infty)$, and so there exists $\delta_{0}=\delta_{0}(\omega)\in(0,1)$ such that $|\varphi_{t}(s)(\omega)|\leq\ell/2$ for all $t\in[s,s+\delta_{0}]$. Using the (joint) continuity of $h$ on $\sZ$ (in particular, the uniform continuity of $h$ on $[s,s+1]\times[\ell/2,3\ell/2]$, the continuity of sample paths of $\varphi(s)$, and the right-continuity of $v$, we obtain that, as $\delta\rightarrow 0+$,
\begin{align*}
&\bigg|\frac{1}{\delta}\int_{s}^{s+\delta}h\big(s+\delta,\ell-\varphi_{t}(s)(\omega)\big)v(t)\,dt-h(s,\ell)v(s)\bigg|\\
&\leq\!\bigg|\frac{1}{\delta}\!\int_{s}^{s+\delta}\!\!h\big(s,\ell\!-\!\varphi_{t}(s)(\omega)\big)v(t)dt\!-\!h(s,\ell)v(s)\bigg|\!+\!\!\sup_{\substack{(t,r),(t',r')\in[s,s+1]\times[\ell/2,3\ell/2] \\ |t-t'|\leq\delta,\,|r-r'|\leq\delta}}\!\big|h(t,r)\!-\!h(t',r')\big|\|v\|_{\infty}\!\rightarrow\!0.
\end{align*}
Similarly, noting that $(r,t)\mapsto(\cP^{+}_{r}(\Gamma^{+})^{2}f)(t)$ is jointly continuous on $\bR^{2}_{+}$ (since $(\cP_{\ell}^{+})_{\ell\in\bR_{+}}$ is strongly continuous and $f\in\sD((\Gamma^{+})^{2})$), we also have, for $\bP$-a.e. $\omega\in\Omega$,
\begin{align*}
\lim_{\delta\rightarrow 0+}\frac{1}{2\delta}\int_{s}^{s+\delta}\Big(\cP^{+}_{|\ell-\varphi_{t}(s)(\omega)|}(\Gamma^{+})^{2}f\Big)(s+\delta)\sigma^{2}(t)\,dt=\frac{1}{2}\big(\cP^{+}_{\ell}(\Gamma^{+})^{2}f\big)(s)\sigma^{2}(s).
\end{align*}
Therefore, by \eqref{eq:Functh}, \eqref{eq:BoundInftyNormh}, and the contraction property of $\cP^{+}_{\ell}$, for any $\ell\in\bR_{+}$, the dominated convergence theorem implies that
\begin{align}\label{eq:LimPWGencPPlus11}
\lim_{\delta\rightarrow 0+}\frac{1}{\delta}\,\wt{\bE}_{s,0}\Big(H(s+\delta,\ell)-H\big(s+\delta,\ell-Z^{2}_{\delta}\big)\Big)
=v(s)\big(\cP^{+}_{\ell}\Gamma^{+}\!f\big)(s)\!-\!\frac{1}{2}\sigma^{2}(s)\big(\cP^{+}_{\ell}(\Gamma^{+})^{2}f\big)(s).\quad
\end{align}

As for the second term in \eqref{eq:DecompPWGencPPlus1}, by \eqref{eq:TranProbWhvarphi}, \eqref{eq:TranProbZ}, and \eqref{eq:BoundFunctH}, we first have
\begin{align*}
&\left|\wt{\bE}_{s,0}\bigg(\1_{\{Z^{2}_{\delta}>\ell\}}\Big(H\big(s+\delta,\ell-Z^{2}_{\delta}\big)-\wt{\bE}_{s+\delta,Z^{2}_{\delta}}\Big(f\Big(Z^{1}_{\wt{\tau}^{+}_{\ell}}\Big)\Big)\Big)\bigg)\right|\\
&\quad\leq\wt{\bE}_{s,0}\bigg(\!\1_{\{Z^{2}_{\delta}>\ell\}}\Big(2\|f\|_{\infty}\!+3\big\|\Gamma^{+}\!f\big\|_{\infty}\big|\ell\!-\!Z^{2}_{\delta}\big|\Big)\!\bigg)=\wh{\bE}_{s,0}\bigg(\!\1_{\{\wh{\varphi}_{s+\delta}>\ell\}}\Big(2\|f\|_{\infty}\!+3\big\|\Gamma^{+}\!f\big\|_{\infty}\big|\ell\!-\!\wh{\varphi}_{s+\delta}\big|\Big)\!\bigg)\\
&\quad =2\|f\|_{\infty}\,\bP\big(\varphi_{s+\delta}(s)>\ell\big)+3\big\|\Gamma^{+}f\big\|_{\infty}\bE\Big(\1_{\{\varphi_{s+\delta}(s)>\ell\}}\big|\ell-\varphi_{s+\delta}(s)\big|\Big).
\end{align*}
Note that $\varphi_{s+\delta}(s)$ has a normal distribution $N(\int_{s}^{s+\delta}v(r)dr,\int_{s}^{s+\delta}\sigma^{2}(r)dr)$ under $\bP$. Denoting by $\Phi$ the $N(0,1)$ distribution function, we thus obtain that, as $\delta\rightarrow 0+$ (in particular, $\delta\in(0,\ell/\|v\|_{\infty})$),
\begin{align}
&\frac{1}{\delta}\left|\wt{\bE}_{s,0}\bigg(\1_{\{Z^{2}_{\delta}>\ell\}}\Big(H\big(s+\delta,\ell-Z^{2}_{\delta}\big)-\wt{\bE}_{s+\delta,Z^{2}_{\delta}}\Big(f\Big(Z^{1}_{\wt{\tau}^{+}_{\ell}}\Big)\Big)\Big)\bigg)\right|\nonumber\\
&\leq\frac{2\|f\|_{\infty}}{\delta}\left(1-\Phi\bigg(\frac{\ell-\|v\|_{\infty}\delta}{\overline{\sigma}\sqrt{\delta}}\bigg)\right)+\frac{3\|\Gamma^{+}f\|_{\infty}}{\delta}\int_{\ell}^{\infty}\frac{z-\ell}{\sqrt{2\pi\delta}\,\underline{\sigma}}\exp\left(-\frac{1}{2\overline{\sigma}^{2}\delta}\bigg(z-\int_{s}^{s+\delta}v(r)dr\bigg)^{2}\right)dz\nonumber\\
&\leq\frac{2\|f\|_{\infty}}{\delta}\!\left(1\!-\!\Phi\bigg(\frac{\ell\!-\!\|v\|_{\infty}\delta}{\overline{\sigma}\sqrt{\delta}}\bigg)\right)+\frac{3\overline{\sigma}\|\Gamma^{+}\!f\|_{\infty}}{\underline{\sigma}\delta}\!\int_{\ell}^{\infty}\!\frac{1}{\sqrt{2\pi\delta}\,\overline{\sigma}}\exp\bigg(\!\!-\!\frac{\big(z\!-\!\|v\|_{\infty}\delta\big)^{2}}{2\overline{\sigma}^{2}\delta}\bigg)\big(\|v\|_{\infty}\delta\!-\!\ell\big)dz\nonumber\\
&\quad +\frac{3\|\Gamma^{+}f\|_{\infty}}{\delta}\int_{\ell}^{\infty}\frac{1}{\sqrt{2\pi\delta}\,\underline{\sigma}}\exp\bigg(\!-\!\frac{\big(z-\|v\|_{\infty}\delta\big)^{2}}{2\overline{\sigma}^{2}\delta}\bigg)\big(z-\|v\|_{\infty}\delta\big)dz\nonumber\\
\label{eq:LimPWGencPPlus12} &\leq\bigg(\frac{2\|f\|_{\infty}}{\delta}\!+\!\frac{3\overline{\sigma}\|\Gamma^{+}\!f\|_{\infty}}{\underline{\sigma}\delta}\bigg)\!\left(\!1\!-\!\Phi\bigg(\frac{\ell\!-\!\|v\|_{\infty}\delta}{\overline{\sigma}\sqrt{\delta}}\bigg)\!\right)+\frac{3\|\Gamma^{+}\!f\|_{\infty}\overline{\sigma}^{2}}{\sqrt{2\pi\delta}\underline{\sigma}}\exp\!\bigg(\!\!-\!\frac{\big(\ell\!-\!\|v\|_{\infty}\delta\big)^{2}}{2\overline{\sigma}^{2}\delta}\bigg)\rightarrow 0.
\end{align}

Finally, by combining \eqref{eq:DecompPWGencPPlus1}, \eqref{eq:LimPWGencPPlus11}, and \eqref{eq:LimPWGencPPlus12}, we obtain that
\begin{align*}
\lim_{\delta\rightarrow 0+}\frac{1}{\delta}\!\left(\!\wt{\bE}_{s+\delta,0}\Big(f\Big(Z^{1}_{\wt{\tau}^{+}_{\ell}}\Big)\Big)\!-\wt{\bE}_{s,0}\bigg(\wt{\bE}_{s+\delta,Z^{2}_{\delta}}\Big(f\Big(Z^{1}_{\wt{\tau}^{+}_{\ell}}\Big)\Big)\!\bigg)\!\right)\!=v(s)\big(\cP^{+}_{\ell}\Gamma^{+}\!f\big)(s)-\frac{1}{2}\sigma^{2}(s)\big(\cP^{+}_{\ell}(\Gamma^{+})^{2}f\big)(s),
\end{align*}
which, together with \eqref{eq:LimPWGencPPlus2}, implies \eqref{eq:NoisyWHSuffCond}. The proof of the theorem is complete.
\end{proof}

\section{Proof of the Main Result}\label{sec:MainProof}

In this section, we will present the proof of our main result on the WHf for the time-inhomogeneous diffusion process $\varphi$. Towards this end we first provide,  in Section \ref{subsec:ProofPropGammaGen}, the proof of Proposition \ref{prop:GammaGen}. The proof of Theorem \ref{thm:Whvarphi} is then presented in Section \ref{subsec:ProofThmWHvarphi}, followed by the proof of Corollary \ref{cor:WHABM} that is given in Section \ref{subsec:ProofCorWHABM}.

\subsection{Proof of Proposition \ref{prop:GammaGen}}\label{subsec:ProofPropGammaGen}

We begin with the proof of Proposition \ref{prop:GammaGen} (i). Our proof is based on the version of Hille-Yosida theorem as stated in \cite[Theorem 1.30]{BottcherSchillingWang:2014}.

\medskip
\noindent
{\it Proof of Proposition \ref{prop:GammaGen} (i).} The proof is divided into the following two steps.

\smallskip
\noindent
\textbf{Step 1.} In this step we will establish the {\it positive maximum principle} for $\Gamma$. Given our setup, $\Gamma$ is said to satisfy the positive maximum principle if for any $f\in C_{e,\text{cdl}}^{\text{ac}}(\bR_{+})$ and $s_{0}\in\bR_{+}$ with $f(s_{0})=\sup_{s\in\bR_{+}}f(s)\geq 0$, we have $(\Gamma f)(s_{0})\leq 0$.

Throughout this step, we fix any $f\in C_{e,\text{cdl}}^{\text{ac}}(\bR_{+})$ and $s_{0}\in\bR_{+}$ such that $f(s_{0})=\sup_{s\in\bR_{+}}f(s)$. Since $g_{f}$ (recalling \eqref{eq:SpaceCacecdl}) vanishes at infinity with exponential rate, so does $f$. Hence, it is necessary to have $f(s_{0})\geq 0$. The proof is further divided into the following three steps.

\smallskip
\noindent
\textbf{Step 1.1.} Assume first that there exist $J\in\bN$ and $0\leq s_{0}<s_{1}<\cdots<s_{J}<\infty$ such that $\supp(f)\subset[0,s_{J}]$, that $g_{f}$ is non-positive on $[s_{j-1},s_{j})$ when $j$ is odd, and that $g_{f}$ is nonnegative on $[s_{j-1},s_{j})$ when $j$ is even (if $J\geq 2$). We will show that, for any $j=1,\ldots,J$,
\begin{align}\label{eq:GammasjNonPost}
\int_{s_{0}}^{s_{j}}g_{f}(r)\gamma(s_{0},r)\,dr\leq 0.
\end{align}
In particular, we have $(\Gamma f)(s_{0})=\int_{s_{0}}^{s_{J}}g_{f}(r)\gamma(s_{0},r)dr\leq 0$.

To begin with, when $j=1$, since $\gamma$ is nonnegative and $g_{f}$ is non-positive on $[s_{0},s_{1})$, clearly we have $\int_{s_{0}}^{s_{1}}g_{f}(r)\gamma(s_{0},r)\,dr\leq 0$. Moreover, when $J\geq 2$ and for $j=2$, since $\gamma(s_{0},\cdot)$ is nonnegative and non-increasing on $(s_{0},\infty)$, $g_{f}$ is non-positive (respectively, nonnegative) on $[s_{0},s_{1})$ (respectively, $[s_{1},s_{2})$), and since $s_{0}$ is a maximum point of $f$, we deduce that
\begin{align}\label{eq:Gammas2NonPost}
\int_{s_{0}}^{s_{2}}g_{f}(r)\gamma(s_{0},r)\,dr 
\leq\gamma(s_{0},s_{1})\int_{s_{0}}^{s_{2}}g_{f}(r)\,dr=\gamma(s_{0},s_{1})\big(f(s_{2})-f(s_{0})\big)\leq 0.
\end{align}

To proceed with the proof of \eqref{eq:GammasjNonPost} for $j=3,\ldots,J$ when $J\geq 3$, we will first prove by induction that, for any $j=2,\ldots,J$,\footnote{Here $\lfloor x\rfloor$ denotes the greatest integer less than or equal to $x$.}
\begin{equation}\label{eq:GammasjBoundgammasj}
\int_{s_{0}}^{s_{j}}g_{f}(r)\gamma(s_{0},r)\,dr\leq\gamma\Big(s_{0},s_{2\lfloor (j-1)/2\rfloor+1}\Big)\int_{{s_0}}^{s_{j}}g_{f}(r)\,dr,
\end{equation}
The case when $j=2$ has been verified in the first inequality of \eqref{eq:Gammas2NonPost}. Now assume that \eqref{eq:GammasjBoundgammasj} holds for $j=2,\ldots,n$ for some $n\in\bN$ with $n<J$. If $n$ is odd so that $g_{f}$ is nonnegative on $[s_{n},s_{n+1})$, since $\gamma(s_{0},\cdot)$ is nonnegative and non-increasing on $(s_{0},\infty)$, by the induction hypothesis we have
\begin{align*}
\int_{s_{0}}^{s_{n+1}}\!g_{f}(r)\gamma(s_{0},r)\,dr\leq\gamma(s_{0},s_{n})\!\int_{0}^{s_{n}}g_{f}(r)\,dr+\gamma(s_{0},s_{n})\!\int_{s_{n}}^{s_{n+1}}\!g_{f}(r)\,dr=\gamma(s_{0},s_{n})\!\int_{s_{0}}^{s_{n+1}}\!g_{f}(r)\,dr.
\end{align*}
Similarly, if $n$ is even so that $g_{f}$ is non-positive on $[s_{n},s_{n+1})$, we have
\begin{align*}
\int_{s_{0}}^{s_{n+1}}\!\!\!g_{f}(r)\gamma(s_{0},r)dr\leq\gamma(s_{0},s_{n-1})\!\!\int_{0}^{s_{n}}\!\!g_{f}(r)dr\!+\!\gamma(s_{0},s_{n+1})\!\!\int_{s_{n}}^{s_{n+1}}\!\!g_{f}(r)dr\leq\gamma(s_{0},s_{n+1})\!\int_{0}^{s_{n+1}}\!\!g_{f}(r)dr.
\end{align*}
The proof of \eqref{eq:GammasjBoundgammasj} for any $j=2,\ldots,J$ is complete by induction.

Returning to the proof of \eqref{eq:GammasjNonPost}, for any $j=3,\ldots,J$, since $\gamma$ is nonnegative and $s_{0}$ is a maximum point of $f$, by \eqref{eq:GammasjBoundgammasj} we obtain that
\begin{align*}
\int_{s_{0}}^{s_{j}}g_{f}(r)\gamma(s_{0},r)\,dr\leq\gamma\Big(s_{0},s_{2\lfloor (j-1)/2\rfloor+1}\Big)\int_{0}^{s_{j}}g_{f}(r)\,dr=\gamma\Big(s_{0},s_{2\lfloor (j-1)/2\rfloor+1}\Big)\big(f(s_{j})-f(s_{0})\big)\leq 0,
\end{align*}
which completes the proof of \eqref{eq:GammasjNonPost} for any $j=1,\ldots,J$.

We conclude this step by noting that the arguments above do not depend on the values of $f$ on $[0,s_{0}]$ nor on the c\`{a}dl\`{a}g property of $g_{f}$.

\smallskip
\noindent
\textbf{Step 1.2.} Next, we assume that $f\in C_{e,\text{cdl}}^{\text{ac}}(\bR_{+})$ has a compact support, i.e., there exists $T\in(s_{0},\infty)$ such that $\supp(f)\subset[0,T]$. We claim that there exists a sequence of Borel measurable functions $(h_{n})_{n\in\bN}$ on $[s_{0},T]$ which satisfy the following properties:
\begin{itemize}
\item [(a)] $\|h_{n}\|_{\infty}\leq\|g_{f}\|_{\infty}$, for any $n\in\bN$;
\item [(b)] there exists a subsequence $(h_{n_{k}})_{k\in\bN}$ of $(h_{n})_{n\in\bN}$ which converges to $g_{f}|_{[s_{0},T]}$ $\text{Leb}$\,-a.e., where $\text{Leb}$ denotes the Lebesgue measure on $(\bR,\cB(\bR))$;
\item [(c)] for each $n\in\bN$, $H_{n}(s_{0})=\sup_{r\in[s_{0},T]}H_{n}(r)$, where
    \begin{align}\label{eq:FunctHn}
    H_{n}(t):= -\int_{t}^{T}h_{n}(r)dr,\quad t\in[s_{0},T];
    \end{align}
\item [(d)] for each $n\in\bN$, there are $J_{n}\in\bN$ and $s_{0}<s^{n}_{1}<\cdots<s^{n}_{J_{n}}=T$, such that $h_{n}$ is non-positive on $[s^{n}_{j-1},s^{n}_{j})$ when $j$ is odd, and non-negative when $j$ is even (if $J_{n}\geq 2$).
\end{itemize}
Once such sequence is constructed, by applying the result of Step 1.1 to each $h_{n}$, we obtain that
\begin{align*}
\int_{s_{0}}^{T}h_{n}(r)\gamma(s_{0},r)\,dr\leq 0,\quad n\in\bN.
\end{align*}
We defer the detailed construction of $(h_{n})_{n\in\bN}$ in Appendix \ref{subsec:ProofPropGammaGenStep12}. Using properties (a) and (b) above as well as the boundedness of $g_{f}$, we deduce from dominated convergence that
\begin{align}\label{eq:PMPGammafcompact}
\big(\Gamma f\big)(s_{0})=\int_{s_{0}}^{T}g_{f}(r)\gamma(s_{0},r)\,dr=\lim_{k\rightarrow\infty}\int_{s_{0}}^{T}h_{n_{k}}(r)\gamma(s_{0},r)\,dr\leq 0.
\end{align}

\noindent
\textbf{Step 1.3.} Finally, we consider any arbitrary $f\in C_{e,\text{cdl}}^{\text{ac}}(\bR_{+})$. For any $T\in(s_{0},\infty)$, let $f_{T}(t):=\1_{[0,T]}(t)(f(t)-f(T))$. Clearly, $f_{T}\in C_{e,\text{cdl}}^{\text{ac}}(\bR_{+})$ with
\begin{align*}
\supp(f_{T})\subset[0,T];\quad f_{T}(s_{0})=\sup_{t\in[s_{0},T]}f_{T}(t),\quad f_{T}(t)= -\int_{t}^{\infty}g_{f}(r)\1_{[0,T)}(r)\,dr,\,\,\,\,t\in\bR_{+},
\end{align*}
where $g_{f}\1_{[0,T)}$ is c\`{a}dl\`{a}g on $\bR_{+}$. By \eqref{eq:Gamma} and \eqref{eq:PMPGammafcompact} we obtain that $\int_{s_{0}}^{\infty}g_{f}(r)\1_{[0,T)}(r)\gamma(s_{0},r)dr=(\Gamma f_{T})(s_{0})\leq 0$, and thus
\begin{align*}
\big(\Gamma f\big)(s_{0})\!=\!\!\int_{s_{0}}^{\infty}\!g_{f}(r)\gamma(s_{0},r)dr\!=\!\!\int_{s_{0}}^{T}\!g_{f}(r)\gamma(s_{0},r)dr+\!\int_{T}^{\infty}\!g_{f}(r)\gamma(s_{0},r)dr\!\leq\!\!\int_{T}^{\infty}\!g_{f}(r)\gamma(s_{0},r)dr\rightarrow 0,
\end{align*}
as $T\rightarrow\infty$, which completes the proof of the positive maximum principle for $f$.

\medskip
\noindent
\textbf{Step 2.} We now complete the proof of Proposition \ref{prop:GammaGen} (i). In view of Lemma \ref{lem:Gamma}, we have $\Gamma f\in C_{0}({\bR}_{+})$ for all $f\in C_{e,\text{cdl}}^{\text{ac}}(\bR_{+})$, which is dense in $C_{0}(\bR_{+})$. Together with the positive maximum principle for $\Gamma$ verified in Step 1 above as well as Assumption \ref{assump:RangeDens}, the statement of Proposition \ref{prop:GammaGen} (i) is a direct consequence of the Hille-Yosida-Ray theorem (cf. \cite[Theorem 1.30]{BottcherSchillingWang:2014}).\hfill $\Box$

\medskip
Let $(\cP_{\ell})_{\ell\in\bR_{+}}$ be the Feller semigroup established in Proposition \ref{prop:GammaGen} (i). In view of Riesz representation theorem (cf. \cite[Theorem 6.19]{Rudin:1986}), for any $\ell,s\in\bR_{+}$, there exists a measure $\mu_{\ell,s}$ on $(\bR_{+},\cB(\bR_{+}))$ such that
\begin{align}\label{eq:RieszRepcPells}
\big(\cP_{\ell}f\big)(s)=\int_{0}^{\infty}f(t)\,\mu_{\ell,s}(dt),\quad f\in C_{0}(\bR_{+}).
\end{align}
Since $\cP_{\ell}$ is a contraction operator the measure $\mu_{\ell,s}$  is a sub-probability. The following lemma is crucial for the proof of Proposition \ref{prop:GammaGen} (ii).

\begin{lemma}\label{lem:EstcPell}
Under Assumptions \ref{assump:vsigma}  and \ref{assump:RangeDens}, for any $\varepsilon>0$, there is a constant $c=c(\varepsilon,\|v\|_{\infty},\overline{\sigma},\underline{\sigma})\in(0,\infty)$, such that for any $T\in(0,\infty)$ and $\ell\in\bR_{+}$,
\begin{align*}
\big\|\cP_{\ell}\1_{[0,T]}\big\|_{\infty}\leq 2e^{-(c/(2T\wedge 1))\varepsilon\ell}.
\end{align*}
\end{lemma}

\begin{proof}
In view of Proposition \ref{prop:gammaPlusMinus} (iii), for any $\varepsilon>0$, there exists $c\in(0,\infty)$, depending only on $\|v\|_{\infty}$, $\overline{\sigma}$, and $\underline{\sigma}$, such that for any $s\in\bR_{+}$ and $r\in[0,c]$, $\gamma(s,s+r)\geq \varepsilon$. Assuming first $T\geq c/2$, let
\begin{align}\label{eq:Functf2T}
f_{2T}(t):=(2T-t)\1_{[0,2T]}(t)= -\int_{t}^{\infty}\big(-\1_{[0,2T)}(r)\big)\,dr.\quad t\in\bR_{+}.
\end{align}
Clearly, $f_{2T}\in C_{e,\text{cdl}}^{\text{ac}}(\bR_{+})$. By \eqref{eq:Gamma} and \eqref{eq:Functf2T} we have, for any $t\in\bR_{+}$,
\begin{align*}
-\big(\Gamma f_{2T}\big)(t)&=\int_{t}^{\infty}\1_{[0,2T)}(r)\gamma(t,r)\,dr 
\geq\1_{[0,2T-c)}(t)\int_{t}^{t+c}\gamma(t,r)\,dr+\1_{[2T-c,2T)}(t)\int_{t}^{2T}\gamma(t,r)\,dr\\
&\geq c\varepsilon\1_{[0,2T-c)}(t)+(2T-t)\varepsilon\1_{[2T-c,2T)}(t)\geq\frac{2T-t}{2T}\,c\varepsilon\1_{[0,2T)}(t)=\frac{c\varepsilon}{2T}f_{2T}(t).
\end{align*}
Together with \cite[Chapter 1, Proposition 1.5 (b)]{EthierKurtz:2005} and the positivity of $(\cP_{\ell})_{\ell\in\bR_{+}}$ (recalling from Proposition \ref{prop:GammaGen} (i) that $(\cP_{\ell})_{\ell\in\bR_{+}}$ is a Feller semigroup), we obtain that, for any $\ell\in\bR_{+}$ and $t\in\bR_{+}$,
\begin{align*}
\frac{\partial}{\partial\ell}\big(\cP_{\ell}f_{2T}\big)(t)=\big(\cP_{\ell}\Gamma f_{2T}\big)(t)\leq -\frac{c\varepsilon}{2T}\big(\cP_{\ell}f_{2T}\big)(t).
\end{align*}
Since $\1_{[0,T]}\leq f_{2T}/T$, the positivity of $(\cP_{\ell})_{\ell\in\bR_{+}}$ with Gr\"{o}nwall's inequality implies that
\begin{align}\label{eq:InfNormcPellInd0TLarge}
\big\|\cP_{\ell}\1_{[0,T]}\big\|_{\infty}\leq\frac{1}{T}\big\|\cP_{\ell}f_{2T}\big\|_{\infty}\leq\frac{e^{-c\varepsilon\ell/(2T)}}{T}\big\|f_{T}\big\|_{\infty}=2e^{-c\varepsilon\ell/(2T)},\quad T\geq\frac{c}{2}.
\end{align}
Finally, when $T\in(0,c/2)$, it follows from the positivity property of $(\cP_{\ell})_{\ell\in\bR_{+}}$ and \eqref{eq:InfNormcPellInd0TLarge} that
\begin{align*}
\big\|\cP_{\ell}\1_{[0,T]}\big\|_{\infty}\leq\big\|\cP_{\ell}\1_{[0,c/2]}\big\|_{\infty}\leq 2e^{-\varepsilon\ell},
\end{align*}
which completes the proof of the lemma.
\end{proof}

\begin{proof}[Proof of Proposition \ref{prop:GammaGen} (ii)]
Let $f\in C_{e}(\bR_{+})$, i.e., there exist $K,\kappa\in(0,\infty)$, such that $|f(t)|\leq Ke^{-\kappa t}$ for all $t\in\bR_{+}$. By \eqref{eq:RieszRepcPells}, Lemma \ref{lem:EstcPell} (with $\varepsilon=1$), and Fubini's theorem, there exists $c_{1}=c_{1}(\|v\|_{\infty},\overline{\sigma},\underline{\sigma})\in(0,\infty)$, such that for any $\ell,s\in\bR_{+}$,
\begin{align*}
&\big|\big(\cP_{\ell}f\big)(s)\big|\leq\int_{0}^{\infty}|f(t)|\,\mu_{\ell,s}(dt)\leq K\int_{0}^{\infty}e^{-\kappa t}\,\mu_{\ell,s}(dt)=K\int_{0}^{\infty}\!\bigg(\int_{t}^{\infty}\kappa e^{-\kappa r}\,dr\bigg)\mu_{\ell,s}(dt)\\
&\quad =K\kappa\!\int_{0}^{\infty}\!e^{-\kappa r}\bigg(\!\int_{0}^{r}\mu_{\ell,s}(dt)\!\bigg)dr\leq K\kappa\!\int_{0}^{\infty}\!e^{-\kappa r}\bigg(\!\int_{0}^{\infty}\!\wt{f}_{r}(t)\,\mu_{\ell,s}(dt)\!\bigg)dr=K\kappa\!\int_{0}^{\infty}\!e^{-\kappa r}\big(\cP_{\ell}\wt{f}_{r}\big)(s)\,dr \\
&\quad\leq K\kappa\int_{0}^{\infty}e^{-\kappa r}\big(\cP_{\ell}\1_{[0,2r]}\big)(s)\,dr\leq 2K\kappa\int_{0}^{\infty}e^{-\kappa r-(c_{1}/(4r\wedge 1))\ell}\,dr,
\end{align*}
where $\wt{f}_{r}$ is some function in $C_{0}(\bR_{+})$ with $\1_{[0,r]}\leq\wt{f}_{r}\leq\1_{[0,2r]}$. Hence, by Fubini's theorem again, for any $L\geq c_{1}^{2}$, we obtain that
\begin{align*}
&\bigg\|\int_{L}^{\infty}\cP_{\ell}f\,d\ell\,\bigg\|_{\infty}\leq\int_{L}^{\infty}2K\kappa\bigg(\int_{0}^{\infty}\!e^{-\kappa r-(c_{1}/(4r\wedge 1))\ell}\,dr\bigg)d\ell\leq\frac{2K\kappa}{c_{1}}\!\int_{0}^{\infty}\!\!\big((4r)\vee c_{1}\big)e^{-\kappa r-(c_{1}/(4r\wedge 1))L}\,dr\\
&\quad\leq\frac{2K\kappa}{c_{1}}\bigg(\sqrt{L}\,e^{-c_{1}\sqrt{L}}\!\int_{0}^{\sqrt{L}/4}\!\!e^{-\kappa r}\,dr\!+\!4\!\int_{\sqrt{L}/4}^{\infty}\!re^{-\kappa r}\,dr\bigg)\!\leq\!\frac{2K\kappa}{c_{1}}\!\left(\frac{L}{4}\,e^{-c_{1}\sqrt{L}}\!+\!\bigg(\!\frac{\sqrt{L}}{\kappa}\!+\!\frac{4}{\kappa^{2}}\!\bigg)e^{-\kappa\sqrt{L}/4}\right),
\end{align*}
which shows the convergence of $\int_{0}^{\infty}\cP_{\ell}f\,d\ell$ in $B_{b}(\bR_{+})$. Moreover, by Proposition \ref{prop:GammaGen} and \cite[Chapter 1, Proposition 1.5 (a)]{EthierKurtz:2005}, for any $L\in(0,\infty)$, we have $\int_{0}^{L}\cP_{\ell}f\,d\ell\in\sD(\overline\Gamma)$ and
\begin{align*}
\cP_{L}f-f=\overline{\Gamma}\int_{0}^{L}\cP_{\ell}f\,d\ell.
\end{align*}
Hence, by \eqref{eq:RieszRepcPells} and Lemma \ref{lem:EstcPell}, we obtain that, for any $T\in(0,\infty)$,
\begin{align*}
&\lim_{L\rightarrow\infty}\bigg\|\overline{\Gamma}\int_{0}^{L}\cP_{\ell}f\,d\ell-(-f)\bigg\|_{\infty}=\lim_{L\rightarrow\infty}\big\|\cP_{L}f\big\|_{\infty}\leq\lim_{L\rightarrow\infty}\sup_{s\in\bR_{+}}\bigg|\int_{0}^{T}f(t)\,\mu_{L,s}(dt)\bigg|+\sup_{t\in[T,\infty)}|f(t)|\\
&\quad\leq 2\|f\|_{\infty}\lim_{L\rightarrow\infty}e^{-(c/(2T)\wedge 1)\varepsilon L}+\sup_{t\in[T,\infty)}|f(t)|=\sup_{t\in[T,\infty)}|f(t)|.
\end{align*}
Since $T\in(0,\infty)$ is arbitrary and $f\in C_{e}(\bR_{+})$, by taking $T\rightarrow\infty$ on the right-hand side of the last equality above, we deduce that $\overline{\Gamma}\int_{0}^{L}\cP_{\ell}f\,d\ell$ converges to $-f$ in $B_{b}(\bR_{+})$, as $L\rightarrow\infty$. Finally, since $\overline{\Gamma}$ is a closed operator, we conclude that $\int_{\bR_{+}}\cP_{\ell}f\,d\ell\in\sD(\overline{\Gamma})$ and that $\overline{\Gamma}\int_{\bR_{+}}\cP_{\ell}f\,d\ell= -f$, which completes the proof of the proposition.
\end{proof}

\subsection{Proof of Theorem \ref{thm:Whvarphi}}\label{subsec:ProofThmWHvarphi}

We are now ready to present the proof of our main result Theorem \ref{thm:Whvarphi}. We will fix any $(s,a)\in\sZ$ throughout the proof, which will proceeds in the following two steps.

\medskip
\noindent
\textbf{Step 1.} In this step, we will prove \eqref{eq:WHvarphi} for any $u\in C(\bR)$, $h\in C_{e}(\bR_{+})$, and any $\bF$-stopping time $\tau$, under the following additional assumptions
\begin{itemize}
\item [(a)] $u\in C_{c}(\bR)$ with $\supp(u)\subset[-M,M]$, for some $M\in(0,\infty)$;
\item [(b)] $h\in C_{e}(\bR_{+})$ is of the form $h=\Gamma f$, for some $f\in C_{e,\text{cdl}}^{\text{ac}}(\bR_{+})$;
\item [(c)] there exists $T\in(0,\infty)$, such that $\tau\leq T$ $\bP$-a.s..
\end{itemize}
With the above choices of $u$ and $h$, \eqref{eq:FinExpIntuh} is clearly satisfied. Also, since $f\in C_{e,\text{cdl}}^{\text{ac}}(\bR_{+})\subset C_{e}(\bR_{+})$, we obtain from Lemma \ref{lem:Gamma} that $\Gamma f\in C_{e}(\bR_{+})$, which, together with Proposition \ref{prop:GammaGen} (i) and \cite[Chapter 1, Proposition 1.5 (a) \& (c)]{EthierKurtz:2005}, implies that, for any $L\in(0,\infty)$,
\begin{align*}
\int_{0}^{L}\cP_{\ell}\,\Gamma f\,d\ell=\int_{0}^{L}\cP_{\ell}\,\overline{\Gamma}f\,dy=\overline{\Gamma}\int_{0}^{L}\cP_{\ell}f\,d\ell.
\end{align*}
By Proposition \ref{prop:GammaGen} (ii) and the closedness of $\overline{\Gamma}$, we obtain that
\begin{align*}
\int_{0}^{\infty}\cP_{\ell}\,\Gamma f\,d\ell=\lim_{L\rightarrow\infty}\int_{0}^{L}\cP_{\ell}\,\Gamma f\,d\ell={\lim_{L\rightarrow\infty}\overline{\Gamma}\int_{0}^{L}\cP_{\ell}f\,d\ell}=\overline{\Gamma}\int_{0}^{\infty}\cP_{\ell}f\,d\ell= -f.
\end{align*}
Hence, under the additional assumptions (a)$-$(c) above, \eqref{eq:WHvarphi} can be written as
\begin{align}
&\bE\bigg(\int_{s}^{\tau}u\big(\varphi_{t}(s,a)\big)\big(\Gamma f\big)(t)\sigma^{2}(t)\,dt\bigg)=-2\int_{0}^{\infty}u(a+\ell)\big(\cP^{+}_{\ell}f\big)(s)\,d\ell-2\int_{0}^{\infty}u(a-\ell)\big(\cP^{-}_{\ell}f\big)(s)\,d\ell\nonumber\\
\label{eq:WHvarphiCompuGammafBoundtau} &\qquad\qquad\qquad +2\,\bE\bigg(\!\int_{0}^{\infty}\!u\big(\varphi_{\tau}(s,a)\!+\!\ell\big)\big(\cP^{+}_{\ell}\!f\big)(\tau)\,d\ell+\!\int_{0}^{\infty}\!u\big(\varphi_{\tau}(s,a)\!-\!\ell\big)\big(\cP^{-}_{\ell}\!f\big)(\tau)\,d\ell\bigg).
\end{align}

To proceed the proof of \eqref{eq:WHvarphiCompuGammafBoundtau}, we first introduce some auxiliary functions. To begin with, let
\begin{align}\label{eq:FunctF}
F(t,x):=2\int_{0}^{\infty}u(x+\ell)\big(\cP^{+}_{\ell}f\big)(t)\,d\ell+2\int_{0}^{\infty}u(x-\ell)\big(\cP^{-}_{\ell}f\big)(t)\,d\ell,\quad (t,x)\in\sZ.
\end{align}
Also, for any $\varepsilon>0$, we define
\begin{align}\label{eq:Functfepspm}
f^{+}_{\varepsilon}(t):=\frac{1}{\varepsilon}\int_{0}^{\varepsilon}\big(\cP^{+}_{\ell}f\big)(t)\,d\ell,\quad f^{-}_{\varepsilon}(t):=\frac{1}{\varepsilon}\int_{0}^{\varepsilon}\big(\cP^{-}_{\ell}f\big)(t)\,d\ell,\quad t\in\bR_{+},
\end{align}
and
\begin{align}\label{eq:FunctFeps}
F_{\varepsilon}(t,x):=2\int_{0}^{\infty}u(x+\ell)\big(\cP^{+}_{\ell}f^{+}_{\varepsilon}\big)(t)\,d\ell+2\int_{0}^{\infty}u(x-\ell)\big(\cP^{-}_{\ell}f^{-}_{\varepsilon}\big)(t)\,d\ell,\quad
(t,x)\in\sZ.
\end{align}
Clearly, $\|f^{\pm}_{\varepsilon}\|_{\infty}\leq\|f\|_{\infty}$, and so by the contraction property of $(\cP_{\ell}^{\pm})_{\ell\in\bR_{+}}$ (see Proposition \ref{prop:FellerSemiGroupcPPlusMinus}),
\begin{align}
\big\|F_{\varepsilon}\big\|_{\infty}\leq 4M\|u\|_{\infty}\big\|\cP^{+}_{\ell}f^{+}_{\varepsilon}\big\|_{\infty}+4M\|u\|_{\infty}\big\|\cP^{-}_{\ell}f^{-}_{\varepsilon}\big\|_{\infty}\leq 8M\|u\|_{\infty}\|f\|_{\infty}.
\end{align}
Moreover, from the strong continuity of $(\cP_{\ell}^{\pm})_{\ell\in\bR_{+}}$ (see Proposition \ref{prop:FellerSemiGroupcPPlusMinus}), we have
\begin{align}\label{eq:ConvInfNormDifffpmepsf}
\big\|f^{\pm}_{\varepsilon}-f\big\|_{\infty}\leq\frac{1}{\varepsilon}\int_{0}^{\varepsilon}\big\|\cP^{\pm}_{\ell}f-f\big\|_{\infty}d\ell\rightarrow 0,\quad\text{as }\,\varepsilon\rightarrow 0+,
\end{align}
and hence
\begin{align}
\big\|F_{\varepsilon}-F\big\|_{\infty}&\leq 2\sup_{x\in\bR}\int_{0}^{\infty}\big|u(x+\ell)\big|\big\|\cP^{+}_{\ell}\big(f^{+}_{\varepsilon}-f\big)\big\|_{\infty}\,d\ell+2\sup_{x\in\bR}\int_{0}^{\infty}\big|u(x-\ell)\big|\big\|\cP^{-}_{\ell}\big(f^{-}_{\varepsilon}-f\big)\big\|_{\infty}\,d\ell\nonumber\\
\label{eq:ConvInfNormDiffFepsF} &\leq 4M\|u\|_{\infty}\Big(\big\|f^{+}_{\varepsilon}-f\big\|_{\infty}+\big\|f^{-}_{\varepsilon}-f\big\|_{\infty}\Big)\rightarrow 0,\quad\text{as }\,\varepsilon\rightarrow 0+.
\end{align}
In addition, since $f\in C_{e,\text{cdl}}^{\text{ac}}(\bR_{+})\subset\sD(\Gamma^{\pm})$, by \eqref{eq:Functfepspm} and \cite[Chapter 1, Proposition 1.5 (a) \& (c)]{EthierKurtz:2005}, we have $f^{\pm}_{\varepsilon}\in\sD(\Gamma^{\pm})$ and
\begin{align}\label{eq:FunctfepspmDomGammapm2}
\Gamma^{\pm}f^{\pm}_{\varepsilon}=\frac{1}{\varepsilon}\,\Gamma^{\pm}\!\!\int_{0}^{\varepsilon}\cP^{\pm}_{\ell}f\,d\ell=\frac{1}{\varepsilon}\int_{0}^{\varepsilon}\cP^{\pm}_{\ell}\Gamma^{\pm}f\,d\ell\in\sD(\Gamma^{\pm}),
\end{align}
i.e., $f^{\pm}_{\varepsilon}\in\sD((\Gamma^{\pm})^{2})$. Hence, with similar reasoning as in \eqref{eq:ConvInfNormDifffpmepsf}, we obtain that, as $\varepsilon\rightarrow 0+$,
\begin{align}\label{eq:ConvInfNormGammapmfpmepsGammapmf}
\big\|\Gamma^{\pm}f^{\pm}_{\varepsilon}-\Gamma^{\pm}f\big\|_{\infty}=\bigg\|\frac{1}{\varepsilon}\int_{0}^{\varepsilon}\cP^{\pm}_{\ell}\Gamma^{\pm}f\,d\ell-\Gamma^{\pm}f\bigg\|_{\infty}\leq\frac{1}{\varepsilon}\int_{0}^{\varepsilon}\big\|\cP^{\pm}_{\ell}\Gamma^{\pm}f-\Gamma^{\pm}f\big\|_{\infty}d\ell\rightarrow 0.
\end{align}

In addition to the assumptions (a)$-$(c) above, we first provide the proof of \eqref{eq:WHvarphiCompuGammafBoundtau} when $u\in C_{c}^{1}(\bR)$. Such choice of $u$, together with \eqref{eq:FunctFeps}, the boundedness of $\cP^{\pm}_{\ell}f^{\pm}_{\varepsilon}$, and dominated convergence, ensures that $F_{\varepsilon}(t,\cdot)$ is differentiable on $\bR$, for every $t\in\bR_{+}$, and that
\begin{align}\label{eq:DerFunctFepsxIni}
\frac{\partial}{\partial x}F_{\varepsilon}(t,x)=2\int_{0}^{\infty}u'(x+\ell)\big(\cP^{+}_{\ell}f^{+}_{\varepsilon}\big)(t)\,d\ell+2\int_{0}^{\infty}u'(x-\ell)\big(\cP^{-}_{\ell}f^{-}_{\varepsilon}\big)(t)\,d\ell,\,\,\,\,(t,x)\in\sZ.
\end{align}
By \cite[Chapter 1, Proposition 1.5 (b)]{EthierKurtz:2005} and integration by parts, and noting that $u$ is compactly supported, we deduce that, for any $(t,x)\in\sZ$,
\begin{align}
&\frac{\partial}{\partial x}F_{\varepsilon}(t,x)=2u(x)\big(f^{-}_{\varepsilon}(t)\!-\!f^{+}_{\varepsilon}(t)\big)-2\!\int_{0}^{\infty}\!\!u(x\!+\!\ell)\frac{d}{d\ell}\big(\cP^{+}_{\ell}f^{+}_{\varepsilon}\big)(t)\,d\ell+2\!\int_{0}^{\infty}\!\!u(x\!-\!\ell)\frac{d}{d\ell}\big(\cP^{-}_{\ell}f^{-}_{\varepsilon}\big)(t)\,d\ell\nonumber\\
\label{eq:DerFunctFepsx} &\quad =2u(x)\big(f^{-}_{\varepsilon}(t)\!-\!f^{+}_{\varepsilon}(t)\big)-2\!\int_{0}^{\infty}\!u(x\!+\!\ell)\big(\cP^{+}_{\ell}\,\Gamma^{+}\!f^{+}_{\varepsilon}\big)(t)\,d\ell+2\!\int_{0}^{\infty}\!u(x\!-\!\ell)\big(\cP^{-}_{\ell}\,\Gamma^{-}\!f^{-}_{\varepsilon}\big)(t)\,d\ell.
\end{align}
A similar argument shows that $\partial F(t,\cdot)/\partial x$ is differentiable on $\bR$, for every $t\in\bR_{+}$, and that
\begin{align}
\frac{\partial^{2}}{\partial x^{2}}F_{\varepsilon}(t,x)&=2\int_{0}^{\infty}u(x+\ell)\big(\cP^{+}_{\ell}(\Gamma^{+})^{2}f^{+}_{\varepsilon}\big)(t)\,d\ell+2\int_{0}^{\infty}u(x-\ell)\big(\cP^{-}_{\ell}(\Gamma^{-})^{2}f^{-}_{\varepsilon}\big)(t)\,d\ell\nonumber\\
\label{eq:DerFunctFepsxx} &\quad +2u(x)\big(\big(\Gamma^{+}f^{+}_{\varepsilon}\big)(t)+\big(\Gamma^{-}f^{-}_{\varepsilon}\big)(t)\big)+2u'(x)\big(f^{-}_{\varepsilon}(t)-f^{+}_{\varepsilon}(t)\big).
\end{align}
The dominated convergence, together with the facts that $\cP^{\pm}_{\ell}(\Gamma^{\pm})^{2}f^{\pm}_{\varepsilon},\Gamma^{\pm}f^{\pm}_{\varepsilon},f^{\pm}_{\varepsilon}\in C_{0}(\bR_{+})$ (since $f^{\pm}_{\varepsilon}\in\sD((\Gamma^{\pm})^{2})$ as shown by \eqref{eq:FunctfepspmDomGammapm2}) and that $u\in C^{1}_{c}(\bR)$, ensures that $\partial^{2}F/\partial x^{2}\in C(\sZ)$.

Moreover, since $f^{\pm}_{\varepsilon}\in\sD((\Gamma^{\pm})^{2})$ (so that $\Gamma^{\pm}f^{\pm}_{\varepsilon}\in\sD(\Gamma^{\pm})$), we obtain from \cite[Chapter 1, Proposition 1.5 (b)]{EthierKurtz:2005} that, for any $\ell\in\bR_{+}$, $\cP^{\pm}_{\ell}f^{\pm}_{\varepsilon}\in\sD(\Gamma^{\pm})$ and $\Gamma^{\pm}\cP^{\pm}_{\ell}f^{\pm}_{\varepsilon}=\cP^{\pm}_{\ell}\Gamma^{\pm}f^{\pm}_{\varepsilon}\in\sD(\Gamma^{\pm})$, i.e., $\cP^{\pm}_{\ell}f^{\pm}_{\varepsilon}\in\sD((\Gamma^{\pm})^{2})$. It follows from Theorem \ref{thm:NoisyWHExistence} that $\cP^{\pm}_{\ell}f^{\pm}_{\varepsilon}$ is right-differentiable on $\bR_{+}$ and
\begin{align*}
\big(\cP^{\pm}_{\ell}f^{\pm}_{\varepsilon}\big)'_+(t)=\pm v(t)\big(\cP^{\pm}_{\ell}\,\Gamma^{\pm}\!f^{\pm}_{\varepsilon}\big)(t)-\frac{1}{2}\sigma^{2}(t)\big(\cP^{\pm}_{\ell}(\Gamma^{\pm})^{2}f^{\pm}_{\varepsilon}\big)(t),\quad t\in\bR_{+},
\end{align*}
where the right-hand side above is c\`{a}dl\`{a}g and bounded in light of Assumption \ref{assump:vsigma} and the fact that $\cP^{\pm}_{\ell}\,\Gamma^{\pm}\!f^{\pm}_{\varepsilon},\cP^{\pm}_{\ell}(\Gamma^{\pm})^{2}f^{\pm}_{\varepsilon}\in C_{0}(\bR_{+})$. This, together with \eqref{eq:FunctFeps} and the dominated convergence argument, implies that for any $x\in\bR$, $F_{\varepsilon}(\cdot,x)$ is right-differentiable on $\bR_{+}$ with
\begin{align}
\frac{\partial_{+}}{\partial t}F_{\varepsilon}(t,x)&=2\int_{0}^{\infty}u(x+\ell)\big(\cP^{+}_{\ell}f^{+}_{\varepsilon}\big)'_{+}(t)\,d\ell+2\int_{0}^{\infty}u(x-\ell)\big(\cP^{-}_{\ell}f^{-}_{\varepsilon}\big)'_{+}(t)\,d\ell\nonumber\\
&=2\int_{0}^{\infty}u(x+\ell)\bigg(v(t)\big(\cP^{+}_{\ell}\,\Gamma^{+}\!f^{+}_{\varepsilon}\big)(t)-\frac{1}{2}\sigma^{2}(t)\big(\cP^{+}_{\ell}(\Gamma^{+})^{2}f^{+}_{\varepsilon}\big)(t)\bigg)d\ell\nonumber\\
\label{eq:RightDerFunctFepst} &\quad -2\int_{0}^{\infty}u(x-\ell)\bigg(v(t)\big(\cP^{-}_{\ell}\,\Gamma^{-}\!f^{-}_{\varepsilon}\big)(t)+\frac{1}{2}\sigma^{2}(t)\big(\cP^{-}_{\ell}(\Gamma^{-})^{2}f^{-}_{\varepsilon}\big)(t)\bigg)d\ell.
\end{align}
Since $\cP^{\pm}_{\ell}\,\Gamma^{\pm}\!f^{\pm}_{\varepsilon},\cP^{\pm}_{\ell}(\Gamma^{\pm})^{2}f^{\pm}_{\varepsilon}\in C_{0}(\bR_{+})$ and {$u\in C_{c}(\bR)$}, we obtain from Assumption \ref{assump:vsigma} and the dominated convergence argument again, that $\partial_{+}F_{\varepsilon}(\cdot,x)/\partial t$ is c\`{a}dl\`{a}g and bounded on $\bR_{+}$, for any $x\in\bR$. In view of Lemma \ref{lem:CadlagRightDerivImpAbsCont}, this implies that $\partial_{+}F_{\varepsilon}(\cdot,x)/\partial t$ is absolutely continuous on $\bR_{+}$, for any $x\in\bR$. Hence, by integration by parts, for any $L\in(0,\infty)$ and any continuously differentiable test function $\rho$ on $[0,L]$ with $\rho(0)=\rho(L)=0$, we obtain that
\begin{align*}
0=F_{\varepsilon}(L,x)\rho(L)-F_{\varepsilon}(0,x)\rho(0)=\int_{0}^{L}\frac{\partial_{+}}{\partial t}F_{\varepsilon}(t,x)\rho(t)\,dt+\int_{0}^{L}F_{\varepsilon}(t,x)\rho'(t)\,dt,\quad x\in\bR,
\end{align*}
which shows that $\partial_{+}F_{\varepsilon}(\cdot,x)/\partial t$ is the generalized derivative of $F_{\varepsilon}(\cdot,x)$ on $[0,L]$ (cf. \cite[Section 2.1, Definition 1]{Krylov:1980}).

For any $n\in\bN$, let $\tau_{n}(s,a):=\tau^{+}_{a+n}(s,a)\wedge\tau^{-}_{a-n}(s,a)$. By It\^{o} formula with generalized derivatives (cf. \cite[Section 2.10, Theorem 1]{Krylov:1980}), and using \eqref{eq:DerFunctFepsx}, \eqref{eq:DerFunctFepsxx}, and \eqref{eq:RightDerFunctFepst}, we deduce that
\begin{align*}
&F_{\varepsilon}\big(\tau\wedge\tau_{n}(s,a),\varphi_{\tau\wedge\tau_{n}(s,a)}(s,a)\big)-F_{\varepsilon}(s,a)=\int_{s}^{\tau\wedge\tau_{n}(s,a)}\frac{\partial}{\partial x}F_{\varepsilon}\big(t,\varphi_{t}(s,a)\big)\sigma(t)\,dW_{t}\\
&\qquad\quad +\int_{s}^{\tau\wedge\tau_{n}(s,a)}\Big(u\big(\varphi_{t}(s,a)\big)\big(\big(\Gamma^{+}f^{+}_{\varepsilon}\big)(t)+\big(\Gamma^{-}f^{-}_{\varepsilon}\big)(t)\big)+u'\big(\varphi_{t}(s,a)\big)\big(f^{-}_{\varepsilon}(t)-f^{+}_{\varepsilon}(t)\big)\Big)\sigma^{2}(t)\,dt.
\end{align*}
In view of Assumption \ref{assump:vsigma}, \eqref{eq:Functfepspm}, and \eqref{eq:DerFunctFepsxIni}, and recalling the contraction property of $\cP_{\ell}^{\pm}$, $\ell\in\bR_{+}$, and the fact that $u\in C_{c}^{1}(\bR)$, we deduce that $\|\sigma F_{\varepsilon}\|_{\infty}\leq 8M\|u'\|_{\infty}\|f\|_{\infty}<\infty$. Hence, by taking expectation on both sides of the above equality, we obtain that
\begin{align}
&\bE\Big(F_{\varepsilon}\big(\tau\wedge\tau_{n}(s,a),\varphi_{\tau\wedge\tau_{n}(s,a)}(s,a)\big)\Big)-F_{\varepsilon}(s,a)\nonumber\\
\label{eq:ExpFepstaun} &\,\,=\bE\bigg(\!\int_{s}^{\tau\wedge\tau_{n}(s,a)}\!\!\!\Big(\!u\big(\varphi_{t}(s,a)\big)\big(\big(\Gamma^{+}\!f^{+}_{\varepsilon}\big)(t)\!+\!\big(\Gamma^{-}\!f^{-}_{\varepsilon}\big)(t)\big)\!+\!u'\big(\varphi_{t}(s,a)\big)\big(f^{-}_{\varepsilon}(t)\!-\!f^{+}_{\varepsilon}(t)\big)\!\Big)\sigma^{2}(t)\,dt\!\bigg).\qquad
\end{align}
Moreover, by \eqref{eq:Shifttau}, Lemma \ref{lem:EstHitTimePlusMinus}, and \eqref{eq:ABMHitTimePlusPhi}, we have
\begin{align*}
\bP\big(\tau_{n}(s,a)\!\leq\!T\big)\!\leq\bP\big(\tau^{+}_{n}(s)\!\leq\!T\big)\!+\!\bP\big(\tau^{-}_{-n}(s)\!\leq\!T\big)\!\leq 2\left(\!1\!-\!\bP\!\left(\sup_{r\in[0,\overline{\sigma}^{2}(T-s)]}\!\!\bigg(\frac{\|v\|_{\infty}r}{\underline{\sigma}^{2}}\!+\!W_{r}\!\bigg)\!<n\!\right)\!\right)\rightarrow 0,
\end{align*}
as $n\rightarrow\infty$. Recalling that $\tau\leq T$ $\bP$-a.s., this implies that $\tau\wedge\tau_{n}(s,a)\rightarrow\tau$ in probability, as $n\rightarrow\infty$. Hence, by dominated convergence and using the boundedness of $F_{\varepsilon}$, $f^{\pm}_{\varepsilon}$, $\Gamma^{\pm}f_{\varepsilon}^{\pm}$, $u$, $u'$, $\tau$, and $\sigma^{2}$, we can take $n\rightarrow\infty$ on both sides of \eqref{eq:ExpFepstaun} to deduce that
\begin{align*}
&\bE\big(F_{\varepsilon}\big(\tau,\varphi_{\tau}(s,a)\big)\big)-F_{\varepsilon}(s,a)\\
&\quad =\bE\bigg(\int_{s}^{\tau}\Big(u\big(\varphi_{t}(s,a)\big)\big(\big(\Gamma^{+}f^{+}_{\varepsilon}\big)(t)+\big(\Gamma^{-}f^{-}_{\varepsilon}\big)(t)\big)+u'\big(\varphi_{t}(s,a)\big)\big(f^{-}_{\varepsilon}(t)-f^{+}_{\varepsilon}(t)\big)\Big)\sigma^{2}(t)\,dt\bigg).
\end{align*}
Finally, by taking $\varepsilon\rightarrow 0+$ on both sides of the above equality, and using \eqref{eq:ConvInfNormDifffpmepsf}, \eqref{eq:ConvInfNormDiffFepsF}, \eqref{eq:ConvInfNormGammapmfpmepsGammapmf}, as well as dominated convergence, we obtain that
\begin{align}\label{eq:endStep1MainProof}
\bE\big(F\big(\tau,\varphi_{\tau}(s,a)\big)\big)-F(s,a)=\bE\bigg(\int_{s}^{\tau}u\big(\varphi_{t}(s,a)\big)\Big(\big(\Gamma^{+}f\big)(t)+\big(\Gamma^{-}f\big)(t)\Big)\sigma^{2}(t)\,dt\bigg),
\end{align}
which is indeed \eqref{eq:WHvarphiCompuGammafBoundtau} in light of \eqref{eq:FunctF}.

As for the validity of \eqref{eq:WHvarphiCompuGammafBoundtau} for $u\in C_{c}(\bR)$, we note that $C_{c}^{1}(\bR)$ is dense in $C_{c}(\bR)$ with the sup-norm. Hence, there exist $(u_{n})_{n\in\bN}\subset C_{c}^{1}(\bR)$ such that $\|u_{n}-u\|_{\infty}\rightarrow 0$, as $n\rightarrow\infty$, and \eqref{eq:WHvarphiCompuGammafBoundtau} holds true for each $u_{n}$. Therefore, the validity of \eqref{eq:WHvarphiCompuGammafBoundtau} for $u\in C_{c}(\bR)$ follows from dominated convergence and the uniform boundedness of $(\|u_{n}\|_{\infty})_{n\in\bN}$.

\medskip
\noindent
\textbf{Step 2.} Next, we will prove \eqref{eq:WHvarphi} for any $h\in C_{e}(\bR_+)$ and any $\bF$-stopping time $\tau$, while $u$ is still assumed to satisfy condition (a) in Step 1. The exponential decay of $h$ and the boundedness of $u$ ensure that \eqref{eq:FinExpIntuh} is satisfied in this case.

To this end, we first observe from Proposition \ref{prop:GammaGen} (ii) that $\int_{0}^\infty(\cP_{\ell}h)d\ell\in\sD(\overline{\Gamma})$. By Proposition \ref{prop:GammaGen} (i), the graph of $\overline\Gamma$ is the closure of that of $\Gamma$. Hence, there exist $(f_{n})_{n\in\bN}\subset C_{e,\text{cdl}}^{\text{ac}}(\bR_{+})$, so that
\begin{align*}
\lim_{n\rightarrow\infty}\bigg\|f_{n}-\int_{0}^{\infty}\cP_{\ell}h\,d\ell\,\bigg\|_{\infty}=0\quad\text{and}\quad\lim_{n\rightarrow\infty}\bigg\|{\Gamma}f_{n}-\overline{\Gamma}\int_{0}^{\infty}\cP_{\ell}h\,d\ell\,\bigg\|_{\infty}=0.
\end{align*}
It follows from \eqref{eq:GammaIntcPell} that
\begin{align}\label{eq:InfNormGammafnh}
\lim_{n\rightarrow\infty}\big\|\Gamma f_{n}+h\big\|_{\infty}=0.
\end{align}
By the contraction property of $\cP_{\ell}^{\pm}$, $\ell\in\bR_{+}$, we also have
\begin{align}\label{eq:UnifInfNormcPellpmfnIntcPellh}
\lim_{n\rightarrow\infty}{\sup_{\ell\in\bR_{+}}}\bigg\|\cP^{\pm}_{\ell}f_{n}-\cP^{\pm}_{\ell}\int_{0}^{\infty}\cP_{y}h\,dy\big\|_{\infty}\leq\lim_{n\rightarrow\infty}\bigg\|f_{n}-\int_{0}^{\infty}\cP_{y}h\,dy\,\bigg\|_{\infty}=0.
\end{align}
By \eqref{eq:WHvarphiCompuGammafBoundtau}, for any $n\in\bN$ and $T\in\bR_{+}$, with $\tau_{T}:=\tau\wedge T$, we have
\begin{align*}
&\bE\bigg(\int_{s}^{\tau_{T}}\!u\big(\varphi_{t}(s,a)\big)\big(\Gamma f_{n}\big)(t)\sigma^{2}(t)\,dt\bigg)= -2\int_{0}^{\infty}\!u(a+\ell)\big(\cP^{+}_{\ell}f_{n}\big)(s)\,d\ell-2\int_{0}^{\infty}\!u(a-\ell)\big(\cP^{-}_{\ell}f_{n}\big)(s)\,d\ell\nonumber\\
&\qquad\qquad\qquad\quad +2\,\bE\bigg(\int_{0}^{\infty}\!u\big(\varphi_{\tau_{T}}(s,a)+\ell\big)\big(\cP^{+}_{\ell}f_{n}\big)(\tau_{T})\,d\ell+\!\int_{0}^{\infty}\!u\big(\varphi_{\tau_{T}}(s,a)-\ell\big)\big(\cP^{-}_{\ell}f_{n}\big)(\tau_{T})\,d\ell\bigg).
\end{align*}
Letting $n\rightarrow\infty$ in the above equality and, by \eqref{eq:InfNormGammafnh}, \eqref{eq:UnifInfNormcPellpmfnIntcPellh}, the assumption that $u\in C_{c}^{1}(\bR_{+})$, and the dominated convergence, we obtain that
\begin{align*}
&\bE\bigg(\int_{s}^{\tau_{T}}u\big(\varphi_{t}(s,a)\big)h(t)\sigma^{2}(t)\,dt\bigg)= -\lim_{n\rightarrow\infty}\bE\bigg(\int_{s}^{\tau_{T}}u\big(\varphi_{t}(s,a)\big)\big(\Gamma f_{n}\big)(t)\sigma^{2}(t)\,dt\bigg)\\
&\quad =2\int_{0}^{\infty}u(a+\ell)\bigg(\cP^{+}_{\ell}\int_{\bR_{+}}\cP_{y}h\,dy\bigg)(s)\,d\ell+2\int_{0}^{\infty}u(a-\ell)\bigg(\cP^{-}_{\ell}\int_{\bR_{+}}\cP_{y}h\,dy\bigg)(s)\,d\ell\\
&\qquad -2\,\bE\bigg(\!\int_{0}^{\infty}\!\!u\big(\varphi_{\tau_{T}}\!(s,a)\!+\!\ell\big)\!\bigg(\!\cP^{+}_{\ell}\!\!\int_{0}^{\infty}\!\!\cP_{y}h\,dy\!\bigg)\!(\tau_{T})\,d\ell+\!\!\int_{0}^{\infty}\!\!u\big(\varphi_{\tau_{T}}\!(s,a)\!-\!\ell\big)\!\bigg(\!\cP^{-}_{\ell}\!\!\int_{0}^{\infty}\!\!\cP_{y}h\,dy\!\bigg)\!(\tau_{T})\,d\ell\!\bigg).
\end{align*}
Moreover, since $\int_{\bR_{+}}\!\cP_{y}h\,dy\in\sD(\overline{\Gamma})$, for any $\ell\in\bR_{+}$, we have $\cP^{+}_{\ell}\!\int_{\bR_{+}}\!\cP_{y}h\,dy\in C_{0}(\bR_{+})$ and $(\cP^{+}_{\ell}\!\int_{\bR_{+}}\!\cP_{y}h\,dy)(\infty)=0$. With the help of the continuity of sample paths of $\varphi(s,a)$, the contraction property of $(\cP^{\pm}_{\ell})_{\ell\in\bR_{+}}$, the assumption that $u\in C_{c}^{1}(\bR)$, \eqref{eq:FinExpIntuh}, and dominated convergence, we conclude by taking $T\rightarrow\infty$ in the above equality that
\begin{align*}
&\bE\bigg(\int_{s}^{\tau}u\big(\varphi_{t}(s,a)\big)h(t)\sigma^{2}(t)\,dt\bigg)=\lim_{T\rightarrow\infty}\bE\bigg(\int_{s}^{\tau_{T}}u\big(\varphi_{t}(s,a)\big)h(t)\sigma^{2}(t)\,dt\bigg)\\
&=2\int_{0}^{\infty}u(a+\ell)\bigg(\cP^{+}_{\ell}\!\int_{0}^{\infty}\cP_{y}h\,dy\bigg)(s)\,d\ell+2\int_{0}^{\infty}u(a-\ell)\bigg(\cP^{-}_{\ell}\!\int_{0}^{\infty}\cP_{y}h\,dy\bigg)(s)\,d\ell\\
&\quad -\!2\,\bE\!\left(\!\1_{\{\tau<\infty\}}\!\bigg(\!\int_{0}^{\infty}\!\!\!u\big(\varphi_{\tau}\!(s,a)\!+\!\ell\big)\!\bigg(\!\cP^{+}_{\ell}\!\!\int_{0}^{\infty}\!\!\cP_{y}h\,dy\!\bigg)\!(\tau)d\ell+\!\!\int_{0}^{\infty}\!\!\!u\big(\varphi_{\tau}\!(s,a)\!-\!\ell\big)\!\bigg(\!\cP^{-}_{\ell}\!\!\int_{0}^{\infty}\!\!\cP_{y}h\,dy\!\bigg)\!(\tau)d\ell\!\bigg)\!\right).
\end{align*}

\smallskip
\noindent
\textbf{Step 3.} Finally, we will complete the proof of \eqref{eq:WHvarphi} for any $\bF$-stopping time $\tau$, $h\in C_{e}(\bR_+)$, and $u\in C(\bR)$, which satisfy the condition \eqref{eq:FinExpIntuh}. Without loss of generality, we assume that both $u$ and $h$ are nonnegative. For general $u$ and $h$ it is sufficient to take $u=u^{+}-u^{-}$, and $h=h^{+}-h^{-}$ and the result follows from the linearity of integral and the operators $(\cP_{\ell})_{\ell\in\bR_{+}}$ and $(\cP^{\pm}_{\ell})_{\ell\in\bR_{+}}$.

Let now $(u_{n})_{n\in\bN}$ be a nondecreasing sequence of nonnegative functions in $C_{c}(\bR)$ such that, for any $n\in\bN$, $u_{n}(x)=u(x)$ for all $x\in [-n,n]$, and that $\supp(u_{n})\subset[-n-1,n+1]$. From the result in Step 2, for every $n\in\bN$, we have
\begin{align}
&\bE\bigg(\int_{s}^{\tau}\!u_{n}\big(\varphi_{t}(s,a)\big)h(t)\sigma^{2}(t)\,dt\bigg)\\
&\quad=2\int_{0}^{\infty}u_{n}(a+\ell)\bigg(\cP^{+}_{\ell}\!\int_{0}^{\infty}\!\cP_{y}h\,dy\bigg)(s)\,d\ell+2\int_{0}^{\infty}u_{n}(a-\ell)\bigg(\cP^{-}_{\ell}\!\int_{0}^{\infty}\!\cP_{y}h\,dy\bigg)(s)\,d\ell\nonumber\\
&\qquad -2\,\bE\left(\1_{\{\tau<\infty\}}\int_{0}^{\infty}\!u_{n}\big(\varphi_{\tau}(s,a)+\ell\big)\bigg(\cP^{+}_{\ell}\!\int_{0}^{\infty}\!\cP_{y}h\,dy\bigg)(\tau)\,d\ell\right)\nonumber\\
\label{eq:WHvarphiun} &\qquad -2\,\bE\left(\1_{\{\tau<\infty\}}\int_{0}^{\infty}\!u_{n}\big(\varphi_{\tau}(s,a)-\ell\big)\bigg(\cP^{-}_{\ell}\!\int_{0}^{\infty}\!\cP_{y}h\,dy\bigg)(\tau)\,d\ell\right).
\end{align}
In particular, for $\tau\equiv\infty$, we have
\begin{align}
&\bE\bigg(\int_{s}^{\infty}u_{n}\big(\varphi_{t}(s,a)\big)h(t)\sigma^{2}(t)\,dt\bigg)\nonumber\\
\label{eq:WHvarphiunInftau} &\quad =2\int_{0}^{\infty}\!u_{n}(a+\ell)\bigg(\cP^{+}_{\ell}\!\int_{0}^{\infty}\!\cP_{y}h\,dy\bigg)(s)\,d\ell+2\int_{0}^{\infty}\!u_{n}(a-\ell)\bigg(\cP^{-}_{\ell}\!\int_{0}^{\infty}\!\cP_{y}h\,dy\bigg)(s)\,d\ell.
\end{align}
By \eqref{eq:FinExpIntuh} and the monotone convergence, we deduce that
\begin{align*}
\lim_{n\rightarrow\infty}\bE\bigg(\int_{s}^{\tau}u_{n}\big(\varphi_{t}(s,a)\big)h(t)\sigma^{2}(t)\,dt\bigg)&=\bE\bigg(\int_{s}^{\tau}u\big(\varphi_{t}(s,a)\big)h(t)\sigma^{2}(t)\,dt\bigg)<\infty,\\
\lim_{n\rightarrow\infty}\bE\bigg(\int_{s}^{\infty}u_{n}\big(\varphi_{t}(s,a)\big)h(t)\sigma^{2}(t)\,dt\bigg)&=\bE\bigg(\int_{s}^{\infty}u\big(\varphi_{t}(s,a)\big)h(t)\sigma^{2}(t)\,dt\bigg)<\infty.
\end{align*}
Together with \eqref{eq:WHvarphiun} and \eqref{eq:WHvarphiunInftau}, we obtain that
\begin{align*}
&\lim_{n\rightarrow\infty}\bE\!\left(\!\1_{\{\tau<\infty\}}\!\int_{0}^{\infty}\!\!\left(\!u_{n}\big(\varphi_{\tau}(s,a)\!+\!\ell\big)\bigg(\!\cP^{+}_{\ell}\!\!\int_{0}^{\infty}\!\!\cP_{y}h\,dy\!\bigg)(\tau)\!+\!u_{n}\big(\varphi_{\tau}(s,a)\!-\!\ell\big)\bigg(\!\cP^{-}_{\ell}\!\!\int_{0}^{\infty}\!\!\cP_{y}h\,dy\!\bigg)(\tau)\!\right)d\ell\right)\\
&\quad =\lim_{n\rightarrow\infty}\bE\bigg(\int_{s}^{\infty}u_{n}\big(\varphi_{t}(s,a)\big)h(t)\sigma^{2}(t)\,dt\bigg)-\lim_{n\rightarrow\infty}\bE\bigg(\int_{s}^{\tau}u_{n}\big(\varphi_{t}(s,a)\big)h(t)\sigma^{2}(t)\,dt\bigg)<\infty.
\end{align*}
Therefore, we can pass the limit, as $n\rightarrow\infty$, for each term on either side of \eqref{eq:WHvarphiun}, which leads to \eqref{eq:WHvarphi} by monotone convergence. The proof of Theorem \ref{thm:Whvarphi} is now complete.

\subsection{Proof of Corollary \ref{cor:WHABM}}\label{subsec:ProofCorWHABM}

In this section, we will present a proof of Corollary \ref{cor:WHABM}, starting with the following technical lemma.

\begin{lemma}\label{lem:SemigroupCommutABM}
Under the setting of Corollary \ref{cor:WHABM}, for any $k,\ell\in\bR_{+}$, $\cP^{-}_{k}$ and $\cP^{+}_{\ell}$ commute. In particular, $(\cP^{+}_{\ell}\cP^{-}_{\ell})_{\ell\in\bR_{+}}$ is a Feller semigroup, and $\cP^{+}_{\ell}\cP^{-}_{\ell}=\cP_{\ell}$ on $B_{b}(\bR_{+})$, for any $\ell\in\bR_{+}$, where $(\cP_{\ell})_{\ell\in\bR_{+}}$ is given as in Proposition \ref{prop:GammaGen} (i). Moreover, for any $f\in C_{e,\text{cdl}}^{\text{ac}}(\bR_{+})$ and $\ell\in\bR_{+}$, we have $\cP_{\ell}^{\pm}f\in C_{e,\text{cdl}}^{\text{ac}}(\bR_{+})$, and
\begin{align*}
\Gamma^{+}\cP^{-}_{\ell}f=\cP^{-}_{\ell}\Gamma^{+}f,\quad\Gamma^{-}\cP^{+}_{\ell}f=\cP^{+}_{\ell}\Gamma^{-}f.
\end{align*}
\end{lemma}

\begin{proof}
When both $v$ and $\sigma$ are constants, the time-homogeneity of $\varphi$ implies that, for any $s\in\bR_{+}$, $\tau^{\pm}(s)-s$ has the identical law as $\tau^{\pm}(0)$. It follows from \eqref{eq:cPPlusMinus} that, for any $\ell\in\bR_{+}$ and $f\in B_{b}(\bR_{+})$,
\begin{align}\label{eq:cPPlusMinusABM}
\big(\cP^{\pm}_{\ell}f\big)(s)=\bE\big(f\big((s+\tau^{\pm}_{\ell}(0)\big)\big),\quad s\in\bR_{+}.
\end{align}
Hence by Fubini's theorem, for any $k,\ell\in\bR_{+}$ and $f\in B_{b}(\bR_{+})$, we have
\begin{align*}
&\big(\cP^{+}_{k}\cP^{-}_{\ell}f\big)(s)=\Big(\cP^{+}_{k}\,\bE\big(f\big(\cdot+\tau^{-}_{\ell}(0)\big)\Big)(s)
=\bE\bigg(\bE\Big(f\big(s+r+\tau^{-}_{\ell}(0)\big)\Big)\,\Big|_{r=\tau^{+}_{k}(0)}\bigg)\\
&\quad=\bE\bigg(\bE\Big(f\big(s+r+\tau^{+}_{k}(0)\big)\Big)\,\Big|_{r=\tau^{-}_{\ell}(0)}\bigg)
=\Big(\cP^{-}_{\ell}\,\bE\big(f\big(\cdot+\tau^{+}_{k}(0)\big)\Big)(s)=\big(\cP^{-}_{\ell}\cP^{+}_{k}f\big)(s),\quad s\in\bR_{+},
\end{align*}
that is, $\cP^{+}_{k}$ and $\cP^{-}_{\ell}$ are commutative. This, together with Proposition \ref{prop:FellerSemiGroupcPPlusMinus}, implies that $(\cP^{+}_{\ell}\cP^{-}_{\ell})_{\ell\in\bR_{+}}$ is a Feller semigroup. Moreover, by \eqref{eq:Gamma} and Proposition \ref{prop:GammaPlusMinusGen}, for any $f\in C_{e,\text{cdl}}^{\text{ac}}(\bR_{+})$, we have
\begin{align*}
\lim_{\ell\rightarrow 0+}\frac{1}{\ell}\,\big\|\cP^{+}_{\ell}\cP^{-}_{\ell}f-f\big\|_{\infty}=\lim_{\ell\rightarrow 0+}\frac{1}{\ell}\,\big\|\cP^{+}_{\ell}\big(\cP^{-}_{\ell}f-f\big)\big\|_{\infty}+\lim_{\ell\rightarrow 0+}\frac{1}{\ell}\,\big\|\cP^{+}_{\ell}f-f\big\|_{\infty}=\big(\Gamma^{+}\!+\Gamma^{-}\big)f=\Gamma f,
\end{align*}
namely, the strong generator of $(\cP^{+}_{\ell}\cP^{-}_{\ell})_{\ell\in\bR_{+}}$ coincides with $\Gamma$ on $C_{e,\text{cdl}}^{\text{ac}}(\bR_{+})$. In view of Proposition \ref{prop:GammaGen} (i), we deduce that
\begin{align*}
\cP_{\ell}f=\cP^{+}_{\ell}\cP^{-}_{\ell}f,\quad\text{for any }f\in B_{b}(\bR_{+}),\quad\ell\in\bR_{+}.
\end{align*}

Next, for any $\ell\in\bR_{+}$, $f\in C_{e,\text{cdl}}^{\text{ac}}(\bR_{+})$, and $s\in\bR_{+}$, we deduce from \eqref{eq:SpaceCacecdl}, \eqref{eq:cPPlusMinusABM}, and Fubini's theorem that
\begin{align*}
\big(\cP^{\pm}_{\ell}f\big)(s)\!=\!\bE\bigg(\!\!-\!\!\int_{s+\tau^{\pm}_{\ell}(0)}^{\infty}\!g_{f}(r)dr\!\bigg)\!\!=\!-\!\!\int_{0}^{\infty}\!\!\bigg(\!\int_{s}^{\infty}\!g_{f}(t\!+\!r)dt\!\bigg)dF_{\ell}^{\pm}(r)\!=\!-\!\!\int_{s}^{\infty}\!\!\bigg(\!\int_{0}^{\infty}\!g_{f}(t\!+\!r)dF_{\ell}^{\pm}(r)\!\bigg)dt,
\end{align*}
where $F_{\ell}^{\pm}$ denotes the distribution function of $\tau_{\ell}^{\pm}$ under $\bP$. Since $g_{f}$ is c\`{a}dl\`{a}g and vanishes at infinity with exponential rate, so is $h(t):=\int_{0}^{\infty}g_{f}(t+r)dF_{\ell}^{\pm}(r)$, $t\in\bR_{+}$. Hence, we have $\cP^{\pm}_{\ell}f\in C_{e,\text{cdl}}^{\text{ac}}(\bR_{+})$.

Finally, by the fact that $C_{e,\text{cdl}}^{\text{ac}}(\bR_{+})\subset\sD(\Gamma^{\pm})$ (given as in Proposition \ref{prop:GammaPlusMinusGen}), the commutativity of $\cP^{+}_{k}$ and $\cP^{-}_{\ell}$, for any $k,\ell\in\bR_{+}$, as well as the strong continuity of $(\cP_{\ell}^{\pm})_{\ell\in\bR_{+}}$ (given as in Proposition \ref{prop:FellerSemiGroupcPPlusMinus}), we obtain that, for any $f\in C_{e,\text{cdl}}^{\text{ac}}(\bR_{+})$ and $\ell\in\bR_{+}$,
\begin{align*}
\Gamma^{+}\cP^{-}_{\ell}f=\lim_{\delta\ell\rightarrow 0+}\frac{1}{\ell}\big(\cP^{+}_{\delta\ell}-I\big)\cP^{-}_{\ell}f=\lim_{\delta\ell\rightarrow 0+}\frac{1}{\ell}\cP^{-}_{\ell}\big(\cP^{+}_{\delta\ell}-I\big)f=\cP^{-}_{\ell}\Gamma^{+}f.
\end{align*}
The other identity can be shown using similar arguments. The proof of the lemma is complete.
\end{proof}

\begin{proof}[Proof of Corollary \ref{cor:WHABM}]
We will only present the proof of \eqref{eq:WHABM} in the case when $u\in C_{c}^{1}(\bR)$. The result for general $u\in C(\bR)$ satisfying \eqref{eq:FinExpIntuh0} follows from an approximation argument similar to those in the proof of Theorem \ref{thm:Whvarphi} (see the last paragraph in Step 1, and Step 3 therein). In what follows, we fix any $c\in(0,\infty)$, $a\in\bR$, and we set $h(t)=e^{-ct}$, $t\in\bR_{+}$.

To begin with, by Proposition \ref{prop:GammaGen} (ii) and the contraction property of $\cP^{\pm}_{\ell},\,\ell\in\bR_{+}$ (since $(\cP^{\pm}_{\ell})_{\ell\in\bR_{+}}$ is a Feller semigroup as shown in Proposition \ref{prop:FellerSemiGroupcPPlusMinus}), we see that both $\int_{0}^{L}\cP_{y}hdy$ and $\cP^{\pm}_{\ell}\int_{0}^{L}\cP_{y}hdy$ converge in $B_{b}(\bR_{+})$, as $L\rightarrow\infty$. It follows from \cite[Chapter 1, Lemma 1.4 (b)]{EthierKurtz:2005}, Proposition \ref{prop:GammaGen} (ii), and Lemma \ref{lem:SemigroupCommutABM} that, for any $\ell\in\bR_{+}$,
\begin{align*}
\cP^{+}_{\ell}\!\!\!\int_{0}^{\infty}\!\!\cP_{y}h\,dy=\cP^{+}_{\ell}\!\!\lim_{L\rightarrow\infty}\!\int_{0}^{L}\!\cP_{y}h\,dy=\!\lim_{L\rightarrow\infty}\!\!\cP^{+}_{\ell}\!\!\!\int_{0}^{L}\!\cP_{y}h\,dy=\!\lim_{L\rightarrow\infty}\!\int_{0}^{L}\!\cP^{+}_{\ell}\cP^{+}_{y}\cP^-_{y}h\,dy=\!\int_{0}^{\infty}\!\!\cP^{+}_{\ell+y}\cP^{-}_{y}h\,dy,
\end{align*}
and similarly,
\begin{align*}
\cP^{-}_{\ell}\int_{0}^{\infty}\cP_{y}h\,dy=\int_{0}^{\infty}\cP^{+}_{y}\cP^{-}_{\ell+y}h\,dy.
\end{align*}
Hence, by Corollary \ref{cor:WHvarphiInf} and Fubini's theorem, we have
\begin{align}
\big(\cE_{c}u\big)(a)&=\frac{c}{\sigma^{2}}\,\bE\bigg(\int_{0}^{\infty}u\big(\varphi_{t}(0,a)\big)h(t){\sigma^2}\,dt\bigg)\nonumber\\
&=\frac{2c}{\sigma^{2}}\int_{0}^{\infty}u(a+\ell)\bigg(\cP^{+}_{\ell}\!\int_{0}^{\infty}\!\cP_{y}h\,dy\bigg)(0)\,d\ell+\frac{2c}{\sigma^{2}}\int_{0}^{\infty}u(a-\ell)\bigg(\cP^{-}_{\ell}\!\int_{0}^{\infty}\!\cP_{y}h\,dy\bigg)(0)\,d\ell\nonumber\\
&=\frac{2c}{\sigma^{2}}\int_{0}^{\infty}\!u(a+\ell)\bigg(\int_{0}^{\infty}\!\cP_{\ell+y}^{+}\cP_{y}^{-}h\,dy\bigg)(0)\,d\ell+\frac{2c}{\sigma^{2}}\int_{0}^{\infty}\!u(a-\ell)\bigg(\int_{0}^{\infty}\!\cP_{y}^{+}\cP_{\ell+y}^{-}h\,dy\bigg)(0)\,d\ell\nonumber\\
&=\frac{2c}{\sigma^{2}}\int_{0}^{\infty}\!\!\int_{y}^{\infty}\!u(a+x-y)\big(\cP_{x}^{+}\cP_{y}^{-}h\big)(0)\,dx\,dy+\frac{2c}{\sigma^{2}}\int_{0}^{\infty}\!\!\int_{0}^{y}u(a+x-y)\big(\cP_{x}^{+}\cP_{y}^{-}h\big)(0)\,dx\,dy\nonumber\\
\label{eq:WHABMLeftIni} &=\frac{2c}{\sigma^{2}}\int_{0}^{\infty}\int_{0}^{\infty}u(a+x-y)\big(\cP_{x}^{+}\cP_{y}^{-}h\big)(0)\,dx\,dy.
\end{align}
Due to the time-homogeneity of $\varphi(a)$, we see that, for any $\ell\in\bR$ and $s\in\bR_{+}$, $\tau^{\pm}_{\ell}(s)-s$ has the identical law as $\tau^{\pm}_{\ell}(0)$ under $\bP$, and so
\begin{align*}
\big(\cP^{\pm}_{\ell}h\big)(s)=\bE\big(e^{-c\tau^{\pm}_{\ell}(s)}\big)=\bE\big(e^{-c(s+\tau^{\pm}_{\ell}(0))}\big)=e^{-cs}\,\bE\big(e^{-c\tau^{\pm}_{\ell}(0)}\big)=h(s)\big(\cP^{\pm}_{\ell}h\big)(0),\quad\ell,s\in\bR_{+}.
\end{align*}
It follows that, for any $x,y\in\bR_{+}$,
\begin{align}\label{eq:cPxypmExp0}
\big(\cP^{+}_{x}\cP^{-}_{y}h\big)(0)=\big(\cP^{+}_{x}h\big)(0)\cdot\big(\cP^{-}_{y}h\big)(0),
\end{align}
and that (noting that $h(t)=e^{-ct}\in C_{e,\text{cdl}}^{\text{ac}}(\bR_{+})\subset\sD(\Gamma^{\pm})$)
\begin{align}\label{eq:GammaPlusMinush}
\big(\Gamma^{\pm}h\big)(s)=\lim_{\ell\rightarrow 0+}\frac{1}{\ell}\big(\cP^{\pm}_{\ell}h-h\big)(s)=h(s)\cdot\lim_{\ell\rightarrow 0+}\frac{1}{\ell}\big(\cP^{\pm}_{\ell}h-h\big)(0)=h(s)\big(\Gamma^{\pm}h\big)(0).
\end{align}
Combining \eqref{eq:GammaPlusMinush} with \cite[Chapter 1, Proposition 1.5 (b)]{EthierKurtz:2005} leads to, for any $\ell\in\bR_{+}$,
\begin{align}\label{eq:DercPellpm}
\frac{\partial}{\partial\ell}\big(\cP^{\pm}_{\ell}h\big)(s)=\big(\cP^{\pm}_{\ell}\Gamma^{\pm}h\big)(s)=\big(\Gamma^{\pm}h\big)(0)\cdot\big(\cP^{\pm}_{\ell}h\big)(s),\quad s\in\bR_{+},
\end{align}
and thus
\begin{align}\label{eq:cPellpmExp}
\big(\cP^{\pm}_{\ell}h\big)(s)=h(s)\,e^{\ell(\Gamma^{\pm}h)(0)},\quad s\in\bR_{+}.
\end{align}
By combining \eqref{eq:WHABMLeftIni}, \eqref{eq:cPxypmExp0}, and \eqref{eq:cPellpmExp}, we obtain that
\begin{align}\label{eq:WHABMLeft}
\big(\cE_{c}u\big)(a)=2\int_{0}^{\infty}\int_{0}^{\infty}u(a+x-y)\,e^{x(\Gamma^{+}h)(0)+y(\Gamma^{-}h)(0)}\,dx\,dy.
\end{align}

Next, we will investigate the expression $\cE_{c}^{+}\cE_{c}^{-}u$ on the right-hand side of \eqref{eq:WHABM}. For any $t\in\bR_{+}$, let $\overline{F}_{t}$ be the distribution function of $\overline{\varphi}_{t}(0,0)$ under $\bP$. Using integration by parts, the assumption that $u\in C_{c}^{1}(\bR)$, and Fubini's theorem, we have
\begin{align*}
\big(\cE_{c}^{+}u\big)(a)&=c\,\bE\bigg(\int_{0}^{\infty}u\big(\overline{\varphi}_{t}(0,a)\big)h(t)\,dt\bigg)=c\int_{0}^{\infty}\bigg(\int_{0}^{\infty}u(a+\ell)\,d\overline{F}_{t}(\ell)\bigg)h(t)\,dt\\
&= -c\int_{0}^{\infty}\bigg(\int_{0}^{\infty}u'(a+\ell)\bP\big(\overline{\varphi}_{t}(0,0)\leq\ell\big)\,d\ell\bigg)h(t)\,dt \\
&= -c\int_{0}^{\infty}\!u'(a+\ell)\bigg(\int_{0}^{\infty}\!\bP\big(\tau^{+}_{\ell}(0)>t\big)h(t)\,dt\bigg)d\ell= -c\int_{0}^{\infty}\!u'(a+\ell)\bE\bigg(\int_{0}^{\tau^{+}_{\ell}(0)}\!h(t)\,dt\bigg)d\ell.
\end{align*}
Recalling $h(t)=e^{-ct}$, it follows from integration by parts, the assumption that $u\in C_{c}^{1}(\bR)$, \eqref{eq:DercPellpm}, and \eqref{eq:cPellpmExp} that
\begin{align}
\big(\cE_{c}^{+}u\big)(a)&=\int_{0}^{\infty}u'(a+\ell)\Big(\bE\big(e^{-c\tau^{+}_{\ell}(0)}\big)-1\Big)\,d\ell=\int_{0}^{\infty}u'(a+\ell)\big(\big(\cP^{+}_{\ell}h\big)(0)-1\big)\,d\ell\nonumber\\
\label{eq:WHABMRightPlus} &= -\big(\Gamma^{+}h\big)(0)\cdot\!\!\int_{0}^{\infty}\!u(a+\ell)\big(\cP_{\ell}^{+}h\big)(0)\,d\ell=-\big(\Gamma^{+}h\big)(0)\cdot\!\!\int_{0}^{\infty}\!u(a+\ell)\,e^{\ell(\Gamma^{+}h)(0)}\,d\ell.\quad
\end{align}
A similar argument as above shows that
\begin{align}\label{eq:WHABMRightMinus}
\big(\cE_{c}^{-}u\big)(a)= -\big(\Gamma^{-}h\big)(0)\cdot\int_{0}^{\infty}u(a-\ell)\,e^{\ell(\Gamma^{-}h)(0)}\,d\ell.
\end{align}
Therefore, by combining \eqref{eq:WHABMRightPlus} and \eqref{eq:WHABMRightMinus}, we deduce that
\begin{align*}
\big(\cE_{c}^{-}\cE_{c}^{+}u\big)(a)=\big(\Gamma^{+}h\big)(0)\big(\Gamma^{-}h\big)(0)\cdot\int_{0}^{\infty}\int_{0}^{\infty}u(a+x-y)\,e^{x(\Gamma^{+}h)(0)+y(\Gamma^{-}h)(0)}\,dx\,dy.
\end{align*}
Finally, by solving $(\Gamma^{\pm}h)(0)$ from \eqref{eq:NoisyWH} (with $v(t)\equiv v$ and $\sigma(t)\equiv\sigma$), and noting that $(\Gamma^{\pm}h)(0)\leq 0$ in light of \eqref{eq:GenGammaPlusMinus}, we obtain that
\begin{align*}
\big(\Gamma^{\pm}h\big)(0)=\pm \frac{v}{{\sigma^{2}}}-\sqrt{\frac{v^{2}}{{\sigma^{4}}}+\frac{2c}{{\sigma^{2}}}},
\end{align*}
and thus
\begin{align}\label{eq:WHABMRight}
\big(\cE_{c}^{-}\cE_{c}^{+}u\big)(a)=2c{\sigma^{-2}}\int_{0}^{\infty}\int_{0}^{\infty}u(a+x-y)\,e^{x(\Gamma^{+}h)(0)+y(\Gamma^{-}h)(0)}\,dx\,dy.
\end{align}
Combining \eqref{eq:WHABMLeft} and \eqref{eq:WHABMRight} completes the proof of corollary.
\end{proof}

\section{Example}\label{sec:EgOneJumpDriftVol}

In this section, we will present an example of functions $v$ and $\sigma$ in \eqref{eq:varphi}, for which Assumptions \ref{assump:vsigma} and \ref{assump:RangeDens} are all satisfied. More precisely, for some $v_{0},v_{1}\in\bR$ and $t_{0},\sigma_{0},\sigma_{1}\in(0,\infty)$, we consider
\begin{align}\label{eq:FunctsvsigmaOneJump}
v(t)=\begin{cases} v_{0},& t\in[0,t_{0}) \\ v_{1},& t\in[t_{0},\infty) \end{cases},\quad\sigma(t)=\begin{cases} \sigma_{0},& t\in[0,t_{0}) \\ \sigma_{1},& t\in[t_{0},\infty) \end{cases},
\end{align}
which clearly satisfy Assumption \ref{assump:vsigma}.

\begin{remark}
By induction, the result can then be extended to any c\`{a}dl\`{a}g piecewise constant functions $v$ and $\sigma$ with finitely many jumps, with the proof omitted here to avoid extra technicalities. Note that by refining the partition we can always assume that $v$ and $\sigma$ share the same set of jump times (with possibly zero-size jumps).
\end{remark}

\medskip
\noindent

\smallskip
\noindent
In order to show  $v$ and $\sigma$ in \eqref{eq:FunctsvsigmaOneJump} satisfy Assumption \ref{assump:RangeDens} with $\lambda=1$, using \eqref{eq:ABMgammaPlus} with regard to $\gamma^+$ and the analogous formula with regard to $\gamma^-$, and recalling that $\gamma=\gamma^{+}+\gamma^{-}$, we have
\begin{align}\label{eq:gammaOneJump}
\gamma(s,t)=\begin{cases} \,\displaystyle{\frac{2\sqrt{2}}{\sqrt{\pi\sigma_{0}^{2}(t-s)}}\exp\bigg(\!-\frac{v_{0}^{2}(t-s)}{2\sigma_{0}^{2}}\bigg)},& 0\leq s<t\leq t_{0}, \\ \vspace{-0.2cm} \\ \,\displaystyle{\frac{2\sqrt{2}}{\sqrt{\pi\sigma_{1}^{2}(t-s)}}\exp\bigg(\!-\frac{v_{1}^{2}(t-s)}{2\sigma_{1}^{2}}\bigg)},& t_{0}\leq s<t. \end{cases}.
\end{align}
The expression of $\gamma(s,t)$ when $0\leq s<t_{0}<t$ is omitted here as it is not needed for the rest of the proof. Since $C_{c}^{1}(\bR_{+})$ is dense in $C_{0}(\bR_{+})$, in order to show that $\{(I-\Gamma)f:\,f\in C_{e,\text{cdl}}^{\text{ac}}(\bR_{+})\}$ is dense in $C_{0}(\bR_{+})$, it is sufficient to verify that this set is dense in $C_{c}^{1}(\bR_{+})$. In what follows, for any fixed $\varepsilon>0$ and $h\in C_{c}^{1}({\bR}_{+})$, with $\supp(h)\subset[0,T)$ for some $T\in(0,\infty)$, we will construct $f^{\varepsilon}\in C_{e,\text{cdl}}^{\text{ac}}(\bR_{+})$ such that $\big\|h-(I-\Gamma)f^{\varepsilon}\big\|_{\infty}\leq\varepsilon$.

To begin with, we first construct $f_{1}\in C_{e,\text{cdl}}^{\text{ac}}(\bR_{+})$ such that
\begin{align}\label{eq:ResGammaf1overt0}
\big((I-\Gamma)f_{1}\big)(s)=h(s),\quad\text{for any }\,s\in[t_{0},\infty).
\end{align}
Note that if $T\in(0,t_{0}]$, \eqref{eq:ResGammaf1overt0} holds trivially with $f_{1}\equiv 0$. Hence, without loss of generality, hereafter we assume $T\in(t_{0},\infty)$. To this end, for any $s\in\bR_{+}$, we consider the process $\varphi(s,a)$ as in \eqref{eq:varphi} with $a=0$ and $v(t)\equiv v_{1}$, $\sigma(t)\equiv\sigma_{1}$, for any $t\in\bR_{+}$, which we now denote by $\varphi^{1}(s)=(\varphi^{1}_{t}(s))_{t\in[s,\infty)}$ in order to distinguish it from the process $\varphi(s)$ with coefficient functions \eqref{eq:FunctsvsigmaOneJump}. Note that those constant coefficients trivially satisfy Assumption \ref{assump:vsigma}. Accordingly, the passage times defined by \eqref{eq:tauPlusMinus}, the functions defined by Proposition \ref{prop:gammaPlusMinus} (ii), and the operators defined by \eqref{eq:cPPlusMinus} with respect to $\varphi^{1}(s)$, are denoted respectively by $\tau^{1,\pm}_{\ell}(s)$, $\cP^{1,\pm}_{\ell}$, and $\gamma^{1,\pm}$, for any $s\in\bR_{+}$ and $\ell\in\bR_{+}$. By Lemma \ref{lem:gammaCont}, $\gamma^{1,\pm}(\cdot,t)$ is continuous on $[0,t)$, and by a version of Proposition \ref{prop:gammaPlusMinus} (ii) (with constant coefficients $v_{1}$ and $\sigma_{1}$) and \eqref{eq:ABMgammaPlus}, we have
\begin{align}\label{eq:gammaPlusMinus1}
\gamma^{1,\pm}(s,t)=\frac{\sqrt{2}}{\sqrt{\pi\sigma_{1}^{2}(t-s)}}\exp\bigg(\!-\frac{v_{1}^{2}(t-s)}{2\sigma_{1}^{2}}\bigg)\mp\frac{2v_{1}}{\sigma_{1}^{2}}\Phi\bigg(\!\mp\frac{v_{1}\sqrt{t-s}}{\sigma_{1}}\bigg).
\end{align}
Next, by a version of Proposition \ref{prop:FellerSemiGroupcPPlusMinus} (with constant coefficients $v_{1}$ and $\sigma_{1}$), $(\cP^{1,\pm}_{\ell})_{\ell\in\bR_{+}}$ is a Feller semigroup, and in view of \eqref{eq:ABMHitTimePlusInt}, for any $\ell\in\bR_{+}$ and $f\in B_{b}(\bR_{+})$, we have
\begin{align*}
\big(\cP^{1,\pm}_{\ell}f\big)(s)=\int_{0}^{\infty}\frac{\ell}{\sqrt{2\pi\sigma_{1}^{2}t^{3}}}\exp\bigg(\!-\frac{\big(\ell\mp v_{1}t\big)^{2}}{2\sigma_{1}^{2}t}\bigg)f(s+t)\,dt,\quad s\in\bR_{+}.
\end{align*}
Together with the commutativity between $(\cP^{1,+}_{\ell})_{\ell\in\bR_{+}}$ and $(\cP^{1,-}_{\ell})_{\ell\in\bR_{+}}$, given by a version of Lemma \ref{lem:SemigroupCommutABM} (with constant coefficients $v_{1}$ and $\sigma_{1}$), we deduce that, for any $k,\ell\in\bR_{+}$ and $f\in B_{b}(\bR_{+})$,
\begin{align}\label{eq:cP1PluscP1Minus}
\big(\cP^{1,+}_{k}\cP^{1,-}_{\ell}f\big)(s)=\big(\cP^{1,-}_{\ell}\cP^{1,+}_{k}f\big)(s)=\int_{0}^{\infty}\rho^{1}_{\ell}(t)f(s+t)\,dt,\quad s\in\bR_{+},
\end{align}
where
\begin{align*}
\rho^{1}_{\ell}(t):=\frac{\ell^{2}e^{-v_{1}^{2}t/2}}{2\pi\sigma_{1}^{2}}\int_{0}^{t}\frac{1}{\sqrt{r^{3}(t-r)^{3}}}\exp\bigg(\!-\frac{2\ell^{2}t}{2\sigma_{1}r(t-r)}\bigg)dr,\quad t\in\bR_{+}.
\end{align*}
Moreover, by a version of Lemma \ref{lem:SemigroupCommutABM} (with constant coefficients $v_{1}$ and $\sigma_{1}$) again, $(\cP^{1,+}_{\ell}\cP^{1,-}_{\ell})_{\ell\in\bR_{+}}$ is a Feller semigroup which coincides with the Feller semigroup given by a version of Proposition \ref{prop:GammaGen} (i) (with constant coefficients $v_{1}$ and $\sigma_{1}$). In particular, the strong generator of $(\cP^{1,+}_{\ell}\cP^{1,-}_{\ell})_{\ell\in\bR_{+}}$ is the closure of $\Gamma^{1}$ defined by a version of \eqref{eq:Gamma} (with constant coefficients $v_{1}$ and $\sigma_{1}$). Denoting the strong generator of $(\cP^{1,\pm}_{\ell})_{\ell\in\bR_{+}}$ by $\Gamma^{1,\pm}$, it follows from a version of Proposition \ref{prop:GammaPlusMinusGen} (with constant coefficients $v_{1}$ and $\sigma_{1}$) and \eqref{eq:gammaPlusMinus1} that, for any $f\in C_{e,\text{cdl}}^{\text{ac}}(\bR_{+})$,
\begin{align}\label{eq:Gamma1}
\big(\Gamma^{1}f\big)(s)=\big(\Gamma^{1,+}f+\Gamma^{1,-}f\big)(s)=\int_{s}^{\infty}\frac{2\sqrt{2}}{\sqrt{\pi\sigma_{1}^{2}(t-s)}}\exp\bigg(\!-\frac{v_{1}^{2}(t-s)}{2\sigma_{1}^{2}}\bigg)g_{f}(t)\,dt,\quad s\in\bR_{+}.\quad
\end{align}
This, together with \eqref{eq:Gamma} and \eqref{eq:gammaOneJump}, implies that, for any $f\in C_{e,\text{cdl}}^{\text{ac}}(\bR_{+})$,
\begin{align}\label{eq:Gamma1fEquGammafovert0}
\big(\Gamma f\big)(s)=\int_{s}^{\infty}\frac{2\sqrt{2}}{\sqrt{\pi\sigma_{1}^{2}(t-s)}}\exp\bigg(\!-\frac{v_{1}^{2}(t-s)}{2\sigma_{1}^{2}}\bigg)g_{f}(t)\,dt=\big(\Gamma^{1}f\big)(s),\quad s\in[t_{0},\infty).
\end{align}

We now define $f_{1}$ as the $1$-resolvent operator associated with $(\cP^{1,+}_{\ell}\cP^{1,-}_{\ell})_{\ell\in\bR_{+}}$ on $h$, namely
\begin{align*}
f_{1}(s):=\int_{0}^{\infty}e^{-\ell}\big(\cP^{1,+}_{\ell}\cP^{1,-}_{\ell}h\big)(s)\,d\ell=\int_{0}^{\infty}\!\bigg(\int_{0}^{\infty}e^{-\ell}\rho^{1}_{\ell}(t)\,d\ell\bigg)h(s+t)\,dt,\quad s\in\bR_{+},
\end{align*}
where the second equality follows from \eqref{eq:cP1PluscP1Minus}. Since $h\in C_{c}^{1}(\bR_{+})$, we see that the last integral above is finite, and that $f_{1}\in C_{c}^{1}(\bR_{+})\subset C_{e,\text{cdl}}^{\text{ac}}(\bR_{+})$, and thus $f_{1}\in\sD(\overline{\Gamma^{1}})$. Moreover, by \cite[Lemma 1.27]{BottcherSchillingWang:2014}, we have $(I-\Gamma^{1})f_{1}=h$ on $\bR_{+}$, which, together with \eqref{eq:Gamma1fEquGammafovert0}, leads to \eqref{eq:ResGammaf1overt0}.

Next, in view of \eqref{eq:ResGammaf1overt0}, we have $h_{0}:=h-(I-\Gamma)f_{1}\in C_{c}(\bR_{+})$ with $\supp(h_{0})\subset [0,t_{0})$. Hence, there exists $h_{0}^{\varepsilon}\in C_{c}^{1}(\bR_{+})$, with $\supp(h_{0}^{\varepsilon})\subset [0,t_{0})$, such that
\begin{align}\label{eq:Consth0eps}
\|h_{0}^{\varepsilon}-h_{0}\|_{\infty}=\|h_{0}^{\varepsilon}-h+(I-\Gamma)f_{1}\|_{\infty}\leq\varepsilon.
\end{align}
We will construct $f_{0}^{\varepsilon}\in C_{c}^{1}(\bR_{+})$ such that
\begin{align}\label{eq:ResGammaf0eps}
\big((I-\Gamma)f_{0}^{\varepsilon}\big)(s)=h_{0}^{\varepsilon}(s),\quad s\in\bR_{+}.
\end{align}
The construction of $f_{0}^{\varepsilon}$ is almost identical to that of $f_{1}$ above, with the constant coefficients $v_{1}$ and $\sigma_{1}$ replaced by $v_{0}$ and $\sigma_{0}$, respectively, and the function $h$ replaced by $h_{0}^{\varepsilon}$. More precisely, let $(\cP^{0,\pm}_{\ell})_{\ell\in\bR_{+}}$ be the analogous operators of $({\cP}^{1,\pm}_{\ell})_{\ell\in\bR_{+}}$ with $(v_{1},\sigma_{1})$ replaced by $(v_{0},\sigma_{0})$. With similar arguments leading to \eqref{eq:cP1PluscP1Minus}, we deduce that $(\cP^{0,{+}}_{\ell}\cP^{0,-}_{\ell})_{\ell\in\bR_{+}}$ is a Feller semigroup, and that for any $k,\ell\in\bR_{+}$ and $f\in B_{b}(\bR_{+})$,
\begin{align}\label{eq:cP0PluscP0Minus}
\big(\cP^{0,+}_{k}\cP^{0,-}_{\ell}f\big)(s)=\big(\cP^{0,-}_{\ell}\cP^{0,+}_{k}f\big)(s)=\int_{0}^{\infty}\rho^{0}_{\ell}(t)f(s+t)\,dt,\quad s\in\bR_{+},
\end{align}
where
\begin{align*}
\rho^{0}_{\ell}(t):=\frac{\ell^{2}e^{-v_{0}^{2}t/2}}{2\pi\sigma_{0}^{2}}\int_{0}^{t}\frac{1}{\sqrt{r^{3}(t-r)^{3}}}\exp\bigg(\!-\frac{2\ell^{2}t}{2\sigma_{0}r(t-r)}\bigg)dr,\quad t\in\bR_{+}.
\end{align*}
Moreover, similar arguments leading to \eqref{eq:Gamma1} imply that the strong generator of $(\cP^{0,{+}}_{\ell}\cP^{0,-}_{\ell})_{\ell\in\bR_{+}}$ is the closure of $\Gamma^{0}$ defined by a version of \eqref{eq:Gamma} (with constant coefficients $v_{0}$ and $\sigma_{0}$), and that for any $f\in C_{e,\text{cdl}}^{\text{ac}}(\bR_{+})$,
\begin{align}\label{eq:Gamma0}
\big(\Gamma^{0}f\big)(s)=\int_{s}^{\infty}\frac{2\sqrt{2}}{\sqrt{\pi\sigma_{0}^{2}(t-s)}}\exp\bigg(\!-\frac{v_{0}^{2}(t-s)}{2\sigma_{0}^{2}}\bigg)g_{f}(t)\,dt,\quad s\in\bR_{+}.
\end{align}
We now define $f_{0}^{\varepsilon}$ as $1$-resolvent operator associated with $(\cP^{0,+}_{\ell}\cP^{0,-}_{\ell})_{\ell\in\bR_{+}}$ on $h_{0}^{\varepsilon}$, namely,
\begin{align*}
f_{0}^{\varepsilon}(s):=\int_{0}^{\infty}e^{-\ell}\big(\cP^{0,+}_{\ell}\cP^{0,-}_{\ell}h_{0}^{\varepsilon}\big)(s)\,d\ell=\int_{0}^{\infty}\!\bigg(\int_{0}^{\infty}e^{-\ell}\rho^{0}_{\ell}(t)\,d\ell\bigg)h_{0}^{\varepsilon}(s+t)\,dt,\quad s\in\bR_{+},
\end{align*}
where the second equality follow from \eqref{eq:cP0PluscP0Minus}. Since $h_{0}^{\varepsilon}\in C_{c}^{1}(\bR_{+})$ with $\supp(h_{0}^{\varepsilon})\subset[0,t_{0})$, we see that the last integral above is finite, and that $f_{0}^{\varepsilon}\in C_{c}^{1}(\bR_{+})\subset C_{e,\text{cdl}}^{\text{ac}}(\bR_{+})$ with $\supp(f_{0}^{\varepsilon})\subset[0,t_{0})$, so that $f_{0}^{\varepsilon}\in\sD(\overline{\Gamma^{0}})$ and $g_{f_{0}^{\varepsilon}}=(f_{0}^{\varepsilon})'$. Together with \eqref{eq:Gamma}, \eqref{eq:gammaOneJump}, and \eqref{eq:Gamma0}, we obtain that
\begin{align}\label{eq:GammaOneJumpGamma0}
\big(\Gamma f_{0}^{\varepsilon}\big)(s)=\big(\Gamma^{0}f_{0}^{\varepsilon}\big)(s)=\begin{cases} \,\displaystyle{\int_{s}^{t_{0}}\frac{2\sqrt{2}}{\sqrt{\pi\sigma_{0}^{2}(t-s)}}\exp\bigg(\!-\frac{v_{0}^{2}(t-s)}{2\sigma_{0}^{2}}\bigg)(f_{0}^{\varepsilon})'(t)\,dt},& t\in[0,t_{0}), \\ \\ \,\,\,\,0,& t\in[t_{0},\infty).\end{cases}\quad
\end{align}
On the other hand, by \cite[Lemma 1.27]{BottcherSchillingWang:2014} again, we have $(I-\Gamma^{0})f_{0}^{\varepsilon}=h_{0}^{\varepsilon}$ on $\bR_{+}$, which, together with \eqref{eq:GammaOneJumpGamma0}, leads to \eqref{eq:ResGammaf0eps}.

Finally, we define $f^{\varepsilon}:=f_{0}^{\varepsilon}+f_{1}$. In view of \eqref{eq:ResGammaf1overt0} and since $\supp(f_{0}^{\varepsilon})\subset[0,t_{0})$, for any $s\in[t_{0},\infty)$, we deduce that
\begin{align*}
\big|h(s)-\big((I-\Gamma)f^{\varepsilon}\big)(s)\big|=\big|h(s)-\big((I-\Gamma)f_{1}\big)(s)\big|=0.
\end{align*}
Moreover, for any $s\in[0,t_{0})$, by \eqref{eq:Consth0eps} and \eqref{eq:ResGammaf0eps} we have
\begin{align*}
\big|h(s)-\big((I-\Gamma)f^{\varepsilon}\big)(s)\big|=\big|h(s)-\big((I-\Gamma)f_{1}\big)(s)-h_{0}^{\varepsilon}(s)\big|\leq\varepsilon.
\end{align*}
which concludes the proof.

\appendix

\section{Additional Lemmas and Proofs}\label{sec:AdditionalProofs}

This appendix includes additional proofs of technical results. For those results presented in both the ``$+$" and the ``$-$" cases, we will only provide the proof for the ``$+$" scenario, as the ``$-$" scenario can be proved in an analogous way.

\subsection{Proof of Proposition \ref{prop:gammaPlusMinus}}\label{subsec:ProofPropgamma}

The proof of Proposition \ref{prop:gammaPlusMinus} relies on the following two technical lemmas, whose proofs are routine and therefore are omitted (and available from the authors upon request). The first lemma provides the estimates for the tail distribution of
$\tau_{\ell}^{+}$ when both $v$ and $\sigma$ are functions of time. An analogous estimate for $\tau_{\ell}^{-}$ can be obtained by replacing $v$ with $-v$.

\begin{lemma}\label{lem:EstHitTimePlusMinus}
For any $t>s\geq 0$ and $\ell\in\bR_{+}$, we have
\begin{align}\label{eq:EstHitTimePlusMinus}
\bP\!\left(\sup_{r\in
[0,\overline{\sigma}^{2}(t-s)]}\!\!\bigg(\frac{\|v\|_{\infty}r}{\underline{\sigma}^{2}}\!+\!W_{r}\!\bigg)\!<\ell\!\right)\!\leq\bP\big(\tau^{+}_{\ell}(s)\!>\!t\big)\!\leq\bP\!\left(\sup_{r\in
[0,\underline{\sigma}^{2}(t-s)]}\!\!\bigg(\!\!-\!\frac{\|v\|_{\infty}r}{\underline{\sigma}^{2}}\!+\!W_{r}\!\bigg)\!<\ell\!\right).
\end{align}
\end{lemma}
The second lemma  recalls some known results related to the hitting time of arithmetic Brownian motion with constant drift and volatility (cf. \cite[Part II, Section 2.2]{BorodinSalminen:2002}).
\begin{lemma}\label{lem:SuvDens}
Suppose that $v(t)\equiv v\in\bR$ and $\sigma(t)\equiv\sigma\in (0,\infty)$, for all $t\in\bR_{+}$. Then, for any $t>s\geq 0$, $\ell\in\bR_{+}$, we have
\begin{align}\label{eq:ABMHitTimePlusPhi}
\bP\big(\tau^{+}_{\ell}(s)\!>\!t\big)&=\Phi\bigg(\frac{\ell-v(t-s)}{\sigma\sqrt{t-s}}\bigg)-e^{2v\ell/\sigma^{2}}\Phi\bigg(\!-\frac{\ell+v(t-s)}{\sigma\sqrt{t-s}}\bigg)\\
\label{eq:ABMHitTimePlusInt} &=1-\int_{0}^{t-s}\frac{\ell}{\sigma\sqrt{2\pi r^{3}}}\exp\bigg(\!-\frac{\big(\ell-vr\big)^{2}}{2\sigma^{2}r}\bigg)dr,
\end{align}
and
\begin{align}
&\bP\big(\tau^{+}_{\ell}(s)>t,\,\varphi_{t}(s)\in
A\big)\\
\label{eq:JointDisttauvarphiInt} &\quad =\frac{1}{\sqrt{2\pi\sigma^{2}(t-s)}}\int_{A}\left(\exp\!\bigg(\!\!-\!\frac{\big(y-v(t-s)\big)^{2}}{2\sigma^{2}(t-s)}\bigg)\!-\exp\!\bigg(\frac{2v\ell}{\sigma^{2}}\!-\!\frac{\big(2\ell-y+v(t-s)\big)^{2}}{2\sigma^{2}(t-s)}\bigg)\right)dy,
\end{align}
where $\Phi$ denotes the standard normal distribution function. Consequently, for any $t>s\geq 0$,
\begin{align}\label{eq:ABMgammaPlus}
\lim_{\ell\rightarrow
0+}\frac{1}{\ell}\,\bP\big(\tau^{+}_{\ell}(s)>t\big)=\frac{\sqrt{2}}{\sqrt{\pi\sigma^{2}(t-s)}}\exp\bigg(\!-\frac{v^{2}}{2\sigma^{2}}(t-s)\bigg)-\frac{2v}{\sigma^{2}}\,\Phi\Big(\!-\frac{v}{\sigma}\sqrt{t-s}\Big).
\end{align}
\end{lemma}

\smallskip
\noindent
{\it Proof of Proposition \ref{prop:gammaPlusMinus}.}  In view of Lemma \ref{lem:EstHitTimePlusMinus} and \eqref{eq:ABMgammaPlus}, for any $t>s\geq 0$, we have
\begin{align}
&\frac{\sqrt{2}}{\sqrt{\pi\overline{\sigma}^{2}(t-s)}}\exp\bigg(\!-\frac{\overline{\sigma}^{2}\|v\|_{\infty}^{2}}{2\underline{\sigma}^{4}}(t-s)\bigg)-\frac{2\|v\|_{\infty}}{\underline{\sigma}^{2}}\Phi\bigg(\!-\frac{\overline{\sigma}\|v\|_{\infty}}{\underline{\sigma}^{2}}\sqrt{t-s}\bigg)\nonumber\\
&\quad\leq\liminf_{\ell\rightarrow 0+}\frac{1}{\ell}\,\bP\big(\tau^{+}_{\ell}(s)>t\big)\leq\limsup_{\ell\rightarrow
0+}\frac{1}{\ell}\,\bP\big(\tau^{+}_{\ell}(s)>t\big)\nonumber\\
\label{eq:EstLimsupinfTailTauPlus}
&\quad\leq\frac{\sqrt{2}}{\sqrt{\pi\underline{\sigma}^{2}(t-s)}}\exp\bigg(\!-\frac{\|v\|_{\infty}^{2}}{2\underline{\sigma}^{2}}(t-s)\bigg)+2\frac{\|v\|_{\infty}}{\underline{\sigma}^{2}}\Phi\bigg(\frac{\|v\|_{\infty}}{\underline{\sigma}}\sqrt{t-s}\bigg).
\end{align}
We first claim that, for any fixed $T>s\geq 0$, and for any $\varepsilon>0$, there exist at most finitely many $t\in(s,T]$ such that
\begin{align*}
\limsup_{\ell\rightarrow 0+}\frac{1}{\ell}\,\bP\big(\tau^{+}_{\ell}(s)>t\big)-\liminf_{\ell\rightarrow 0+}\frac{1}{\ell}\,\bP\big(\tau^{+}_{\ell}(s)>t\big)>\varepsilon.
\end{align*}
Otherwise, there exists $\varepsilon_{0}>0$ and an increasing sequence $(t_{n})_{n\in\bN}\subset(s,T]$ such that
\begin{align*}
\limsup_{\ell\rightarrow 0+}\frac{1}{\ell}\,\bP\big(\tau^{+}_{\ell}(s)>t_{n}\big)-\liminf_{\ell\rightarrow
0+}\frac{1}{\ell}\,\bP\big(\tau^{+}_{\ell}(s)>t_{n}\big)>\varepsilon_{0},\quad\text{for any }\,n\in\bN.
\end{align*}
Since $\bP(\tau^{+}_{\ell}(s)>t)$ is non-increasing in $t$, the above statement implies that
\begin{align*}
&\limsup_{\ell\rightarrow 0+}\frac{1}{\ell}\,\bP\big(\tau^{+}_{\ell}(s)>t_{1}\big)-\liminf_{\ell\rightarrow 0+}\frac{1}{\ell}\,\bP\big(\tau^{+}_{\ell}(s)>T\big)\\
&\quad >\varepsilon_{0}+\liminf_{\ell\rightarrow 0+}\frac{1}{\ell}\,\bP\big(\tau^{+}_{\ell}(s)>t_{1}\big)-\liminf_{\ell\rightarrow
0+}\frac{1}{\ell}\,\bP\big(\tau^{+}_{\ell}(s)>T\big)\\
&\quad\geq\varepsilon_{0}+\limsup_{\ell\rightarrow 0+}\frac{1}{\ell}\,\bP\big(\tau^{+}_{\ell}(s)>t_{2}\big)-\liminf_{\ell\rightarrow
0+}\frac{1}{\ell}\,\bP\big(\tau^{+}_{\ell}(s)>T\big)>\cdots\\
&\quad >n\varepsilon_{0}+\limsup_{\ell\rightarrow 0+}\frac{1}{\ell}\,\bP\big(\tau^{+}_{\ell}(s)>t_{n+1}\big)-\liminf_{\ell\rightarrow
0+}\frac{1}{\ell}\,\bP\big(\tau^{+}_{\ell}(s)>T\big)\rightarrow\infty,\quad\text{as }\,n\rightarrow\infty,
\end{align*}
which is a clear contradiction to \eqref{eq:EstLimsupinfTailTauPlus}. For any $s\in\bR_{+}$, let $\cT_{s}$ be the collection of $t\in(s,\infty)$ such that
$\lim_{\ell\rightarrow 0+}\ell^{-1}\bP(\tau^{+}_{\ell}(s)>t)$ exists and is finite. Noting that
\begin{align*}
(s,\infty)\setminus\cT_{s}&=\bigg\{t\in(s,\infty):\,\limsup_{\ell\rightarrow 0+}\frac{1}{\ell}\,\bP\big(\tau^{+}_{\ell}(s)>t\big)-\liminf_{\ell\rightarrow
0+}\frac{1}{\ell}\,\bP\big(\tau^{+}_{\ell}(s)>t\big)>0\bigg\}\\
&=\bigcup_{T=\lfloor s\rfloor+1}^{\infty}\bigcup_{k=1}^{\infty}\bigg\{t\in(s,T]:\,\limsup_{\ell\rightarrow
0+}\frac{1}{\ell}\,\bP\big(\tau^{+}_{\ell}(s)>t\big)-\liminf_{\ell\rightarrow 0+}\frac{1}{\ell}\,\bP\big(\tau^{+}_{\ell}(s)>t\big)>\frac{1}{k}\bigg\},
\end{align*}
we conclude that, for each $s\in\bR_{+}$, $(s,\infty)\setminus\cT_{s}$ is at most countable.

Next, we will show that, for any $s\in\bR_{+}$, $\lim_{\ell\rightarrow 0+}\ell^{-1}\bP(\tau^{+}_{\ell}(s)>\cdot)$ is continuous on $\cT_{s}$. Using Dambis-Dubins-Schwarz theorem (cf. \cite[Chapter 3, Theorem 4.6]{KaratzasShreve:1998}), for any $t\in(s,\infty)$,
\begin{align*}
\bP\big(\tau^{+}_{\ell}(s)>t\big)=\bP\bigg(\inf\bigg\{r\in\big[\langle M\rangle_{s},\infty):\int_{\langle M\rangle_{s}}^{r}\frac{v\big(\beta(u)\big)}{\sigma^{2}\big(\beta(u)\big)}\,du+B_{r}-B_{\langle M\rangle_{s}}\geq\ell\bigg\}>\langle M\rangle_{t}\bigg).
\end{align*}
Since $\langle M\rangle$ is continuous on $\bR_{+}$,  it is sufficient to prove the continuity of $\lim_{\ell\rightarrow 0+}\ell^{-1}\bP(\tau^{+}_{\ell}(s)>\cdot)$ on $\cT_{s}$ when $\sigma\equiv 1$.
In what follows, we will fix $s\in\bR_{+}$ and assume that $\sigma\equiv 1$. For any $t,t'\in\cT_{s}$ with $t'>t$, we first have
\begin{align}
\bP\big(\tau^{+}_{\ell}(s)>t\big)-\bP\big(\tau^{+}_{\ell}(s)>t'\big)&=\bP\bigg(\inf\bigg\{r\in[s,\infty):\int_{s}^{r}v(u)\,du+W_{r}-W_{s}\geq\ell\bigg\}\in(t,t']\bigg)\nonumber\\
\label{eq:DiffTailTau} &=\bP\bigg(\!\inf\!\bigg\{r\!\in\!\bR_{+}\!:\!\int_{0}^{r}\!v(s\!+\!u)du+W_{r}\geq\ell\bigg\}\!\in (t-s,t'-s]\bigg).\qquad
\end{align}
For any fixed $T\in[t'-s,\infty)$, we define a probability measure $\overline{\bP}_{1}$ on $(\Omega,\sF_{T})$ via
\begin{align*}
\frac{d\overline{\bP}_{1}}{d\bP|_{\cF_{T}}}=\exp\bigg(-\int_{0}^{T}v(s+r)\,dW_{r}-\frac{1}{2}\int_{0}^{T}v^{2}(s+r)\,dr\bigg).
\end{align*}
By Girsanov Theorem, the process $\overline{W}:=(\overline{W}_{t})_{t\in[0,T]}$ defined by
\begin{align*}
\overline{W}_{t}:=\int_{0}^{t}v(s+r)\,dr+W_{t},\quad t\in[0,T],
\end{align*}
is a standard Brownian motion under $\overline{\bP}_{1}$. With the help of Cauchy Schwarz inequality, we deduce from \eqref{eq:DiffTailTau} that, for any $\ell\in(0,\infty)$,
\begin{align}
&\frac{1}{\ell}\,\bP\big(\tau^{+}_{\ell}(s)>t\big)-\frac{1}{\ell}\,\bP\big(\tau^{+}_{\ell}(s)>t'\big)=\frac{1}{\ell}\,\bP\Big(\inf\Big\{r\in\bR_{+}:\overline{W}_{r}\geq\ell\Big\}\in
(t-s,t'-s]\Big)\nonumber\\
&\quad
=\frac{1}{\ell}\,\overline{\bE}_{1}\left(\exp\bigg(\int_{0}^{t'-s}v(s+r)\,d\overline{W}_{r}-\frac{1}{2}\int_{0}^{t'-s}v^{2}(s+r)\,dr\bigg)\1_{\{\inf\{r\in\bR_{+}:\overline{W}_{r}\geq\ell\}\in
(t-s,t'-s]\}}\right)\nonumber\\
&\quad\leq\left(\frac{1}{\ell}\,\overline{\bE}_{1}\left(\exp\bigg(2\int_{0}^{t'-s}v(s+r)\,d\overline{W}_{r}-\int_{0}^{t'-s}v^{2}(s+r)\,dr\bigg)\1_{\{\inf\{r\in\bR_{+}:\overline{W}_{r}\geq\ell\}\in
(t-s,t'-s]\}}\right)\right)^{1/2}\nonumber\\
\label{eq:DiffTailTauDecompHatP} &\qquad\cdot\bigg(\frac{1}{\ell}\,\overline{\bP}_{1}\Big(\inf\{r\in\bR_{+}:\overline{W}_{r}\geq\ell\}\in (t-s,t'-s]\Big)\bigg)^{1/2}.
\end{align}
Similarly, we define another probability measure $\overline{\bP}_{2}$ on $(\Omega,\sF_{T})$ via
\begin{align*}
\frac{d\overline{\bP}_{2}}{d\overline{\bP}_{1}}=\exp\bigg(2\int_{0}^{T}v(s+r)\,d\overline{W}_{r}-2\int_{0}^{T}v^{2}(s+r)\,dr\bigg).
\end{align*}
The Girsanov Theorem implies that the process $\overline{B}:=(\overline{B}_{t})_{t\in[0,T]}$, where
\begin{align*}
\overline{B}_{t}:=\overline{W}_{t}-2\int_{0}^{t}v(s+r)\,dr,\quad t\in[0,T],
\end{align*}
is a standard Brownian motion under $\overline{\bP}_{2}$. Hence, the first factor in \eqref{eq:DiffTailTauDecompHatP} can be estimated as
\begin{align}
&\frac{1}{\ell}\,\overline{\bE}_{1}\left(\exp\bigg(2\int_{0}^{t'-s}v(s+r)\,d\overline{W}_{r}-\int_{0}^{t'-s}v^{2}(s+r)\,dr\bigg)\1_{\{\inf\{r\in\bR_{+}:\overline{W}_{r}\geq\ell\}\in
(t-s,t'-s]\}}\right)\nonumber\\
&\quad
=\exp\bigg(\int_{0}^{t'-s}v^{2}(s+r)\,dr\bigg)\cdot\frac{1}{\ell}\,\overline{\bP}_{2}\left(\inf\bigg\{r\in\bR_{+}:\overline{B}_{r}+2\int_{0}^{r}v(s+u)\,du\geq\ell\bigg\}\in
(t-s,t'-s]\right)\nonumber\\
\label{eq:EstDiffTailTauDecompBarP} &\quad\leq
e^{\|v\|_{\infty}T}\cdot\frac{1}{\ell}\,\overline{\bP}_{2}\Big(\inf\Big\{r\in\bR_{+}:\overline{B}_{r}-2\|v\|_{\infty}r\geq\ell\Big\}>t-s\Big).
\end{align}
Combining \eqref{eq:DiffTailTauDecompHatP} and \eqref{eq:EstDiffTailTauDecompBarP}, and using \eqref{eq:ABMHitTimePlusInt} and \eqref{eq:ABMgammaPlus}, for any
$t,t'\in\cT_{s}$ with $t'>t$, we have
\begin{align*}
&\lim_{\ell\rightarrow 0+}\frac{1}{\ell}\,\bP\big(\tau^{+}_{\ell}(s)>t\big)-\lim_{\ell\rightarrow 0+}\frac{1}{\ell}\,\bP\big(\tau^{+}_{\ell}(s)>t'\big)\\
&\quad\leq e^{\|v\|_{\infty}T/2}\bigg(\lim_{\ell\rightarrow
0+}\frac{1}{\ell}\,\overline{\bP}_{2}\Big(\inf\Big\{r\in\bR_{+}:\overline{B}_{r}-2\|v\|_{\infty}r\geq\ell\Big\}>t-s\Big)\bigg)^{1/2}\\
&\qquad\cdot\bigg(\lim_{\ell\rightarrow 0+}\frac{1}{\ell}\,\overline{\bP}_{1}\Big(\inf\{r\in\bR_{+}:\overline{W}_{r}\geq\ell\}\in (t-s,t'-s]\Big)\bigg)^{1/2}\\
&\quad
=e^{\|v\|_{\infty}T/2}\bigg(\frac{\sqrt{2}\,e^{-2\|v\|_{\infty}^{2}(t-s)}}{\sqrt{\pi(t-s)}}-4\|v\|_{\infty}\Phi\Big(\!-2\|v\|_{\infty}\sqrt{t-s}\Big)\bigg)^{1/2}\bigg(\lim_{\ell\rightarrow
0+}\int_{t-s}^{t'-s}\frac{e^{-\ell^{2}/(2r)}}{\sqrt{2\pi r^{3}}}\,dr\bigg)^{1/2}\\
&\quad\leq\frac{2^{1/4}e^{\|v\|_{\infty}T/2}}{\pi^{1/4}(t-s)^{1/4}}\bigg(\int_{t-s}^{t'-s}\frac{1}{\sqrt{2\pi r^{3}}}\,dr\bigg)^{1/2},
\end{align*}
which completes the proof of the continuity of $\lim_{\ell\rightarrow 0+}\ell^{-1}\bP(\tau^{+}_{\ell}(s)>\cdot)$ on $\cT_{s}$.

Finally, since $\cT_{s}$ is dense in $(s,\infty)$, for any $t\!\in\!(s,\infty)$, there exist an increasing sequence $(\underline{t}_{n})_{n\in\bN}\subset\cT_{s}$ and a
decreasing sequence $(\overline{t}_{n})_{n\in\bN}\subset\cT_{s}$ such that $\lim_{n\rightarrow\infty}\underline{t}_{n}=\lim_{n\rightarrow\infty}\overline{t}_{n}=t$, and
so
\begin{align*}
\limsup_{\ell\rightarrow 0+}\frac{1}{\ell}\,\bP\big(\tau^{+}_{\ell}(s)\!>\!t\big)\!-\liminf_{\ell\rightarrow
0+}\frac{1}{\ell}\,\bP\big(\tau^{+}_{\ell}(s)\!>\!t\big)\leq\lim_{\ell\rightarrow
0+}\!\frac{1}{\ell}\,\bP\big(\tau^{+}_{\ell}(s)\!>\!\underline{t}_{n}\big)-\!\lim_{\ell\rightarrow
0+}\!\frac{1}{\ell}\,\bP\big(\tau^{+}_{\ell}(s)\!>\!\overline{t}_{n}\big)\rightarrow 0,
\end{align*}
as $n\rightarrow\infty$, where the convergence follows from the continuity of $\lim_{\ell\rightarrow 0+}\ell^{-1}\bP(\tau^{+}_{\ell}(s)>\cdot)$ on $\cT_{s}$. This
completes the proof of part (i), which, together with \eqref{eq:EstLimsupinfTailTauPlus}, leads to the estimates for $\gamma^{+}(s,t)$ in part (iii). As for part (ii),
since $\cT_{s}=(s,\infty)$, we have shown that $\gamma^{+}(s,\cdot)$ is continuous on $(s,\infty)$ for every $s\in\bR_{+}$. The non-increasing property of
$\gamma^{+}(s,\cdot)$ on $(s,\infty)$ follows immediately from the same property for $\bP(\tau^{+}_{\ell}(s)>\cdot)$. The proof of Proposition \ref{prop:gammaPlusMinus} is now complete.\hfill
$\Box$

\subsection{{Continuity of $\gamma^\pm(\cdot,t)$}}\label{sec:gammaCont}
In this section, we show that $\gamma^\pm(\cdot,t)$ is continuous in $[0,t)$ (see Lemma \ref{lem:gammaCont}). We will start by introducing a useful time change. Define
\begin{align}\label{eq:Defbeta}
\beta(t):=\inf\{y\in\bR_{+}:\int_{0}^{y}\sigma^{2}(r)dr>t\}, \quad t\in\bR_+
\end{align}
Due to Assumption \ref{assump:vsigma}, $\beta$ is invertible with $\beta^{-1}(t)=\int_0^t\sigma^2(r)dr$ for $t\in\bR_+$. Note that both $\beta$ and $\beta^{-1}$ are continuous and strictly increasing. We then define
\begin{align*}
\wh\varphi_t(s,a) := \varphi_{\beta(t)}(\beta(s),a) = a + \int_{\beta(s)}^{\beta(t)} v(r) d r + \int_{\beta(s)}^{\beta(t)} \sigma(r) d W_r.
\end{align*}
Note by changing $r$ to $\beta(z)$, $\int_{\beta(s)}^{\beta(t)} v(r) \dif r = \int_{s}^{t} \frac{v(\beta(z))}{\sigma^2(\beta(z))} d z.$ In addition, by Dambis-Dubins-Schwarz theorem (cf. \cite[Chapter 3, Theorem 4.6]{KaratzasShreve:1998}), $\wh W_t:= \int_0^{\beta(t)}\sigma(r) dW_r$ is a standard Brownian motion under $\bP$. Therefore, by denoting $\wh v(\cdot):=v(\beta(\cdot))/\sigma^{2}(\beta(\cdot))$, we have
\begin{align}\label{eq:phihat}
\wh\varphi_t(s,a) =  a + \int_s^t \wh v(r)dr + \wh W_{t} - \wh W_s, \quad t\ge s\ge 0,
\end{align}
where we note that $\wh v$ still satisfies Assumption \ref{assump:vsigma} (i). Define
\begin{align}\label{eq:tauhat}
\widehat\tau^+_\ell(s):=\inf\set{r\ge s:\wh\varphi_r(s)\ge\ell}.
\end{align}
Thus, $\tau^+_\ell(s) = \inf\set{r\ge s:\wh\varphi_{\beta^{-1}(r)}(\beta^{-1}(s)) \ge \ell}$. Let $\hat r:=\beta^{-1}(r)$. We have $r\ge s$ and $\wh\varphi_{\beta^{-1}(r)}(\beta^{-1}(s)) \ge \ell$ if and only if $\hat r\ge\beta^{-1}(s)$ and $\wh\varphi_{\hat r}(\beta^{-1}(s)) \ge \ell$. Consequently,
\begin{align*}
\tau^+_\ell(s) = \beta( \inf\set{\hat r\ge\beta^{-1}(s): \wh\varphi_{\hat r}(\beta^{-1}(s)) \ge \ell} ) = \beta\left(\widehat\tau^+_\ell(\beta^{-1}(s))\right).
\end{align*}
Therefore, $\bP(\tau^+_\ell(s)>t) = \bP(\widehat\tau^+_\ell(\beta^{-1}(s))>\beta^{-1}(t))$. In view of  \eqref{eq:phihat} and Proposition \ref{prop:gammaPlusMinus} (with $\tau$ replace by $\wh\tau$), we define $\wh\gamma^+(s,t):=\lim_{\ell\to0+}\ell^{-1}\bP(\wh\tau^+_\ell(s)>t)$. It follows that
\begin{align}\label{eq:gammaplushat}
\gamma^+(s,t) = \widehat\gamma^+(\beta^{-1}(s),\beta^{-1}(t)), \quad t\ge s\ge 0.
\end{align}
A similar result for $\gamma^-$ also holds.

\begin{lemma}\label{lem:Estgamma}
Let $\overline v$ and $\underline v$ be piece-wise constant real-valued function on $\bR_+$ such that $\overline v, \underline v$ are bounded by $c>0$ and $\overline v(r)-\underline v(r)=\varepsilon > 0$ for $r\in[0,t]$. Consider $\overline\gamma^+$ and $\underline\gamma^+$ as in Proposition \ref{prop:gammaPlusMinus} with $\sigma\equiv 1$ and $v$ replaced by $\overline v$ and $\underline v$, respectively. Then, there is $C_c>0$ such that $\underline\gamma^+(s,t) - \overline\gamma^+(s,t) \in [0, \varepsilon C_c(1+\sqrt{t-s}+(t-s))]$ for every $0\le s<t$.
\end{lemma}

\begin{proof}
Recall the definition of $\gamma$ from Proposition \ref{prop:gammaPlusMinus}. Define $\overline\tau^+_{\ell}(s)$ and $\underline\tau^+_{\varepsilon,\ell}(s)$ in the obvious way. Without loss of generality we suppose $s=0$ and omit the $s$ in $\tau$. Note $\bP(\underline\tau^+_\ell>t)\ge\bP(\overline\tau^+_\ell>t)$ and thus $\underline\gamma^+_\varepsilon\ge\overline\gamma^+_\varepsilon$. It is sufficient to estimate $\bP(\underline\tau^+_{\ell}>t)-\bP(\overline\tau^+_{\ell}>t)$. To this end for any fixed $T\in[t,\infty)$, we define a probability measure $\overline{\bP}$ on $(\Omega,\sF_{T})$ via
\begin{align*}
\frac{d\overline{\bP}}{d\bP|_{\cF_{T}}}=\exp\bigg(-\int_{0}^{T}\underline v(r)\,dW_{r}-\frac{1}{2}\int_{0}^{T}\underline v^{2}(r)\,dr\bigg).
\end{align*}
By Girsanov theorem, the process $\underline{W}:=(\underline{W}_{t})_{t\in[0,T]}$ defined by $\underline{W}_{t}:=W_t+\int_{0}^{t}\underline v(r)\,dr$ for $t\in[0,T]$ is a standard Brownian motion under $\underline{\bP}$. We introduce $\overline\bP$ and $\overline W$ similarly. Then, for $\ell\in(0,\infty)$,
\begin{align*}
&\bP\big(\underline\tau^{+}_{\ell}>t\big)-\bP\big(\overline\tau^{+}_{\ell}>t\big)\nonumber\\
&\quad
= \underline{\bE}\left(e^{\int_{0}^{t}\underline v(r)\,d\underline{W}_{r}-\frac{1}{2}\int_{0}^{t}\underline v^{2}(r)\,dr}\1_{\{\inf\{r\in\bR_{+}:\underline{W}_{r}\geq\ell\}>t\}}\right) - \overline{\bE}\left(e^{\int_{0}^{t}\overline v(r)\,d\overline{W}_{r}-\frac{1}{2}\int_{0}^{t}\overline v^{2}(r)\,dr}\1_{\{\inf\{r\in\bR_{+}:\overline{W}_{r}\geq\ell\}>t\}}\right)\\
&\quad= \bE\left( \left( e^{\int_{0}^{t}\underline v(r)\,d {W}_{r}-\frac{1}{2}\int_{0}^{t}\underline v^{2}(r)\,dr} - e^{\int_{0}^{t}\overline v(r)\,d {W}_{r}-\frac{1}{2}\int_{0}^{t}\overline v^{2}(r)\,dr} \right) \1_{\set{\inf\set{r\in\bR_+:W_r\ge\ell}>t}} \right).
\end{align*}
Using mean value theorem and the fact that $\max\set{x,y} \le x+y$ for $x,y>0$, we yield the following pathwise
\begin{align*}
&\left|e^{\int_{0}^{t}\underline v(r)\,d {W}_{r}-\frac{1}{2}\int_{0}^{t}\underline v^{2}(r)\,dr} - e^{\int_{0}^{t}\overline v(r)\,d {W}_{r}-\frac{1}{2}\int_{0}^{t}\overline v^{2}(r)\,dr}\right|\\
&\quad\le \left( e^{\int_{0}^{t}\underline v(r)\,d {W}_{r}+\frac{1}{2}\int_{0}^{t}\underline v^{2}(r)\,dr} + e^{\int_{0}^{t}\overline v(r)\,d {W}_{r}-\frac{1}{2}\int_{0}^{t}\overline v^{2}(r)\,dr} \right) \left(\left|\varepsilon W_{t}\right| + \frac{1}2\int_{0}^{t}|\overline v^2(r)-\underline v^2(r)|dr\right).
\end{align*}
Then, noting that $\frac{1}2\int_{0}^{t}|\overline v^2(r)-\underline v^2(r)|dr\le t c\varepsilon$ due to mean value theorem and the assumption, we yield
\begin{align*}
&\bP\big(\underline\tau^{+}_{\ell}>t\big)-\bP\big(\overline\tau^{+}_{\ell}>t\big)\\
&\quad\le \bE\left( \left( e^{\int_{0}^{t}\underline v(r)\,d {W}_{r}+\frac{1}{2}\int_{0}^{t}\underline v^{2}(r)\,dr} + e^{\int_{0}^{t}\overline v(r)\,d {W}_{r}-\frac{1}{2}\int_{0}^{t}\overline v^{2}(r)\,dr} \right) \left(\left|\varepsilon W_{t}\right| + tc\varepsilon\right) \1_{\set{\inf\set{r\in\bR_+:W_r\ge\ell}>t}} \right).
\end{align*}
Next, we will estimate
\begin{align*}
\underline I_\varepsilon &:= \bE\left( e^{\int_{0}^{t}\underline v(r)\,d {W}_{r}+\frac{1}{2}\int_{0}^{t}\underline v^{2}(r)\,dr}  \left(\left|\varepsilon W_{t}\right| + t b\varepsilon\right) \1_{\set{\inf\set{r\in\bR_+:W_r\ge\ell}>t}} \right)\\
&= \varepsilon\underline\bE\left( e^{\int_{0}^{t}\underline v(r)\,d \underline{W}_{r}+\frac{1}{2}\int_{0}^{t}\underline v^{2}(r)\,dr}  \left(\left|\underline W_{t}\right| + t b\right) \1_{\set{\inf\set{r\in\bR_+:\underline W_r\ge\ell}>t}} \right).
\end{align*}
We cast the the expectation back using Girsanov theorem again and yield
\begin{align*}
\underline I_\varepsilon &\le \varepsilon \bE\left( \left(\left| W_{t} + \int_{0}^t \underline v(r)dr\right| + t c \right)\1_{\set{W_r+\int_0^r \underline v(z)dz<\ell,r\in[0,t]}} \right) \le \varepsilon \bE\left( \left(\left| W_{t} - ct \right| + 3 t c \right)\1_{\set{W_r-cr<\ell,\,r\in[0,t]}} \right).
\end{align*}
Note that $\bP(W_r-cr<\ell,\,r\in[0,t]) = \Phi\left(\frac{\ell+ct}{\sqrt{t}}\right)-e^{2c\ell}\Phi\left(\!-\frac{\ell-ct}{\sqrt{t}}\right)$ by \eqref{eq:ABMHitTimePlusPhi}. In addition, by \eqref{eq:JointDisttauvarphiInt}, we have
\begin{align*}
&\bE\left( \left| W_{t} - ct \right| \1_{\set{W_r-cr<\ell,\,r\in[0,t]}} \right) = \frac{1}{\sqrt{2\pi t}}\int_{-\infty}^\ell |y|\left(\exp\!\bigg(\!\!-\!\frac{\big(y-ct\big)^{2}}{2t}\bigg)\!-\exp\!\bigg(2v\ell\!-\!\frac{\big(2\ell-y+ct\big)^{2}}{2t}\bigg)\right)dy\\
&\quad\le \frac{1}{\sqrt{2\pi t}}\int_{-\infty}^\ell (2\ell\1_{[0,\ell]}(y)-y)\left(\exp\!\bigg(\!\!-\!\frac{\big(y-ct\big)^{2}}{2t}\bigg)\!-\exp\!\bigg(2b\ell\!-\!\frac{\big(2\ell-y+ct\big)^{2}}{2t}\bigg)\right)dy\\
&\quad\le -\frac{1}{\sqrt{2\pi t}}\int_{-\infty}^\ell y\left(\exp\!\bigg(\!\!-\!\frac{\big(y-ct\big)^{2}}{2t}\bigg)\!-\exp\!\bigg(2c\ell\!-\!\frac{\big(2\ell-y+ct\big)^{2}}{2t}\bigg)\right)dy + 2\ell\bP(W_r-cr<\ell,\,r\in[0,t])\\
&\quad= -\frac12\left(ct\Phi\left(\frac{\ell-\wh bt}{\sqrt{t}}\right)-e^{2c\ell}(2\ell+ct)\left(1-\Phi\left(\frac{\ell+ct}{\sqrt{t}}\right)\right)\right) + 2\ell\bP(W_r-cr<\ell,\,r\in[0,t]),
\end{align*}
where $\Phi$ is the CDF of standard normal distribution. 
Using the fact that
\begin{align*}
\lim_{\ell\to0+}\frac1\ell\left(\Phi\left(\frac{\ell-ct}{\sqrt{t}}\right)+\Phi\left(\frac{\ell+ct}{\sqrt{t}}\right)-1\right)=2\sqrt{\frac{2}{\pi t}}e^{-\frac12c^2t},
\end{align*}
we yield
\begin{align*}
\lim_{\ell\to0+}\frac{1}{\ell}\underline I_\varepsilon &\le \varepsilon \left(\left(\sqrt{\frac{2t}{\pi}}ce^{-\frac12c^2t}+2(1+c^2t)(1-\Phi(c\sqrt{t}))\right) + 3tc\left(\sqrt{\frac{2}{\pi t}}e^{-\frac12c^2t}-2c\Phi(-ct)\right)\right),
\end{align*}
i.e., $\lim_{\ell\to0+}\underline I_\varepsilon\frac{1}{\ell}\le \varepsilon \underline C_{c}\left(1+\sqrt{t}+t\right)$ for some $\underline C_{c}>0$. A similar argument provides an estimate for $\overline v$. The proof is complete.
\end{proof}

\begin{lemma}\label{lem:cdlhApprox}
Let $h$ be a real-valued c\`adl\`ag function on $[0,T]$. Then, for any $\varepsilon>0$, there is $h_\varepsilon$ such that $h_\varepsilon$ is c\`adl\`ag piecewise constant with only finitely many jumps on $[0,T]$ and $h(t)\in[h_\varepsilon(t)-{\varepsilon}, h_\varepsilon(t)+{\varepsilon}]$ for $t\in[0,T]$.
\end{lemma}

\begin{proof}
Consider $\varepsilon>0$. Let $a_0=0$ and define $a_{n+1}:=\inf\set{t\in[a_n,T):|h(a_n)-h(t)|>\varepsilon}\wedge T$ if $a_n<T$. Note that, due to the right continuity of $h$, $a_{n+1}-a_n>0$ for any $n\in\bN$. We claim that the procedure of defining $a_n$ will stop within finite step, i.e., there is $N_\varepsilon\in\bN$ such that $a_{N_\varepsilon}=T$. Indeed, suppose otherwise, then there is $T'\in[0,T]$ such that $a_n\le T'$ for $n\in\bN$. This implies that for any $\delta\in(0,T')$, there is $x,y\in(T'-\delta,T')$ such that $|h(x)-h(y)|>\varepsilon$, which contradicts the assumption that $h$ has left limit at $T'$. Finally, we set $h_\varepsilon(t):=\sum_{n=0}^{N_\varepsilon-1}h(a_n)\1_{[a_n,a_{n+1})}(t)+h(T)\1_{\set{T}}(t)$, and the proof is complete.
\end{proof}

\begin{lemma}\label{lem:gammaContPcwConst}
If $v$ and $\sigma$ are piece-wise constant with only finitely many jumps, then for every $t\in\bR_+$, the function $\gamma^\pm(\cdot,t)$ are continuous on $[0,t)$.
\end{lemma}
\begin{proof}
Suppose ${0}<t_{1}<\cdots<t_n<t$ are the jump points of $v$ and $\sigma$, and we denote accordingly $v_k$ and $\sigma_k$ the value of $v$ and $\sigma$ at $t\in[t_k,t_{k+1})$. There is no need to consider jumps in $[t,\infty)$ as this does not affect $\gamma^\pm(\cdot,t)$. Note that for $t_{k}\le s< t_{k+1}$, $\bP(\tau^+_\ell(s)>t)$ can be expressed in terms of iterated integral using \eqref{eq:JointDisttauvarphiInt},
\begin{align}\label{eq:SuvivalProbConca}
\bP(\tau^+_\ell(s)>t) =  \frac{1}{\sqrt{2\pi\sigma_k^{2}(t_{k+1}-s)}} \int_{-\infty}^\ell \left(e^{-\!\frac{(y-v_k(t_{k+1}-s))^{2}}{2\sigma_k^{2}(t_{k+1}-s)}}\!-e^{\frac{2v_k\ell}{\sigma_k^{2}}\!-\!\frac{(2\ell-y+v_k(t_{k+1}-s))^{2}}{2\sigma_k^{2}(t_{k+1}-s)}}\right)\bP(\tau^+_{\ell-y}(t_{k+1})>t) dy.
\end{align}
Note that for $k=n-1$, $\bP(\tau^+_{\ell-y}(t_n)>t)$ can be expressed via \eqref{eq:ABMHitTimePlusPhi}, and thus is continuously differentiable in $\ell\in[0,\infty)$ and has bounded derivative. It follows from \eqref{eq:SuvivalProbConca} and backward induction in $k$ that for any $s\in[0,t)$, $\bP(\tau^+_\ell(s)>t)$ is also continuously differentiable at $\ell\in[0,\infty)$ and has bounded derivative.
Consequently,
\begin{align*}
\gamma^+(s,t) &= \frac{\partial}{\partial\ell}\bP(\tau^+_\ell(s)>t)\big|_{\ell=0}\\
&= \int_{-\infty}^0 \frac{-2y}{\sigma_k^{3}\sqrt{2\pi(t_{k+1}-s)}}\exp\bigg(\!-\frac{\big(y-v_{k}(t_{k+1}-s)\big)^{2}}{2\sigma_{k}^{2}(t_{k+1}-s)}\bigg)\bP(\tau^+_{-y}(t_{k+1})>t) dy\\
&= \int_{-\infty}^{-v_{k}\sqrt{t_{k+1}-s}/\sigma_{k}}\frac{2\ell_k(s,z)}{\sigma_{0}^{2}\sqrt{2\pi}(t_{k+1}-s)}\bP\big(\tau^{+}_{\ell_k(s,z)}(t_{k+1})>t\big)e^{-z^{2}/2}\,dz,
\end{align*}
where $\ell_k(s,z):= -\sigma_{k}z\sqrt{t_{k+1}-s}-v_{k}(t_{k+1}-s)$. It is clear that $\gamma^+(s,t)$ is continuous it $s\in(t_k,t_{k+1})$ and $\lim_{s\to t_k+}\gamma^+(s,t)=\gamma^+(t_{k},t)$. What is left to show that $\lim_{s\to t_{k+1}-}\gamma^+(s,t)=\gamma^+(t_{k+1},t)$. To this end note that due to Proposition \ref{prop:gammaPlusMinus} (i) and the dominated convergence, we have
\begin{align*}
\lim_{s\rightarrow t_{k+1}-}\gamma^{+}(s,t)&=\int_{-\infty}^{0}\bigg(\lim_{s\rightarrow t_{k+1}-}\frac{2\ell^{2}_k(s,z)}{\sigma_{k}^{2}\sqrt{2\pi}(t_{k+1}-s)}\cdot\frac{1}{\ell_k(s,z)}\bP\Big(\tau^{+}_{\ell_k(s,z)}(t_{k+1})>t\Big)\bigg)e^{-z^{2}/2}\,dz\\
&=\gamma^{+}(t_{k+1},t)\int_{-\infty}^{0}\frac{2z^{2}}{\sqrt{2\pi}}\,e^{-z^{2}/2}\,dz=\gamma^{+}(t_{k+1},t).
\end{align*}
The proof is complete.
\end{proof}

\begin{lemma}\label{lem:gammaCont}
For any $v$ and $\sigma$ satisfying Assumptions \ref{assump:vsigma}, then for every $t\in\bR_+$, the function $\gamma^\pm(\cdot,t)$ are continuous on $[0,t)$.
\end{lemma}
\begin{proof}
Due to \eqref{eq:gammaplushat} and the fact that $\beta^{-1}$ is continuous, it is sufficient to assume $\sigma\equiv 1$. We will only prove $\gamma^+(\cdot,T)$ is continuous on $[0,T)$ as the proof for $\gamma^-$ can be done analogously. To this end, in view of Lemma \ref{lem:cdlhApprox}, for $\varepsilon>0$ we let $v_\varepsilon$ a real-valued c\`adl\`ag function with only finitely many jumps on $[0,T]$ such that $ v(r)\in[v_\epsilon(r)-\varepsilon, v_\epsilon(r)+\varepsilon]$ for $r\in[0,T]$. We also define $\gamma^+_\varepsilon$ accordingly in the sense of Proposition \ref{prop:gammaPlusMinus}. Due to Lemma \ref{lem:Estgamma} (with $\overline v:=v_\varepsilon+\varepsilon$ and $\underline v:=v_\varepsilon-\varepsilon$), we have
\begin{align}\label{eq:gammaEpsilon}
\left|\gamma^+_\varepsilon(r,T)-\gamma^+(r,T)\right|  \le 2\varepsilon C_b(1+\sqrt{T}+T),\quad r\in[0,T).
\end{align}
Moreover, $\gamma^+(\cdot,T)$ is continuous on $[0,T)$ due to Lemma \ref{lem:gammaContPcwConst}. Therefore,
\begin{align*}
\gamma^+_\varepsilon(s,T) - 2\varepsilon C_b(1+\sqrt{T}+T) \le \liminf_{r\to s}\gamma^+(r,T) \le \limsup_{r\to s}\gamma^+(r,T) \le \gamma^+_\varepsilon(s,T) + 2\varepsilon C_b(1+\sqrt{T}+T).
\end{align*}
Since $\lim_{\varepsilon\to0+}\gamma^+_\varepsilon(s,T)=\gamma^+(s,T)$ due to \eqref{eq:gammaEpsilon}, the proof is complete.
\end{proof}

\subsection{Proof of Lemma \ref{lem:Gamma}}\label{subsec:ProofLemGamma}

We fix any $f\in C_{e,\text{cdl}}^{\text{ac}}(\bR_{+})$ throughout the proof. By Proposition \ref{prop:gammaPlusMinus} (iii), for any
$s\in\bR_{+}$, we have
\begin{align*}
\int_{s}^{\infty}|g_{f}(t)\gamma(s,t)|\,dt\leq\wt{K}\int_{s}^{\infty}\bigg(\frac{1}{\sqrt{t-s}}+1\bigg)e^{-\kappa t}\,dt=\wt{K}e^{-\kappa
s}\int_{0}^{\infty}\bigg(\frac{1}{\sqrt{t}}+1\bigg)e^{-\kappa t}\,dt,
\end{align*}
where $\wt{K}:=K(\sqrt{2}/\sqrt{\pi\underline{\sigma}^{2}}+2\|v\|_{\infty}/\underline{\sigma}^{2})$. Since the last integral above is finite, we obtain that $\Gamma f$ is well-defined and vanishes at infinity with exponential rate.

It remains to verify the continuity of $\Gamma f$ on $\bR_{+}$. For any $\varepsilon>0$, we pick $\delta\in(0,1)$ so that $\int_{0}^{\delta}(t^{-1/2}+1)dt\leq\varepsilon$. For any $s\in\bR_{+}$ and $s'\in(s,s+\delta]$, by Proposition \ref{prop:gammaPlusMinus} (iii) and the boundedness of $g_{f}$, we deduce that
\begin{align*}
\Big|\big(\Gamma f\big)(s)\!-\!\big(\Gamma
f\big)(s')\Big|&\leq\!\int_{s}^{s+\delta}\!\big|g_{f}(t)\big|\gamma(s,t)dt+\!\int_{s'}^{s+\delta}\!\big|g_{f}(t)\big|\gamma(s',t)dt+\bigg|\!\int_{s+\delta}^{\infty}\!g_{f}(t)\big(\gamma(s,t)\!-\!\gamma(s',t)\big)dt\bigg|\\
&\leq 2\wt{K}\int_{s}^{s+\delta}\bigg(\frac{1}{\sqrt{t-s}}+1\bigg)dt+\bigg|\int_{s+\delta}^{\infty}g_{f}(t)\big(\gamma(s,t)-\gamma(s',t)\big)dt\bigg|\\
&\leq 2\wt{K}\varepsilon+\bigg|\int_{0}^{\infty}\1_{[s+\delta,\infty)}(t)g_{f}(t)\big(\gamma(s,t)-\gamma(s',t)\big)dt\bigg|.
\end{align*}
The second term above vanishes to zero, as $s'\rightarrow s+$, due to Lemma \ref{lem:gammaCont}, the dominated convergence theorem, as well as the
estimate
\begin{align*}
\1_{[s+\delta,\infty)}(t)\big|g_{f}(t)\big|\big|\gamma(s,t)-\gamma(s',t)\big|\leq 2\wt{K}\delta^{-1/2}e^{-\kappa t},
\end{align*}
which follows from Proposition \ref{prop:gammaPlusMinus} (iii) and the exponential decay of $g_{f}$. Thus $\Gamma f$ is right-continuous at any $s\in\bR_{+}$. The left
continuity of $\Gamma f$ on $\bR_{+}$ can be shown using similar arguments. The proof of Lemma \ref{lem:Gamma} is complete.

\subsection{Proof of Lemma \ref{lem:ContWtExptau}}\label{subsec:ProofCont}

We will only present the proof for the ``plus" case, as the ``minus" case can be verified in an analogous way. We will fix $f\in C_{0}({\bR}_{+})$ throughout the proof, for which we 2 stipulate $f(\infty)=0$.

\medskip
\noindent
\textbf{(i)} We fix any $\ell_{1},\ell_{2}\in\bR_{+}$ with $\ell_{1}<\ell_{2}$, and so $\wt{\tau}^{+}_{\ell_{1}}\leq\wt{\tau}^{+}_{\ell_{2}}$. Hence, we have $f(Z^{1}_{\wt{\tau}^{+}_{\ell_{2}}})=f(\infty)=0$ on $\{\wt{\tau}^{+}_{\ell_{1}}=\infty\}$ and $Z^{2}_{\wt{\tau}^{+}_{\ell_{1}}}=\ell_{1}$ on $\{\wt{\tau}^{+}_{\ell_{1}}<\infty\}$. Together with Lemma \ref{lem:StrongMarkovWtcM} and \eqref{eq:ShiftWttau}, we deduce that
\begin{align}
\wt{\bE}_{s,0}\Big(f\Big(Z^{1}_{\wt{\tau}^{+}_{\ell_{2}}}\Big)\Big)&=\wt{\bE}_{s,0}\bigg(\1_{\{\wt{\tau}^{+}_{\ell_{1}}<\infty\}}\1_{\{\wt{\tau}^{+}_{\ell_{1}}\leq\wt{\tau}^{+}_{\ell_{2}}\}}\wt{\bE}_{s,0}\Big(f\Big(Z^{1}_{\wt{\tau}^{+}_{\ell_{2}}}\Big)\Big|\wt{\sF}_{\wt{\tau}^{+}_{\ell_{1}}}\Big)\bigg)\nonumber\\
&=\wt{\bE}_{s,0}\bigg(\!\1_{\{\wt{\tau}^{+}_{\ell_{1}}<\infty\}}\1_{\{\wt{\tau}^{+}_{\ell_{1}}\leq\wt{\tau}^{+}_{\ell_{2}}\}}\wt{\bE}_{Z^{1}_{\wt{\tau}^{+}_{\ell_{1}}},Z^{2}_{\wt{\tau}^{+}_{\ell_{1}}}}\!\!\Big(f\Big(Z^{1}_{\wt{\tau}^{+}_{\ell_{2}}}\Big)\Big)\!\bigg)\!=\!\wt{\bE}_{s,0}\bigg(\!\1_{\{\wt{\tau}^{+}_{\ell_{1}}<\infty\}}\wt{\bE}_{Z^{1}_{\wt{\tau}^{+}_{\ell_{1}}},\ell_{1}}\!\Big(f\Big(Z^{1}_{\wt{\tau}^{+}_{\ell_{2}}}\Big)\Big)\!\bigg)\nonumber\\
\label{eq:WtExpftauell2} &=\wt{\bE}_{s,0}\bigg(\1_{\{\wt{\tau}^{+}_{\ell_{1}}<\infty\}}\,\wt{\bE}_{Z^{1}_{\wt{\tau}^{+}_{\ell_{1}}},0}\Big(f\Big(Z^{1}_{\wt{\tau}^{+}_{\delta\ell}}\Big)\Big)\bigg),
\end{align}
where $\delta\ell:=\ell_{2}-\ell_{1}$. Therefore, by \eqref{eq:WtExptauExptau}, Lemma \ref{lem:EstHitTimePlusMinus}, and \eqref{eq:ABMHitTimePlusPhi}, we obtain that
\begin{align*}
&\bigg|\wt{\bE}_{s,0}\Big(f\Big(Z^{1}_{\wt{\tau}^{+}_{\ell_{2}}}\Big)\Big)-\wt{\bE}_{s,0}\Big(f\Big(Z^{1}_{\wt{\tau}^{+}_{\ell_{1}}}\Big)\Big)\bigg|\leq\wt{\bE}_{s,0}\left(\1_{\{\wt{\tau}^{+}_{\ell_{1}}<\infty\}}\bigg|\wt{\bE}_{Z^{1}_{\wt{\tau}^{+}_{\ell_{1}}},0}\Big(f\Big(Z^{1}_{\wt{\tau}^{+}_{\delta\ell}}\Big)\Big)-f\Big(Z^{1}_{\wt{\tau}^{+}_{\ell_{1}}}\Big)\bigg|\,\right)\\
&\quad\leq\sup_{t\in\bR_{+}}\Big|\wt{\bE}_{t,0}\Big(f\Big(Z^{1}_{\wt{\tau}^{+}_{\delta\ell}}\Big)\Big)-f(t)\Big|=\sup_{t\in\bR_{+}}\Big|\bE\Big(f\big(\tau^{+}_{\delta\ell}(t)\big)\Big)-f(t)\Big|\\
&\quad\leq\sup_{s_{1},s_{2}\in\bR_{+}:\,|s_{2}-s_{1}|\leq\delta\ell}\big|f(s_{2})-f(s_{1})\big|+2\|f\|_{\infty}\sup_{t\in\bR_{+}}\bP\Big(\tau_{\delta\ell}^{+}(t)>\delta\ell+t\Big)\\
&\quad\leq\sup_{\substack{s_{1},s_{2}\in\bR_{+} \\ |s_{2}-s_{1}|\leq\delta\ell}}\!\big|f(s_{2})-f(s_{1})\big|+2\|f\|_{\infty}\!\left(\!\Phi\bigg(\frac{\sqrt{\delta\ell}}{\underline{\sigma}}\big(1\!-\!\|v\|_{\infty}\big)\!\bigg)\!-\!e^{-2\|v\|_{\infty}^{2}\delta\ell/\underline{\sigma}^{2}}\Phi\bigg(\!\!-\!\frac{\sqrt{\delta\ell}}{\underline{\sigma}}\big(1\!+\!\|v\|_{\infty}\big)\!\bigg)\!\right),
\end{align*}
where the right-hand side tends to $0$ as $\delta\ell=\ell_{2}-\ell_{1}\rightarrow 0+$, uniformly for all $s\in\bR_{+}$. This finishes the proof of part (i).

\medskip
\noindent
\textbf{(ii)} Since $f\in C_{0}(\bR_{+})$, and $Z^{1}_{\wt{\tau}^{+}_{\ell}}\geq s$, $\wt{\bP}_{s,0}$-a.s., we have
\begin{align*}
\Big|\wt{\bE}_{s,0}\Big(f\Big(Z^{1}_{\wt{\tau}^{\pm}_{\ell}}\Big)\Big)\Big|\leq\sup_{r\in[s,\infty)}\big|f(r)\big|\rightarrow 0,\quad\text{as }\,s\rightarrow\infty.
\end{align*}
It remains to show that $s\mapsto\wt{\bE}_{s,0}(f(Z^{1}_{\wt{\tau}^{\pm}_{\ell}}))$ is continuous on $\bR_{+}$, for any $\ell\in\bR_{+}$. Note that the continuity is trivial when $\ell=0$ since $\wt{\bE}_{s,0}(f(Z^{1}_{\wt{\tau}^{\pm}_{0}}))=f(s)$. We therefore fix any $\ell\in(0,\infty)$ for the rest of the proof. For any $s_{1},s_{2}\in\bR_{+}$ with $s_{1}<s_{2}$ and $\delta s:=s_{2}-s_{1}$, by Lemma \ref{lem:StrongMarkovWtcM} and the fact that $Z^{1}_{\delta s}=s_{1}+\delta_{s}=s_{2}$, $\wt{\bP}_{s_{1},0}$-a.s., we first have
\begin{align*}
\wt{\bE}_{s_{1},0}\Big(f\Big(Z^{1}_{\wt{\tau}^{+}_{\ell}}\Big)\Big)&=\wt{\bE}_{s_{1},0}\Big(\1_{\{\delta s\leq\wt{\tau}^{+}_{\ell}\}}f\Big(Z^{1}_{\wt{\tau}^{+}_{\ell}}\Big)\Big)+\wt{\bE}_{s_{1},0}\Big(\1_{\{\delta s>\wt{\tau}^{+}_{\ell}\}}f\Big(Z^{1}_{\wt{\tau}^{+}_{\ell}}\Big)\Big)\\
&=\wt{\bE}_{s_{1},0}\bigg(\1_{\{\delta s\leq\wt{\tau}^{+}_{\ell}\}}\wt{\bE}_{s_{1},0}\Big(f\Big(Z^{1}_{\wt{\tau}^{+}_{\ell}}\Big)\Big|\wt{\sF}_{\delta s}\Big)\bigg)+\wt{\bE}_{s_{1},0}\Big(\1_{\{\delta s>\wt{\tau}^{+}_{\ell}\}}f\Big(Z^{1}_{\wt{\tau}^{+}_{\ell}}\Big)\Big)\\
&=\wt{\bE}_{s_{1},0}\bigg(\1_{\{\delta s\leq\wt{\tau}^{+}_{\ell}\}}\wt{\bE}_{s_{2},Z^{2}_{\delta s}}\Big(f\Big(Z^{1}_{\wt{\tau}^{+}_{\ell}}\Big)\Big)\bigg)+\wt{\bE}_{s_{1},0}\Big(\1_{\{\delta s>\wt{\tau}^{+}_{\ell}\}}f\Big(Z^{1}_{\wt{\tau}^{+}_{\ell}}\Big)\Big)
\end{align*}
It follows that
\begin{align}
\bigg|\wt{\bE}_{s_{1},0}\Big(\!f\Big(Z^{1}_{\wt{\tau}^{+}_{\ell}}\Big)\!\Big)\!-\!\wt{\bE}_{s_{2},0}\Big(\!f\Big(Z^{1}_{\wt{\tau}^{+}_{\ell}}\Big)\!\Big)\bigg|&\leq\left|\wt{\bE}_{s_{1},0}\left(\1_{\{\delta s\leq\wt{\tau}^{+}_{\ell}\}}\bigg(\wt{\bE}_{s_{2},Z^{2}_{\delta s}}\Big(f\Big(Z^{1}_{\wt{\tau}^{+}_{\ell}}\Big)\Big)-\wt{\bE}_{s_{2},0}\Big(f\Big(Z^{1}_{\wt{\tau}^{+}_{\ell}}\Big)\Big)\bigg)\right)\right|\nonumber\\
&\quad +\left|\wt{\bE}_{s_{1},0}\bigg(\!\1_{\{\delta s>\wt{\tau}^{+}_{\ell}\}}\wt{\bE}_{s_{2},0}\Big(\!f\Big(Z^{1}_{\wt{\tau}^{+}_{\ell}}\Big)\!\Big)\!\bigg)\right|\!+\!\bigg|\wt{\bE}_{s_{1},0}\Big(\!\1_{\{\delta s>\wt{\tau}^{+}_{\ell}\}}f\Big(Z^{1}_{\wt{\tau}^{+}_{\ell}}\Big)\!\Big)\bigg|\nonumber\\
\label{eq:DecompWtExps12Wttau} &\leq\wt{\bE}_{s_{1},0}\Big(\overline{f}\Big(Z^{2}_{\delta s}\Big)\Big)+2\|f\|_{\infty}\wt{\bP}_{s,0}\big(\wt{\tau}^{+}_{\ell}<\delta s\big),
\end{align}
where we define
\begin{align*}
\overline{f}(r):=\sup_{\substack{s\in\bR_{+},\ell_{1},\ell_{2}\in\bR_{+} \\ |\ell_{1}-\ell_{2}|<|r|}}\bigg|\wt{\bE}_{s,0}\Big(f\Big(Z^{1}_{\wt{\tau}^{+}_{\ell_{1}}}\Big)\Big)-\wt{\bE}_{s_{1},0}\Big(f\Big(Z^{1}_{\wt{\tau}^{+}_{\ell_{2}}}\Big)\Big)\bigg|,\quad r\in\bR.
\end{align*}
In view of part (i), $\overline{f}$ is a bounded even function such that $\lim_{r\rightarrow 0}\overline{f}(r)=0$. Since $Z^{2}_{\delta s}$ admits a normal distribution under $\wt{\bP}_{s_{1},0}$ with mean $\int_{s_{1}}^{s_{2}}v(t)dt$ and variance $\int_{s_{1}}^{s_{2}}\sigma^{2}(t)dt)$, which converges to $0$ in distribution as $\delta s\rightarrow 0$, we obtain that
\begin{align}\label{eq:LimWtExps12Wttau1}
\lim_{\delta s\rightarrow 0}\wt{\bE}_{s_{1},0}\Big(\overline{f}\Big(Z^{2}_{\delta s}\Big)\Big)=0.
\end{align}
Moreover, by \eqref{eq:WtExptauExptau}, Lemma \ref{lem:EstHitTimePlusMinus}, and \eqref{eq:ABMHitTimePlusPhi}, we have
\begin{align}\label{eq:LimWtExps12Wttau2}
\wt{\bP}_{s_{1},0}\big(\wt{\tau}^{+}_{\ell}\!\!<\!\delta s\big)\!=\!\bP\big(\tau^{+}_{\ell}(s_{1})\!<\!s_{2}\big)\!\leq\!\left(\!1\!-\!\Phi\bigg(\frac{\ell\!-\!\|v\|_{\infty}\delta s}{\underline{\sigma}\sqrt{\delta
s}}\bigg)\!+\!e^{2\|v\|_{\infty}\ell/\underline{\sigma}^{2}}\Phi\bigg(\!\!-\!\frac{\ell\!+\!\|v\|_{\infty}\delta s}{\underline{\sigma}\sqrt{\delta
s}}\bigg)\!\right)\rightarrow 0,\qquad
\end{align}
as $\delta s\rightarrow 0$. Combining \eqref{eq:DecompWtExps12Wttau}, \eqref{eq:LimWtExps12Wttau1}, and \eqref{eq:LimWtExps12Wttau2} completes the proof of part (ii).

\subsection{Additional Technical Proof of Proposition \ref{prop:GammaGen} (i)}\label{subsec:ProofPropGammaGenStep12}

\smallskip
\noindent
{We now complete the construction of the sequence $(h_{n})_{n\in\bN}$ as claimed in Step 1.2 of the proof of Proposition \ref{prop:GammaGen} (i).} We begin the construction with introducing some notations. For any $n\in\bN$, we set $A_{n}:=\{r\in[s_{0},T]:|g_{f}(r)|>1/n\}$, which is clearly a nondecreasing sequence of Borel sets, and we denote its limit by $A:= \lim_{n\rightarrow\infty}A_{n}=\cup_{n\in\bN}A_{n}=\{r\in[s_{0},T]:|g_{f}(r)|>0\}$. We also define
\begin{align}\label{eq:FunctsgnGn}
g_{n}(t):=g_{f}(t)\1_{A_{n}}(t),\quad G_{n}(t)= -\int_{t}^{T}g_{n}(r)\,dr,\quad t\in[s_{0},T],\quad n\in\bN.
\end{align}
Note that $g_n$ is a c\`adl\`ag function on $[0,T]$ (in this case we stipulate the left limit at $t=0$ as $g_n(0)$ and the right limit at $t=T$ as $g_n(T)$).  The existence of an approximation in the sense of  Lemma \ref{lem:cdlhApprox} (with $\varepsilon=1/2n$) implies that, for each $n\in\bN$, there exist $K_{n}\in\bN$  and $s_{0}=t_{0}^{n}<t^{n}_{1}<\cdots<t^{n}_{K_{n}}=T$ such that $g_{n}$ is either non-positive or nonnegative on each subinterval $[t_{k-1}^{n},t_{k}^{n})$.

Next, we will construct a Borel set $D_{n}\subset A_{n}$ for every $n\in\bN$, such that the sequence of functions
\begin{align}\label{eq:Functhn}
h_{n}(t):=g_{f}(t)\1_{A_{n}\setminus D_{n}}(t)=g_{n}(t)\1_{A_{n}\setminus D_{n}}(t),\quad t\in[s_{0},T],\quad n\in\bN,
\end{align}
satisfies properties (c) and (d). Note that such sequence automatically satisfies property (a). Toward this end, we fix any $n\in\bN$ in the following discussion. By \eqref{eq:FunctsgnGn}, invoking the property of $g_{n}$ described right below \eqref{eq:FunctsgnGn}, we see that $\sup_{t\in[s_{0},T]}G_{n}(t)$ must be achieved at least one of the partition points $t_{0}^{n}=s_{0},t_{1}^{n},\ldots,t_{K_{n}-1}^{n},t_{K_{n}}^{n}=T$, and we denote the smallest such point by $t_{*}^{(n)}$. By the definition of $G_{n}$ in \eqref{eq:FunctsgnGn}, we see that $g_{n}$ must be nonnegative on the subinterval with right-end point $t_{*}^{n}$. Moreover, we further divide those subintervals $[t_{k-1}^{n},t_{k}^{n})$, on which $g_{n}$ is nonnegative, into disjoint intervals, denoted by $I_{1}^{n},\ldots,I_{M_{n}}^{n}$ for some $M_{n}\in\bN$, such that $\sup I_{m-1}^{n}\leq\inf I_{m}^{n}$ {and that
\begin{equation}\label{eq:green}
\text{Leb}(I_{m}^{n})\leq 1/(n\|g_{f}\|_{\infty}).
\end{equation}
Then, there exists $m^{n}_{*}\in\bZ_{+}$ such that $\sup {I}^{n}_{m^{n}_{*}}=t^{n}_{*}$ (with convention $\sup\emptyset=s_{0}$), and we have\footnote{We let $I_{0}^{n}=\emptyset$.}
\begin{align*}
\int_{\bigcup_{m=0}^{m^{n}_{*}}I^{n}_{m}}g_{n}(r)\,dr\geq\int_{s_{0}}^{t^{n}_{*}}g_{n}(r)\,dr=G_{n}(t^{n}_{*})-G_{n}(s_{0})=\sup_{t\in[s_{0},T]}G_{n}(t)-G_{n}(s_{0}).
\end{align*}
Also, we let $\underline{m}^{n}$ be the smallest nonnegative integer such that
\begin{align}\label{eq:Underlinemn}
\int_{\bigcup_{m=0}^{\underline{m}^{n}}I_{m}^{n}}g_{n}(r)\,dr\geq G_{n}(t^{n}_{*})-G_{n}(s_{0})=\sup_{t\in[s_{0},T]}G_{n}(t)-G_{n}(s_{0}),
\end{align}
and define
\begin{align}\label{eq:SetDn}
D_{n}:=\bigcup_{m=0}^{\underline{m}^{n}}I_{m}^{n}\cap A_{n}.
\end{align}
Together with the definitions of $A_{n}$ and $g_{n}$, we have $g_{n}(t)>1/n$ for any $t\in D_{n}$. Combining \eqref{eq:FunctHn}, \eqref{eq:FunctsgnGn}, \eqref{eq:Functhn}, \eqref{eq:Underlinemn}, and \eqref{eq:SetDn}, we deduce that
\begin{align}
H_{n}(s_{0})&= -\int_{s_{0}}^{T}h_{n}(r)\,dr= -\int_{A_{n}\setminus D_{n}}g_{n}(r)\,dr= -\int_{A_{n}}g_{n}(r)\,dr+\int_{D_{n}}g_{n}(r)\,dr\nonumber\\
\label{eq:CompHns0Gntn*} &= -\int_{s_{0}}^{T}g_{n}(r)\,dr+\int_{\bigcup_{m=0}^{\underline{m}^{n}}I_{m}^{n}}g_{n}(r)\,dr\geq G_{n}(s_{0})+G_{n}(t^{n}_{*})-G_{n}(s_{0})=G_{n}(t^{n}_{*}).
\end{align}
Noting that $\underline{m}^{n}\leq m^{n}_{*}$, by letting $t^{n}_{D}:=\sup D_{n}$, we have $t^{n}_{D}\leq t_{*}^{n}$. Moreover, since $\cup_{m=0}^{\underline{m}^{n}}I_{m}^{n}$ contains all the points in $[s_{0},t^{n}_{D}]$ at which $g_{n}$ is positive, by \eqref{eq:Functhn} we see that $h_{n}$ is non-positive on $[s_{0},t_{D}^{n}]$, and thus $H_{n}$, defined by \eqref{eq:FunctHn}, is non-increasing on $[s_{0},t_{D}^{n}]$. In addition, it follows from \eqref{eq:Functhn} that $h_{n}(t)=g_{n}(t)$ for all $t\in(t^{n}_{D},T]$, which, together with \eqref{eq:FunctHn} and \eqref{eq:FunctsgnGn}, implies that $H_{n}(t)=G_{n}(t)$ for all $t\in[t^{n}_{D},T]$. Therefore, we obtain that
\begin{align}\label{eq:HnMax}
\sup_{t\in[s_{0},T]}H_{n}(t)=\max\bigg\{\sup_{t\in[s_{0},t^{n}_{D}]}H_{n}(t),\,\sup_{t\in(t^{n}_{D},T]}H_{n}(t)\bigg\}=\max\big\{H_{n}(s_{0}),\,G_{n}(t^{n}_{*})\big\}=H_{n}(s_{0}),\quad
\end{align}
where the last equality follows from \eqref{eq:CompHns0Gntn*}. Hence, we have shown that $(h_{n})_{n\in\bN}$ defined as in \eqref{eq:Functhn} satisfies property (c). As for property (d), recalling that as shown above, for each $n\in\bN$, there exist $K_{n}\in\bN$ and $s_{0}=t_{0}^{n}<t^{n}_{1}<\cdots<t^{n}_{K_{n}}=T$ such that $g_{n}$ is either non-positive or nonnegative on each subinterval $[t_{k-1}^{n},t_{k}^{n})$. By the construction of $D_{n}$ and $h_{n}$, we see that the same property holds for each $h_{n}$. That is, for each $n\in\bN$, there exists $J_{n}\in\bN$ and $s_{0}=s^{n}_{0}<s^{n}_{1}<\cdots<s^{n}_{J_{n}}=T$ such that $h_{n}$ is either non-positive or nonnegative on each $[s_{j-1}^{n},s_{j}^{n})$. By merging all the consecutive subintervals on which $h_{n}$ has the same sign, we can always assume that $h_{n}$ has alternating signs on $[s_{j-1}^{n},s_{j}^{n})$, $j=1,\ldots,J_{n}$. Moreover, in view of \eqref{eq:FunctHn} and \eqref{eq:HnMax}, $h_{n}$ must be non-positive on $[s_{0}^{n},s_{1}^{n})$. Therefore, we obtain a finite partition $s_{0}=s^{n}_{0}<s^{n}_{1}<\cdots<s^{n}_{J_{n}}=T$ such that $h_{n}$ is non-positive on $[s_{j-1}^{n},s_{j}^{n})$ when $j$ is odd and nonnegative when $j$ is even, which is indeed property (d).

It remains to show that the functions $(h_{n})_{n\in\bN}$, defined as in \eqref{eq:Functhn} (with $D_{n}$ given by \eqref{eq:SetDn}), satisfy property (b). Toward this end, by \eqref{eq:FunctsgnGn}, \eqref{eq:green}, \eqref{eq:Underlinemn}, \eqref{eq:SetDn}, and since $g_{n}$ is nonnegative on each $I_{m}^{n}$, if $\underline{m}^{n}\geq 1$, we have
\begin{align}\label{eq:EstIntDngn}
\int_{D_{n}}g_{n}(r)\,dr=\int_{\bigcup_{m=0}^{\underline{m}^{n}}I_{m}^{n}}g_{n}(r)\,dr\leq\int_{\bigcup_{{m}=0}^{\underline{m}^{n}-1}I^{n}_{m}}g_{n}(r)\,dr+\frac{1}{n}\leq\sup_{t\in[s_{0},T]}G_{n}(t)-G_{n}(s_{0})+\frac{1}{n},\quad
\end{align}
where the last inequality is due to the fact that $\underline{m}_{n}$ is the smallest nonnegative integer such that \eqref{eq:Underlinemn} holds true. When $\underline{m}^{n}=0$, clearly $D_{n}=I^{n}_{0}=\emptyset$ and \eqref{eq:EstIntDngn} holds trivially. Moreover, recalling $f\in C_{e,\text{cdl}}^{\text{ac}}(\bR_{+})$ and $f(s_{0})=\sup_{s\in\bR_{+}}f(s)$, by \eqref{eq:FunctsgnGn} we have
\begin{align*}
&\sup_{t\in[s_{0},T]}G_{n}(t)-G_{n}(s_{0})=\sup_{t\in[s_{0},T]}G_{n}(t)-\sup_{t\in[s_{0},T]}f(t)+f(s_{0})-G_{n}(s_{0})\leq 2\sup_{t\in[s_{0},T]}\big|G_{n}(t)-f(t)\big|\nonumber\\
&\leq 2\sup_{t\in[s_{0},T]}\!\int_{t}^{T}\!\big|g_{n}(r)-g_{f}(r)\big|\,dr=2\sup_{t\in[s_{0},T]}\!\int_{[t,T]\cap A}\!\big|g_{f}(r)\1_{A_{n}}(r)-g_{f}(r)\big|\,dr\leq 2\|g_{f}\|_{\infty}\,\text{Leb}(A\setminus A_{n}).
\end{align*}
Together with \eqref{eq:EstIntDngn}, we obtain that
\begin{align*}
\int_{D_{n}}g_{n}(r)\,dr\leq\frac{1}{n}+2\|g_{f}\|_{\infty}\,\text{Leb}(A\setminus A_{n})\rightarrow 0,\quad\text{as }\,n\rightarrow 0.
\end{align*}
Finally, since $g_{n}>1/n$ on $D_{n}$, we deduce that $\lim_{n\rightarrow\infty}\text{Leb}(D_{n})=0$, and together with \eqref{eq:Functhn}, we conclude that, as $n\rightarrow\infty$,
\begin{align*}
\text{Leb}\big(\big\{t\in[s_{0},T]:g_{f}(r)\neq h_{n}(t)\big\}\big)=\text{Leb}\big((A\setminus A_{n})\cup D_{n}\big)\leq\text{Leb}(A\setminus A_{n})+\text{Leb}(D_{n})\rightarrow 0,
\end{align*}
which clearly implies property (b).

\subsection{A Technical Lemma for the Proof of (\ref{eq:NoisyWHSuffCondInt})}\label{subsec:ProofAbsCont}

The identity \eqref{eq:NoisyWHSuffCondInt} follows immediately from the following lemma.

\begin{lemma}\label{lem:CadlagRightDerivImpAbsCont}
Let $f\in C(\bR_{+})$. Assume that $f$ is right-differentiable on $\bR_{+}$, and that its right-derivative, denoted by $f'_{+}$, is c\`{a}dl\`{a}g and bounded on $\bR_{+}$. Then, $f$ is globally Lipschitz continuous, and we have
\begin{align*}
f(t)-f(0)=\int_{0}^{t}f'_{+}(s)\,ds,\quad t\in\bR_{+}.
\end{align*}
\end{lemma}


\begin{proof}
We will fix any $t\in (0,\infty)$ for the rest of the proof. Since $f'_{+}$ is c\`{a}dl\`{a}g, it has at most countably many jumps in $[0,t)$, denoted by $(r_{n})_{n\in\bN}$ (note that this sequence is not ordered in general). For each $n\in\bN$, we define $s_{n}:=\inf\{r>r_{n}:f'_{+}(r)\neq f'_{+}(r-)\}\wedge t$. The right-continuity of $f'_{+}$ implies that $I_{n}:=[r_{n},s_{n})$ is nonempty and disjoint from each other, and that $f'_{+}$ is continuous on $I_{n}$ (with right-continuity at $r_{n}$), $n\in\bN$. Assume without loss of generality that $[0,t)\setminus(\cup_{n\in\bN}I_{n})\neq\emptyset$. From the construction of $(I_{n})_{n\in\bN}$, we see that $[0,t)\setminus(\cup_{n\in\bN}I_{n})$ is a countable union of intervals $J_{m}:=[a_{m},b_{m})$, $m\in\bN$, so that $[0,t)=(\cup_{n\in\bN}I_{n})\cup(\cup_{m\in\bN}J_{m})$. Moreover, $f'_{+}$ is continuous on each $J_{m}$ (with right-continuity at $a_{m}$). By \cite[Corollary 2.1.2]{Pazy:1983}, $f$ is continuously differentiable on each $I_{n}$ and $J_{m}$ (with right-continuous differentiability at each left-end point). Hence, by the fundamental theorem of calculus and the continuity of $f$,
\begin{align}\label{eq:AbsContgInJm}
f(s_{n})-f(r_{n})=\int_{r_{n}}^{s_{n}}f'_{+}(r)\,dr,\quad f(b_{m})-f(a_{m})=\int_{a_{m}}^{b_{m}}f'_{+}(r)\,dr,\quad n,m\in\bN.
\end{align}
In particular, $f$ is Lipschitz continuous on $[0,t]$ with Lipschitz constant $\|f'_{+}\|_{\infty}<\infty$. Hence, for any $N\in\bN$, by rearranging the end points of $(I_{n})_{n=1}^{N}$ and $(J_{m})_{m=1}^{N}$ into
\begin{align*}
0=:t_{0}\leq t_{1}<t_{2}\leq t_{3}<t_{4}\leq\cdots\leq t_{4N-1}<t_{4N}\leq t_{4N+1}:=t,
\end{align*}
and denoting by $A_{N}:=(\cup_{n=1}^{N}I_{n})\cup(\cup_{m=1}^{N}J_{m})=\cup_{k=1}^{2N}[t_{2k-1},t_{2k})$, we deduce from \eqref{eq:AbsContgInJm} that
\begin{align*}
\bigg|f(t)\!-\!f(0)\!-\!\int_{A_{N}}\!f'_{+}(r)\,dr\bigg|=\left|\sum_{k=0}^{2N}\big(f(t_{2k+1})-f(t_{2k})\big)\right|\leq\|f'_{+}\|_{\infty}\cdot\text{Leb}\big([0,t)\setminus A_{N}\big)\rightarrow 0,\quad N\rightarrow\infty.
\end{align*}
Finally, we obtain from the dominate convergence that
\begin{align*}
f(t)-f(0)=\lim_{N\rightarrow\infty}\int_{A_{N}}f'_{+}(r)\,dr=\int_{0}^{t}f'_{+}(r)\,dr,
\end{align*}
which completes the proof of the lemma.
\end{proof}

\end{document}